\documentclass[Ass:hmaxpaper,11pt]{article} 
\usepackage{amsmath,amssymb,amsthm,amsfonts,amstext,amsbsy,amscd,mathtools}
\usepackage[utf8]{inputenc}
\usepackage{hyperref}
\usepackage[T1]{fontenc}
\usepackage[lofdepth,lotdepth]{subfig}
\usepackage{url}
\usepackage{graphics}
\usepackage{graphicx}
\usepackage{float}
\usepackage[top=3cm, bottom=3cm, left=2cm, right=2.5cm]{geometry}
\usepackage{dsfont}
\usepackage{color}
\usepackage{multirow}
\usepackage{multicol}
\usepackage{dsfont}
\usepackage{algpseudocode}
\usepackage{enumitem}
\definecolor{forestgreen}{rgb}{0.13, 0.55, 0.13}

\newtheorem{theorem}{Theorem}[section]
\newtheorem{lemma}[theorem]{Lemma}
\newtheorem{proposition}[theorem]{Proposition}

\newtheorem{remark}[theorem]{Remark}
\newtheorem{example}[theorem]{Example}

\newtheorem{assumption}[theorem]{Assumption}

\newcommand{\rmd}{{\rm d}}

\renewcommand{\le}{\leqslant}
\renewcommand{\ge}{\geqslant}

\renewcommand{\tilde}{\widetilde}

\def\@fnsymbol#1{\ensuremath{\ifcase#1\or *\or \dagger \or \ddagger\or
   \mathsection\or \mathparagraph\or \|\or **\or \dagger\dagger
   \or \ddagger\ddagger \else\@ctrerr\fi}}

\begin{document}

\title{Interacting Hawkes processes with multiplicative inhibition}


\author{ Céline Duval\footnote{Universit\'e de  Paris, Laboratoire  MAP5, UMR CNRS
8145. E-mail: \href{mailto: celine.duval@u-paris.fr}{celine.duval@u-paris.fr}},\  \ Eric Lu\c con\footnote{Universit\'e de  Paris, Laboratoire  MAP5, UMR CNRS
8145.  E-mail: \href{mailto: eric.lucon@u-paris.fr}{eric.lucon@u-paris.fr}} \ and  Christophe Pouzat \footnote{IRMA, Université de Strasbourg, CNRS UMR 7501.  E-mail: \href{mailto: christophe.pouzat@math.unistra.fr}{christophe.pouzat@math.unistra.fr}}
}
\date{}

\maketitle

\begin{abstract} 
In the present work, we introduce a general class of mean-field interacting nonlinear Hawkes processes modelling the reciprocal interactions between two neuronal populations, one excitatory and one inhibitory. The model incorporates two features: inhibition, which acts as a multiplicative factor onto the intensity of the excitatory population and additive retroaction from the excitatory neurons onto the inhibitory ones. We give first a detailed analysis of the well-posedness of this interacting system as well as its dynamics in large population. The second aim of the paper is to give a rigorous analysis of the longtime behavior of the mean-field limit process. We provide also numerical evidence that inhibition and retroaction may be responsible for the emergence of limit cycles in such system.
\end{abstract}

 \section{Introduction\label{sec:intro}}
\subsection{Biological description\label{sec:bio_desc}}

 Neurons are peculiar network forming cells whose (biological) function is to receive signals from other neurons, 'integrate' these signals and transmit the 'integration result' to other neurons or \emph{effector} cells like muscles \cite{luo:15}. Since most neurons are large cells by biological standards (e.g. the motoneurons controlling the muscles of our big toes are about a meter long) they use a specific mechanism to reliably transmit signals from one of their extremities to another \cite{castelfranco.hartline:16}: a propagating electrical wave of \emph{short} duration (\(\sim\) 1-2 ms) called the \emph{action potential} or \emph{spike}. Action potentials are \emph{all-or-none} events (a necessary condition for getting reliable transmission across long distances) that are triggered when the neuron's \emph{membrane potential} (i.e. the electrical potential difference between the inside and the outside of the neuron) is 'large enough'. To simplify, between two successive action potentials the neuron 'sums' its inputs, changing thereby its membrane potential and, when this membrane potential crosses a more or less sharp threshold, the neuron fires a spike that propagates \cite{luo:15}. Neurons come in two main categories: i) \emph{excitatory} neurons make their \emph{postsynaptic} partners more likely to spike; ii) \emph{inhibitory} neurons make their \emph{postsynaptic} partners less likely to spike. Although, we are dealing  with biological material, implying that there are many sub-types within both of these categories \cite{braitenberg.schuz:98,luo:15},  we are interested in studying a  simplified model of the \emph{neocortex}--the outmost and most recently evolved part of the vertebrate brain--and we are going to consider just two neuronal types, excitatory and inhibitory with the actual neocortical proportions of 80\% and 20\% \cite{braitenberg.schuz:98}.

 \subsection{Our model}

 In this paper, a general class of mean-field interacting Hawkes processes is considered, modeling the reciprocal interactions between a population $A$ of excitatory neurons and a population $B$ of inhibitory neurons. Hawkes processes, originally introduced in \cite{MR0378093}, have met a recent interest in the mathematical neuroscience literature, for their ability to model the dependence of the activity of a neuron in the history of the network. We adopt here a usual simplifying hypothesis on the network, that is to suppose all-to-all connections between neurons, within a large population (mean-field framework). We refer to Section~\ref{sec:Hawkes_inhibition} below for a detailed account of the use of Hawkes processes in the modeling of neuronal networks. 
 
More precisely, consider a population of neurons of size $N\geq 1$, that is divided into population $A$ (that is considered to be excitatory) with size $N_{ A}:= \alpha N$ with $ \alpha\in (0, 1)$ (e.g. $ \alpha=0.8$ \cite{braitenberg.schuz:98}) and a population $B$ with size $N_{ B}= (1- \alpha) N$ which is considered to be inhibitory. A particular instance of the model we describe in the paper (and its main motivation) is then given in terms of a family of counting processes $(Z^{1}_{ t}, \ldots, Z^{ N_{ A}}_{ t})$ (population $A$) and $(Z^{ N_{ A}+1}_{ t}, \ldots, Z^{N}_{  t})$ (population $B$) with coupled conditional stochastic intensities given respectively by $\lambda^{A}$ and $\lambda^{B}$ as follows
\begin{equation}
\label{eq:model_intro}
\begin{cases}
 \lambda_{  t}^{A}:&= \Phi_{ A}\left(\frac{ 1}{ N} \underset{1\le j\le N_{A}}{\sum}\int_{ 0}^{t^{ -}} h_{ 1}(t-u) {\rmd Z_{u}^{j}}) \right) \Phi_{ B\to A} \left( \frac{ 1}{ N}  \underset{N_{A}+1\le j\le N}{\sum} \int_{ 0}^{t^{ -}} h_{ 2}(t-u) {\rmd Z_{u}^{j}}\right), \\
 \lambda_{  t}^{B}:&= \Phi_{ B} \left(\frac{ 1}{ N}  \underset{N_{A}+1\le j\le N}{\sum} \int_{ 0}^{t^{ -}} h_{ 3}(t-u) {\rmd Z_{u}^{j}})\right) + \Phi_{ A\to B} \left(\frac{ 1}{ N}  \underset{1\le j\le N_{A}}{\sum} \int_{ 0}^{t^{ -}} h_{ 4}(t-u) {\rmd Z_{u}^{j}}\right).
\end{cases}
\end{equation}
We refer to Section~\ref{sec:modgen} below for precise definitions and main hypotheses concerning Model \eqref{eq:model_intro}. The dynamics given by \eqref{eq:model_intro} is of Hawkes type: the intensity of each particle depends on the history of the whole system, through memory kernels $h_{ i}$, $i=1, \ldots, 4$ and firing rate functions $ \Phi_{ A}$ and $ \Phi_{ B}$ (a classical example, so-called linear, being when $ \Phi_{ A}(x)= \mu_{ A}+ x$ and $ \Phi_{ B}(x)= \mu_{ B}+ x$, where $ \mu_{ A}, \mu_{ B}\geq 0$ are constant intensities modeling spontaneous activity in populations $A$ and $B$ in absence of interaction). The main novelty of the Model \eqref{eq:model_intro} (in comparison with previous Hawkes models with inhibition, see Section~\ref{sec:Hawkes_inhibition} below) lies in the multiplicative nature of the influence of inhibitory population $B$ onto population $A$, through the inhibition kernel $ \Phi_{ B\to A}$ (think for now as $ \Phi_{ B\to A}$ being a decreasing nonnegative function on $[0, +\infty)$, with $ \Phi_{ B\to A} (0)=1$ and $ \Phi_{ B\to A}(x) \xrightarrow[ x\to \infty]{}0$): activity of population $A$ should decrease as activity of population $B$ rises. In this respect, a natural control case is when $ \Phi_{ B\to A} \equiv 1$ where $ \lambda_{ t}^{ A}$ in \eqref{eq:model_intro} reduces to a simple, possibly nonlinear, Hawkes process \cite{MR3449317} decoupled from population $B$. The model secondly incorporates retroaction from population $A$ onto population $B$, that is supposed to be mostly additive, although possibly modulated by a nonlinear feedback kernel $ \Phi_{ A\to B}$. 

\subsection{Stochastic models for inhibition}
\label{sec:neuro_inhibition}
Two different asymptotic behaviors for \eqref{eq:model_intro} will be successively considered in the paper: its behavior in large population ($N\to\infty$) and in large time ($t\to\infty$). The large population analysis goes back to well-proven techniques for mean-field systems and, as we will see below, the natural limit of \eqref{eq:model_intro} will be described in terms of an inhomogeneous Poisson process with nonlinear intensity. Similar mean-field analysis have been successfully carried out for a large class of models (that do not restrict to interacting point processes), incorporating various relevant features in neuroscience, such as refractory periods \cite{CHEVALLIER20173870,raad2020stability2},  spatial inhomogeneities \cite{CHEVALLIER20191,22657695}, excitability \cite{Luon2019}. At this point, due to the nonlinear (and often non Markovian) character of the mean-field limit, one major difficulty concerns the characterization of its dynamics as $t\to\infty$, this difficulty being all the more present when one wants to incorporate inhibition. Whereas there is a vast literature on mean-field limits as $N\to\infty$, a significantly fewer number of models are amenable to a rigorous longtime analysis. In this respect, we will show that the present Model \eqref{eq:model_intro} is sufficiently tractable for a detailed rigorous analysis as $t\to\infty$. This  allows us in particular to highlight some biologically relevant effects of inhibition on the system, see Section~\ref{sec:effect_inhibition}.

The point of the current section is to motivate these theoretical difficulties in existing models in neuroscience and to compare them with \eqref{eq:model_intro}.

\subsubsection{Voltage-based models}
\label{sec:voltage_models}
A signal gets transmitted from one neuron to the next by the activation of a synapse: the action potential of the \emph{presynaptic} neuron reaches the presynaptic terminal, this triggers the release of packets of neurotransmitters (small molecules like glutamate for the excitatory synapses and GABA for the inhibitory ones) that diffuse in the small space between the pre- and post-synaptic neurons and bind to \emph{receptor-channels} located in the membrane of the postsynaptic neuron  \cite{luo:15}. Channels are macromolecules spanning the cell membrane and forming a pore through the latter; the pore can be closed or opened. After transmitter binding to these receptor-channels, the latter open and let specific ions flow through their pore (mainly sodium for excitatory synapses and chloride for the inhibitory ones). These ion fluxes or \emph{currents} will induce a change of the postsynaptic neuron membrane potential.

A simple quantitative neuronal model compatible with this biological description was introduced in 1907 by Lapicque \cite{lapicque:07}. It is now known as the \emph{Lapicque} or the \emph{integrate and fire} model and it describes the membrane potential dynamics of 'point-like' neurons (their spatial extension is not explicitly modeled) as follows \cite{tuckwell:88a,burkitt:06}:
  \begin{equation}\label{eq:IF1}
C_m \frac{\rmd V_i}{\rmd t} = -\frac{V_i}{R_m} + \sum_{j \in \mathcal{S}_{i,E}} I_{j\rightarrow i}(t) + \sum_{k \in \mathcal{S}_{i,I}} I_{k\rightarrow i}(t)\quad \text{if} \quad V_i(t) < V_{thr}\, , 
  \end{equation}
where \(i\) is the neuron index; \(C_m\) is the neuron capacitance; \(R_m\) is the neuron membrane resistance; \(\mathcal{S}_{i,E}\), respectively \(\mathcal{S}_{i,I}\), are the indices of excitatory, respectively inhibitory, neurons presynaptic to \(i\); \(I_{j\rightarrow i}(t) \ge 0\), respectively \(I_{k\rightarrow i}(t) \le 0\), are the synaptic currents due to neuron \(j\), respectively \(k\), at time \(t\); \(V_{thr}\) is the 'threshold' voltage. Every time \(V_i(t) = V_{thr}\) an action potential is emitted and \(V_i\) is reset to 0. Very often the \(I_{j\rightarrow i}(t)\) are set to:
  \begin{equation}\label{eq:Intro1}
I_{j\rightarrow i}(t) = w_{j\rightarrow i} \sum_{l} \delta(t-t_{j,l}) \, , 
  \end{equation}
where \(w_{j\rightarrow i}\) is referred to as the \emph{synaptic weight}; \(\delta\) stands for the Dirac delta function; the \(t_{j,l}\) are the successive spike times of neuron \(j\).  In this model, when there are no inputs, the membrane potential relaxes towards 0 with a time constant \(\tau = R_m C_m\). A presynaptic spike from an excitatory neuron \(j\) generates an instantaneous upward 'kick' of amplitude \(w_{j\rightarrow i}\), while a presynaptic spike from an inhibitory neuron \(k\) generates an instantaneous downward 'kick' of amplitude \(w_{k\rightarrow i}\). 

The actual action potential is not explicitly modeled--it was not doable at the time of Lapicque since the biophysics of this phenomenon was not understood--but is replaced by a point event. Notice that with the synaptic input description illustrated by   \eqref{eq:Intro1} the current generated in the postsynaptic neuron by a given synapse does not depend on the membrane voltage of the former. This constitutes a crude approximation of the actual biophysics of synaptic current generation. A much better  approximation \cite{koch.poggio.torre:83,tuckwell:88b}--but harder to work with analytically--is provided by using:
\begin{equation}\label{eq:Intro2}
I_{j\rightarrow i}(t) = g_{j\rightarrow i} (V_{rev}-V_i) \, \sum_{l} \delta(t-t_{j,l}) \, , 
  \end{equation}
where \(g_{j\rightarrow i}\) is the \emph{synaptic conductance}, \(V_{rev}\) is the synaptic current \emph{reversal potential}--as its name says this is the voltage at which the current changes sign--, it is negative or null for inhibitory synapses and larger than the threshold voltage for excitatory ones. Taking \(V_{rev}=0\) for the inhibitory inputs we see an important and empirically correct feature appearing: the downward 'kick' generated by the activation of an inhibitory input becomes \emph{proportional} to the membrane voltage \(V_i\); the larger the latter, the larger the kick.

The biophysical and molecular events leading to the action potential emission and propagation are by now well understood \cite{luo:15} and are known to involve the \emph{stochastic} opening and closing of \emph{voltage-gated ion channels}: the probability of finding a channel closed or opened depends on the membrane voltage. This stochastic dynamics of the channels can lead to fluctuations in the action potential emission or propagation time when a given (deterministic) stimulation is applied repetitively to a given neuron \cite{verveen.derkesen:68,yarom.hounsgaard:11}. Similarly, the synaptic receptor-channels do fluctuate between open and close states, but that's the binding of the transmitter rather than the membrane potential that influences the probability of finding the channel in the open state. An even (much) larger source of fluctuations at the synapse results from the variable number of transmitter packets that get released upon a presynaptic action potential arrival \cite{luo:15}, even if the same presynaptic neuron is repetitively activated in the same conditions. The result of all these fluctuation sources is a rather 'noisy' aspect of the membrane potential of cortical neurons that legitimates the use of 'stochastic units' as building blocks of neural network models \cite{tuckwell:88b,yarom.hounsgaard:11}. Historically, the first stochastic units were built by adding a Brownian motion process term to the right hand side of   \eqref{eq:IF1} \cite{gerstein.mandelbrot:64,burkitt:06,sacerdote.giraudo:12}. 

Among the vast literature concerning rigorous mean-field approaches in the context of integrate and fire systems and their derivations, we refer in particular to e.g. \cite{MR3349003,brunel2000dynamics} and references therein. This approach leads to Chapman-Kolmogorov / Fokker-Planck equations that are hard to work with analytically and numerically \cite{MR3190511}, the main difficulty being the existence of blowups in finite time for the system, due to the possibility of having a macroscopic number of neurons simultaneously close to the firing threshold $V_{ thr}$. Several extensions with random firing rates have been proposed in the literature, \cite{MR3311484,Cormier2020} and references therein, but the longtime analysis of such models (especially when it concerns the effect of inhibition) remains not fully understood (see nonetheless the recent significant advances in this direction in \cite{Cormier2020}).

\subsubsection{Hawkes processes}
\label{sec:Hawkes_inhibition}
The model we consider in \eqref{eq:model_intro} enters into the framework of Hawkes processes \cite{MR0378093,MR3449317}. The point of view adopted here is to model directly the stochastic intensity  \cite{brillinger:88,chornoboy.schramm.karr:88} by identifying a neuron's sequence of strereotyped action potentials with the point process \((t_{j,l})_{j,l}\) in \eqref{eq:Intro2}. In this context, the Hawkes framework generically reads
\begin{equation}\label{eq:Hawkes1}
\lambda_t^i = \Phi_{i}\left( \sum_{j \in \mathcal{S}_{i,E}} \int_{0}^t h_{j\rightarrow i}(t-u) \rmd Z_u^j \right)\, ,
\end{equation}
where \(\lambda_t^i\) is the intensity of neuron \(i\), $\Phi_{i}$ a positive function, \(Z_{j,t}\) is the counting process associated with neuron \(j\), \(h_{j\rightarrow i}(t)\) is the \emph{synaptic kernel} associated with the synapse between neurons \(j\) and \(i\). To qualitatively match the case of the integrate and fire model  \eqref{eq:IF1}, we would take \(\Phi_{i}(x)=\mu_0^i+x\) where $\mu_{i}^{0}$ is the basal rate of neuron $j$ and \(h_{j\rightarrow i}(t) = (w_{j\rightarrow i}/\tau) \, \exp -t/\tau\).
It is not really possible to account for the vast recent literature on mean-field analysis of Hawkes processes in neuroscience, we refer to e.g. \cite{MR1411506,MR3449317,DITLEVSEN20171840,CHEVALLIER20191,CHEVALLIER20173870} and references therein for further details. Note that this framework allows for a large versatility in the modeling, since it can accommodate for various features such as age-dependent behaviors/refractory periods \cite{CHEVALLIER20173870,Chevallier2017,raad2020stability2} or spatially-structured dynamics \cite{CHEVALLIER20191}.  

Intensity-based models such as \eqref{eq:Hawkes1} lead to much better looking equations, that are also easier to work with numerically \cite{hawkes:71,chornoboy.schramm.karr:88,MR3449317}: in the mean-field case, one naturally obtains large population limits in terms of inhomogeneous Poisson processes with intensity solving convolution equations \cite{MR3449317}. A particularly important instance of \eqref{eq:Hawkes1} that is explicitly solvable concerns the linear case where $ \Phi_{ i}(x)= \mu_{ i}+ x$ with nonnegative synaptic kernels $h_{ i\to j}$: the corresponding convolution equation for the mean-field intensity becomes linear and one can explicitly compute its limit as $t\to\infty$, \cite{MR3449317}, Th.~10 and~11.

As far as inhibition is concerned, following the approach of Section~\ref{sec:voltage_models}, one would rightfully criticize the use of a Hawkes process  by pointing out the difficulty, not to say the impossibility, of modifying it to accommodate inhibition without loosing the nice analytical properties  of linear Hawkes processes \cite{chornoboy.schramm.karr:88,MR1411506}.  In this context, the main strategy followed in the literature so far has been to model inhibition by introducting kernels $h_{ j\to i}(\cdot)$ in \eqref{eq:Hawkes1} taking negative values \cite{costa2020renewal,raad2020stability,raad2020stability2,reynaud2013inference,bonnet2021maximum} or positive kernels $h$ (e.g. of Erlang type) multiplied by random and possibly negative coefficients \cite{DITLEVSEN20171840,Duarte:2016aa,pfaffelhuber2021mean}. This is what we could qualify as an \emph{additive} inhibition, in the sense that the intensity of a neuron is the (temporal) additive superposition of several memory kernels $h_{ i\to j}$, in the same spirit as the scheme suggested by \eqref{eq:Intro1} for integrate and fire models. In order to keep $ \lambda_{ t}^{ i}$ positive, one crucially needs here to modulate this superposition by a positive (necessarily nonlinear) synaptic kernel $ \Phi_{ i}$, (e.g. $\Phi_{i}(x)=\mu^{i}+x_{+}$ where $x_{ +}=\max(x, 0)$). The major counterpart is that one looses the nice analytical properties which makes the longtime analysis of linear Hawkes models tractable.

The strategy of the present Model \eqref{eq:model_intro} is different, as it relies on a \emph{multiplicative inhibition}, in the spirit of the scheme provided by  \eqref{eq:Intro2} for integrate and fire models: we  view the inhibitory inputs as 'modulations' of the excitatory ones and  describe inhibition as a multiplicative factor between 0 and 1 applied to the excitatory inputs. In particular, the multiplicative structure of \eqref{eq:model_intro} no longer requires the synaptic kernels to take negative values: we may and we will only consider nonnegative kernels $h_{ i}$, $i=1,\ldots, 4$. In this way, the non-negativity of the intensity is then automatically preserved, together with the 'nice' properties of the canonical linear Hawkes process: we may still take advantage of the positivity of the synaptic kernels  and the respective monotonicities of $ \Phi_{ A}$ and $ \Phi_{ B\to A}$ (positive and decreasing) to obtain asymptotic results similar to the linear case.

 \subsection{The effect of inhibition: quenching of supercriticality and oscillations}
 \label{sec:effect_inhibition}
We briefly review here the biological features concerning the effects of inhibition that one can retrieve from Model~\eqref{eq:model_intro}. We refer to the sequel for more details. The first main observation that one can make from Model~\eqref{eq:model_intro} is that inhibition has an ability to quench supercriticality: Model~\eqref{eq:model_intro} can accommodate with situations where population $A$ alone (that is isolated from population $B$, i.e. $ \Phi_{ B\to A}\equiv 1$) is supercritical, in the sense that the intensity of $A$ diverges as $t\to\infty$, whereas the addition of inhibition and retroaction forces the system to go back to subcriticality. We refer in particular to Theorem~\ref{th:A_subcritical} below, where it can be shown rigorously that \eqref{eq:model_intro} cannot be supercritical in presence of both inhibition and retroaction. This corresponds to a biological situation commonly observed, where epileptic seizures may be observed when inhibition is altered (see \cite{luo:15} and Remark~\ref{rem:effect_inhib} below).

Another particularly interesting feature of Model \eqref{eq:model_intro} is that inhibition may lead to oscillations, a ubiquitous phenomenon in the nervous system \cite{buzsaki:06}. Emergence of oscillations in mean-field processes is a longstanding issue in the literature, that is not specific to neuroscience applications, as it reflects a common feature of self-organization in statistical physics (see e.g. \cite{Collet:2015aa,Collet2016,Meylahn:2018aa,MR3336876}), biological models of synchrony \cite{Carmona2019,aleandri2019} or neuroscience \cite{Luon2019,2018arXiv181100305L,DITLEVSEN20171840}. We mention here specifically \cite{DITLEVSEN20171840} where global oscillations have been proven for interacting Hawkes processes with circular connectivity.
 
\subsection{About simulations}
The Python code, together with a detailed documentation, that has been used for the simulations made in the present paper is available at \url{https://plmlab.math.cnrs.fr/xtof/hawkes-x-hawkes}. 

\subsection{Organization of the paper}

The paper is organized as follows: the precise definitions and main hypotheses are given in Section~\ref{sec:wellposedness_large_pop}, where a generalized version \eqref{eq:model} of \eqref{eq:model_intro} is introduced. We first address well-posedness results for the microscopic system (Section~\ref{sec:well_posedness_part_syst}) and for its mean-field limit (Section~\ref{sec:mean_field_limit}) as well as propagation of chaos estimates. The second purpose of the paper it to give a detailed analysis, in the particular case of Model \eqref{eq:model_intro}, of the behavior as $t\to \infty$ of the limiting mean-field process in Section~\ref{sec:MFL}. Focus is put on the presence/absence of inhibition and retroaction: the easier cases of decoupled dynamics are first treated in Section~\ref{sec:nonfully_coupled}. The model with full connectivity (that is with inhibition and feedback) is considered in Section~\ref{sec:fully_coupled}. Fluctuations results and some resulting statistical tests for inhibition is given in Section~\ref{sec:fluctuations}. Examples of dynamics that fulfil the hypotheses of Section~\ref{sec:MFL} are then given in Section~\ref{sec:examples}.  We give finally in Section~\ref{sec:oscillations} intuition and numerical evidence that inhibition and retroaction may lead to oscillations, for some particular instances of the inhibition kernel $ \Phi_{ B\to A}$. Proofs of the main results are gathered in Section~\ref{sec:proofs_part_MF} and Section~\ref{sec:proofs_longtime}.

\section{Well-posedness and large population behavior}
\label{sec:wellposedness_large_pop}
\subsection{Model and assumptions\label{sec:modgen}}

Let $\alpha\in(0,1)$ and $N=N_{ A}+N_{ B}\gg 1$ with $ N_{ A}= \alpha N$ and $N_{ B}= (1- \alpha)N$ the respective size of population $A$ and $B$.
The neurons within population $A$ (resp. population $B$) are indexed by $i=1,\ldots, N_{ A}$ (resp. $i=N_{ A}+1,\ldots, N$). Consider on a filtered probability space $( \Omega, \mathcal{ F}, \left(\mathcal{ F}_{ t}\right)_{ t\geq 0}, \mathbf{ P})$ an independent family of i.i.d. Poisson measures $(\pi_{i}(\rmd s,\rmd z), i\in\{1,\ldots,N\})$ with intensity measure $\rmd s\times \rmd z$ on $[0,\infty)\times[0,\infty)$. Let $(x, y) \mapsto F(x, y)$ and $(x, y) \mapsto G(x, y)$ two nonnegative functions defined on $(0,\infty)^{2}$. Consider  the family of c\`adl\`ag $(\mathcal{ F}_{ t})_{ t\geq0}$ point processes $(Z_{t}^{ i})_{ t\geq0, i=1, …, N}$ given by
\begin{equation}
\label{eq:Hawkes}
Z_{t}^{i}= \int_{0}^{t}\int_{0}^{\infty}\mathbf{1}_{z\le \lambda_{ s}^{ i}}\pi_{i}(\rmd s,\rmd z), i=1, …, N,
\end{equation}
where the intensity $ \lambda^{ i}$, $i=1,…, N$, is given as
\begin{equation}
\label{eq:model}
\begin{cases}
\lambda_{ t}^{ i}= \lambda_{  t}^{A}:= F  \left(\frac{ 1}{ N} \underset{1\le j\le N_{A}}{\sum}\int_{ 0}^{t^{ -}} h_{ 1}(t-u) {\rmd Z_{u}^{j}}), \frac{ 1}{ N}  \underset{N_{A}+1\le j\le N}{\sum} \int_{ 0}^{t^{ -}} h_{ 2}(t-u) {\rmd Z_{u}^{j}})\right), &i=1, \ldots, N_{ A},\\
\lambda_{ t}^{ i}= \lambda_{  t}^{B}:= G\left(\frac{ 1}{ N}  \underset{N_{A}+1\le j\le N}{\sum} \int_{ 0}^{t^{ -}} h_{ 3}(t-u) {\rmd Z_{u}^{j}}), \frac{ 1}{ N}  \underset{1\le j\le N_{A}}{\sum} \int_{ 0}^{t^{ -}} h_{ 4}(t-u) {\rmd Z_{u}^{j}}\right),&i=N_{ A}+1,\ldots, N.
\end{cases}
\end{equation} 
In the sequel, we sometimes write $Z_{t}^{i,A}$ (resp. $Z_{t}^{i,B}$) in place of $Z_{t}^{i}$ for $i=1, \ldots, N_{ A}$ (resp. $i=N_{A}+1,\ldots, N$) in order to specify the population the $i$th neuron belongs to. Let $p\geq1$ and $T>0$. For any function $h$ on $[0, \infty)$ such that $ \left\vert h \right\vert^{ p}$ is locally integrable (resp. integrable) denote by $ \left\Vert h \right\Vert_{p, T}$ (resp. $ \left\Vert h \right\Vert_{p}$) as the $L^{ p}$ norm of $h$ on $[0, T]$ (resp. on $[0, \infty)$):
\begin{align*}
\left\Vert h \right\Vert_{ p, T} &:= \left( \int_{ 0}^{T} \left\vert h(s) \right\vert^{ p} {\rm d}s\right)^{ \frac{ 1}{ p}},\hspace{2cm}\mbox{resp. }\ \ 
\left\Vert h \right\Vert_{ p} := \left( \int_{ 0}^{\infty} \left\vert h(s) \right\vert^{ p} {\rm d}s\right)^{ \frac{ 1}{ p}}.
\end{align*}
For any locally bounded (resp. bounded) function $h$, denote in a same way $
\left\Vert h \right\Vert_{ \infty, T} :=\sup_{ s\in [0, T]} \left\vert h(s) \right\vert$ and $
\left\Vert h \right\Vert_{ \infty} :=\sup_{ s\geq0} \left\vert h(s) \right\vert.$
For any Lipschitz-continuous function $f$ on $ \mathbb{ R}$, we denote by $ \left\Vert f \right\Vert_{ L}$ the Lispchitz constant of $f$. Throughout the paper, Assumptions~\ref{ass:main_h} and~\ref{ass:main_FG} below will be required.
\begin{assumption}[Assumptions on $h_{ i}$, $i=1, \ldots, 4$]\label{ass:main_h}
The kernels $h_{ i}$, $i=1, \ldots, 4$ on $[0, +\infty)$ are such that
\begin{equation}
\label{def:H}
H(t):={\alpha} h_{ 1}(t)+{(1-\alpha)} h_{ 2}(t)+{(1-\alpha)} h_{ 3}(t)+{\alpha} h_{ 4}(t), \quad t\geq0,
\end{equation}
is locally integrable, i.e. $\| H\|_{1,T}<\infty$, $T>0$. 
\end{assumption}
At this point, note that we need not assume that $h_{ i}\geq0$, since we are only concerned for now with well-posedness results concerning both the particle system \eqref{eq:model} and its mean-field limit as well as propagation of chaos results. The positivity of the kernels will only be assumed for the study of the longtime behavior of the mean-field process in Section~\ref{sec:MFL}. However, in view of the discussion of Section~\ref{sec:Hawkes_inhibition}, the fact that the $h_{ i}$ might take negative values is not something that we have in mind in this paper.
\begin{assumption}[Assumptions on $F$ and $G$]\label{ass:main_FG}
Suppose that $F$ and $G$ are nonnegative functions such that the following holds: there exist constants $c_{ i}\geq0, i=1, \ldots, 5$ such that
\begin{align}
\sup_{ x,y\geq0} \frac{ F(x,y)}{ c_{ 1}(x+y) + c_{ 2}} &\leq 1, \label{hyp:F_bound}\\
\sup_{ x\neq x^{ \prime}\geq0,y\geq0} \frac{ \left\vert F(x, y) - F(x^{ \prime}, y) \right\vert}{ c_{ 3}\left\vert x-x^{ \prime} \right\vert}&\leq 1,\label{hyp:F_Lip_x}\\
\sup_{ x\geq0, y\neq y^{ \prime}\geq0} \frac{ \left\vert F(x,y) - F(x, y^{ \prime})\right\vert}{ (c_{ 4}x+c_{ 5}) \left\vert y-y^{ \prime} \right\vert} &\leq 1 \label{hyp:F_Lip_y}
\end{align}

Suppose also that $G$ is globally Lipschitz in both variables: for some $c_{ 6}>0$
\begin{equation}
\label{hyp:G_Lip}
\left\vert G(x, y)- G(x^{ \prime}, y^{ \prime} )\right\vert \leq c_{ 6} \left( \left\vert x-x^{ \prime} \right\vert + \left\vert y-y^{ \prime} \right\vert\right),\ x,x^{ \prime},y,y^{ \prime}\geq0.
\end{equation}
\end{assumption}
Note that an immediate consequence of \eqref{hyp:G_Lip} is that $G$ is sublinear: for $c_{ 7}:= G(0, 0)$,
\begin{equation}
\label{hyp:G_bound}
G(x, y) \leq c_{ 7} + c_{ 6} \left(x + y\right),\ x,y\geq0.
\end{equation}
\begin{example}
\label{ex:F_Lip}
Any globally Lipschitz function $F$ (with Lipschitz constant $ \left\Vert F \right\Vert_{ L}$) satisfies Assumption~\ref{ass:main_FG} for $c_{ 1}= c_{ 3}=c_{ 5}=\left\Vert F \right\Vert_{ L}$, $c_{ 2}=F(0,0)$ and $c_{ 4}=0$.
\end{example}
In place of Assumption~\ref{ass:main_FG}, we will sometimes specify to a more restrictive model, which corresponds to \eqref{eq:model_intro}:
\begin{assumption}[Multiplicative model]
\label{ass:model_intro}
Suppose  that $F$ and $G$ satisfy
\begin{equation}
\label{eq:ex_F_G}
F(x, y)= \Phi_{ A}(x) \Phi_{ B\to A}(y),\ G(x, y)= \Phi_{ B}(x) + \Phi_{ A\to B}(y),
\end{equation}
where $\Phi_{ A}$, $\Phi_{ B\to A}$, $\Phi_{ B}$ and $\Phi_{ A\to B}$ are nonnegative functions, each of them globally Lipschitz with $\Phi_{ B\to A}$ bounded (and with no loss of generality we assume $ 0 \leq \Phi_{ B\to A} \leq 1$).
\end{assumption}
Assumption~\ref{ass:model_intro} is a particular case of Assumption~\ref{ass:main_FG}, for $c_{ 1}=c_{ 3}:= \left\Vert \Phi_{ A} \right\Vert_{ L}$, $c_{ 2}:= \Phi_{ A}(0)$, $c_{ 4}:= \left\Vert \Phi_{ B\to A} \right\Vert_{ L} \left\Vert \Phi_{ A} \right\Vert_{ L}$ and $ c_{ 5}:=\left\Vert \Phi_{ B\to A} \right\Vert_{ L} \Phi_{ A}(0)$. Although it is not required for the moment, the examples to have in mind to model inhibition is $ \Phi_{ B\to A}$ being nonincreasing with $\Phi_{B\to A}(y)\to0$ as $y\to\infty$ and $ \Phi_{ A\to B}$ nondecreasing with $ \Phi_{ A\to B}(y)\to+\infty$ as $y\to \infty$. A further particular example is
\begin{example}
\label{ex:linear}
A case of particular interest is for $F$ and $G$ satisfying Assumption~\ref{ass:model_intro} with $ \Phi_{ A}$ and $ \Phi_{ B}$ being linear:
\begin{equation*}
\Phi_{ A}(x)= \mu_{ A} +x,\ \Phi_{ B}(x)= \mu_{ B} +x,\ x\geq 0,
\end{equation*}
where $ \mu_{ A},\mu_{ B}\geq0$. In this example, suppose in addition that $h_{ i}\geq0$ for $i=1, \ldots, 4$.
\end{example}
Assumption~\ref{ass:model_intro} and Example~\ref{ex:linear}  (for which $c_{1}=c_{3}=c_{4}=1$ and $c_{2}=c_{2}=\mu_{A}$) are considered in detail in Section \ref{sec:MFL}. In the remaining of this Section, we only require Assumption~\ref{ass:main_FG}.

\subsection{Well-posedness of the particle system}
\label{sec:well_posedness_part_syst}
The first result concerns well-posedness of the particle system \eqref{eq:Hawkes}:
\begin{proposition}\label{prop:ExistUniq}
Suppose Assumptions~\ref{ass:main_h} and~\ref{ass:main_FG} hold. Then, conditional on the Poisson measures $ ( \pi_{ i})_{ i=1, \ldots, N}$, there exists a pathwise unique Hawkes process $(Z^{i})_{ 1\le i\le N}$ as in \eqref{eq:Hawkes} such that $\sum_{i=1}^{N}\mathbf{E}[Z_{t}^{i}]<\infty$, for all $t\ge0$.
\end{proposition}
In the case where $F$ is globally Lipschitz (Example~\ref{ex:F_Lip}), existence and uniqueness of \eqref{eq:Hawkes} can be proven directly through a fixed-point procedure as in \cite{MR3449317}. Then, Proposition~\ref{prop:ExistUniq} is a direct consequence of \cite{MR3449317}, Theorem~6, up to minor notational changes. The additional technical difficulty in our case (consider e.g. Assumption~\ref{ass:model_intro}) is that even though $ \Phi_{ A}$ and $ \Phi_{ B\to A}$ are globally Lipschitz, the product $ \Phi_{ A} \Phi_{ B\to A}$ is not and the usual result of \cite{MR3449317} no longer applies. The key point to note is that $ 0 \leq F(x, y) \leq c_{ 1}(x+y) + c_{ 2}$, one can stochastically dominate our process by a \emph{linear Hawkes process} that is known to exist. Therefore, the proof relies on a thinning procedure (similar to \cite{costa2020renewal}) and is carried out in Section~\ref{sec:proof_WP_part}. For similar ideas see  \cite{raad2020stability2} Prop. 1.4 or \cite{costa2020renewal} Prop.	 2.

\subsection{The mean-field limit}
\label{sec:mean_field_limit}
Following a standard procedure for mean-field systems \cite{SznitSflour}, it is easy to derive the limit nonlinear process of \eqref{eq:Hawkes} as $N\to\infty$. In the Hawkes setting, it reduces to an inhomogeneous Poisson process \cite{MR3449317}. Informally speaking, as $N\to \infty$, the empirical mean in \eqref{eq:model} concentrates around an expectation and it is natural to introduce the following: let $(\pi, \tilde{ \pi})$ be independent Poisson point processes on $ [0, +\infty)^{ 2}$ with intensity $ {\rm d}s {\rm d}z$. Consider 
\begin{equation}
\label{eq:barZ_gen}
\begin{cases}
\bar Z_{t}^{A}&= \int_{0}^{t}\int_{0}^{\infty}\mathbf{1}_{z\le F\left(\alpha\int_{ 0}^{s} h_{ 1}(s-u) {\rm d} \mathbf{ E} (\bar Z_{u}^{A}), (1- \alpha) \int_{ 0}^{s} h_{ 2}(s-u) {\rm d} \mathbf{ E} \left(\bar Z_{u}^{B}\right)\right)}\pi(\rmd s,\rmd z),\\
\bar Z_{t}^{B}&= \int_{0}^{t}\int_{0}^{\infty}\mathbf{1}_{z\le G \left((1- \alpha) \int_{ 0}^{s} h_{ 3}(s-u) {\rm d} \mathbf{ E} \left(\bar Z_{u}^{B}\right), \alpha\int_{ 0}^{s} h_{ 4}(s-u) {\rm d} \mathbf{ E} \left(\bar Z_{u}^{{A}}\right)\right)}\tilde{ \pi}(\rmd s,\rmd z),
\end{cases}\ t\geq0\ .
\end{equation}
The nonlinearity of the process $\bar Z:= (\bar Z^{ A}, \bar Z^{ B})$ lies in the fact that it interacts with its own expectation $(m_{  t}^{A}, m_{  t}^{B}):= \left( \mathbb{ E} \left(\bar Z_{t}^{A}\right), \mathbb{ E} \left(\bar Z_{t}^{B}\right)\right)$ which solves
\begin{equation}
\label{eq:mtAB}
\begin{cases}
m_{  t}^{A}&= \int_{0}^{t}{F \left( \alpha\int_{ 0}^{s} h_{ 1}(s-u) {\rm d} m_{  u}^{A},(1- \alpha) \int_{ 0}^{s} h_{ 2}(s-u) {\rm d} m_{  u}^{B} \right)}\rmd s,\\
m_{  t}^{B}&= \int_{0}^{t}  G \left((1- \alpha) \int_{ 0}^{s} h_{ 3}(s-u) {\rm d} m_{  u}^{B}, \alpha\int_{ 0}^{s} h_{ 4}(s-u) {\rm d} m_{  u}^{A}\right) \rmd s.
\end{cases}
\end{equation}
The following result is a well-posedness result for both \eqref{eq:barZ_gen} and \eqref{eq:mtAB}:
\begin{proposition}
\label{prop:mAB}
Under Assumptions~\ref{ass:main_h} and~\ref{ass:main_FG}, for all $T>0$, there exists a unique nondecreasing in both coordinates locally bounded solution $(m_{  t}^{A}, m_{  t}^{B})_{ t\in [0, T]}$ to \eqref{eq:mtAB} that is of class $ \mathcal{ C}^{ 1}$ on $[0, T]$. Moreover, for any given independent Poisson point processes $ \pi$ and $ \tilde{ \pi}$ on $[0, +\infty)$, there exists a pathwise unique process $(\bar Z_{t}^{A}, \bar Z_{t}^{B})$ to \eqref{eq:barZ_gen} such that $( \mathbb{ E}(\bar Z_{t}^{A}), \mathbb{ E}(\bar Z_{t}^{B}))$ is locally bounded.
\end{proposition}
Proposition~\ref{prop:mAB} is proven in Section~\ref{sec:proof_WP_MF}.

\medskip

We now turn to the propagation of chaos result: there exists a well-chosen coupling such that the particle system \eqref{eq:model} and the nonlinear process \eqref{eq:barZ_gen} are close. Suppose that $H$ is locally square integrable:
\begin{equation}
\label{hyp:H2int}
\|H\|_{2,T} <\infty,\ T>0.
\end{equation}
Define $(\bar Z_{ t}^{ i})_{ i=1, \ldots, N}= \left((\bar Z_{ t}^{ i})_{ i=1, \ldots, N_{ A}}, (\bar Z_{ t}^{ i})_{ i=N_{ A}+1, \ldots, N}\right)$ solution to \eqref{eq:barZ_gen} driven by the same Poisson measures $ ( \pi_{ i})_{ i=1, \ldots, N}$ as for \eqref{eq:Hawkes}. The main convergence result is the following:
\begin{proposition}
\label{prop:conv}
Under Assumption \ref{ass:main_FG} and \eqref{hyp:H2int}, for all $T>0$, there exists some constant $C>0$ (depending on $T$ and the parameters of the model) such that \begin{equation}
\label{eq:conv}
\sup_{ i=1, \ldots, N} \mathbf{ E} \left[\sup_{ t\in [0, T]} \left\vert Z_{ t}^{i} - \bar Z_{ t}^{i}\right\vert\right] \leq \frac{  C}{ \sqrt{N}}.
\end{equation}

\end{proposition}

\begin{remark}\label{rem:maj_cst_conv} Under the Assumptions of Proposition \ref{prop:conv}, assuming additionally   (see \eqref{eq:kappas} for the definitions of $\kappa_{ i}$, $i=1, \ldots, 4$) that $\max\{(c_{3}+c_{4}\kappa_{1}\|\lambda^{A}\|_{\infty}+c_{5})(\kappa_{1}+\kappa_{2}), c_{6}(\kappa_{3}+\kappa_{4})\}<1$, we have furthermore that the constant $C$ in \eqref{eq:conv} can be chosen linear in $T$: $C= \tilde C T$, where $\tilde C>0$ does not depend on $T$. This constraint requires an explicit estimate on $\sup_{ t\geq 0} \lambda_{ t}^{ A}$. Although we will obtain in Section~\ref{sec:MFL} that $\sup_{ t\geq 0} \lambda_{ t}^{ A}<\infty$ under fairly general assumptions, it may be difficult to derive an explicit upper bound for $\sup_{ t\geq 0} \lambda_{ t}^{ A}$ in general. Nonetheless, this can easily be achieved in the model of Example \ref{ex:linear}, under non optimal conditions: since  $ \Phi_{ B\to A}\leq 1$, population $A$ is  stochastically bounded by a linear Hawkes process with memory kernel $ \alpha h_{ 1}$, that is subcritical if $ \kappa_{ 1}:= \alpha \left\Vert h_{ 1} \right\Vert_{ L^{ 1}}<1$, we obtain immediately the uniform a priori bound $ \sup_{ t\geq 0}\lambda^{ A}_{ t} \leq \frac{ \mu_{ A}}{ 1- \kappa_{ 1}}$ (see Lemma 23 of \cite{MR3449317}). In that case, the quantities $c_{j},\ 1\le j\le 6,$ are explicit and the latter condition simplifies into  $\max\{ (1+\kappa_{1}\frac{\mu_{A}}{1-\kappa_{1}})(\kappa_{1}+\kappa_{2}),\ (\kappa_{3}+\kappa_{4})\}<1$.
 \end{remark}

Proposition~\ref{prop:conv} is proven in Section~\ref{sec:proof_prop_chaos}.

\section{Long time dynamics of the mean-field limit in the inhibition model\label{sec:MFL} }

We are now interested in the behavior as $t\to\infty$ of the macroscopic intensities $ \lambda_{  t}^{A}:= \frac{ {\rm d}}{ {\rm d}t} m_{  t}^{A}$ and $  \lambda_{ t}^{B}:= \frac{ {\rm d}}{ {\rm d}t} m_{  t}^{B}$, where $(m_{A, t}, m_{  t}^{B})$ solves \eqref{eq:mtAB}. We specify here the analysis to the model described in Assumption~\ref{ass:model_intro}: for the choice of $F$ and $G$ in \eqref{eq:ex_F_G}, we see from \eqref{eq:mtAB} that $( \lambda_{ A}, \lambda_{ B})$ solves
\begin{equation}
\label{eq:lambdaAB}
\begin{cases}
\lambda_{  t}^{A}&= \Phi_{ A} \left( \alpha\int_{ 0}^{t} h_{ 1}(t-u)  \lambda_{ u}^{A}{\rm d} u\right) \Phi_{ B\to A} \left((1- \alpha) \int_{ 0}^{t} h_{ 2}(t-u)  \lambda_{ u}^{B} {\rm d}u\right),\\
 \lambda_{ t}^{B}&= \Phi_{ B} \left((1- \alpha) \int_{ 0}^{t} h_{ 3}(t-u)  \lambda_{ u}^{B}{\rm d}u\right) + \Phi_{ A\to B} \left( \alpha\int_{ 0}^{t} h_{ 4}(t-u)  \lambda_{ u}^{A} {\rm d}u\right).
\end{cases}
\end{equation}
In the specific case of Example \ref{ex:linear}, \eqref{eq:lambdaAB} becomes \begin{equation}
\label{eq:lambdaAB_lin}
\begin{cases}
\lambda_{  t}^{A}&= \left(\mu_{ A} + \alpha\int_{ 0}^{t} h_{ 1}(t-u)  \lambda_{ u}^{A}{\rm d} u\right) \Phi_{ B\to A} \left((1- \alpha) \int_{ 0}^{t} h_{ 2}(t-u)  \lambda_{ u}^{B} {\rm d}u\right),\\
 \lambda_{ t}^{B}&= \mu_{ B} + (1- \alpha) \int_{ 0}^{t} h_{ 3}(t-u)  \lambda_{ u}^{B} {\rm d}u + \Phi_{ A\to B} \left(\alpha \int_{ 0}^{t} h_{ 4}(t-u)  \lambda_{ u}^{A} {\rm d}u\right).
\end{cases}
\end{equation}
Here that we assume that $h_{ i}\geq0$ for $i=1, \ldots, 4$. By Proposition~\ref{prop:mAB}, there is a unique solution $(\lambda^{ A}, \lambda^{ B})$ in $ \mathcal{ C}([0, T], \left(\mathbb{ R}^{ +}\right)^{ 2})$ to the system \eqref{eq:lambdaAB}: existence is provided by $( \lambda_{  t}^{A},  \lambda_{ t}^{B}):= (m_{  t}^{\prime A}, m_{  t}^{\prime B})$ where $(m^{ A}, m^{ B})$ solves \eqref{eq:mtAB} and uniqueness holds since for any solution $(\lambda^{ A}, \lambda^{ B})$, the pairwise $(m_{  t}^{A}, m_{  t}^{B}):= \left( \int_{ 0}^{t}  \lambda_{ s}^{A} {\rm d}s, \int_{ 0}^{t} \lambda_{ s}^{B} {\rm d}s\right)$ is the unique solution to \eqref{eq:mtAB}. 

To simplify notations below, set 
\begin{align}
\label{eq:kappas}
\kappa_{ 1}:= \alpha \left\Vert h_{ 1} \right\Vert_{ 1},\ \kappa_{ 2}:= (1- \alpha) \left\Vert h_{ 2} \right\Vert_{ 1},\  \kappa_{ 3}:= (1- \alpha) \left\Vert h_{ 3} \right\Vert_{ 1},\  \kappa_{ 4}= \alpha \left\Vert h_{ 4} \right\Vert_{ 1}.
\end{align}
It appears that the longtime behavior of system \eqref{eq:lambdaAB} (and of its particular case \eqref{eq:lambdaAB_lin}) depends strongly on the connectivity between populations $A$ and $B$: crucial criteria are the absence/presence of inhibition from $B$ to $A$ ($ \kappa_{ 2}=0$ or $ \kappa_{ 2}>0$) and the absence/presence of retroaction from $A$ to $B$ ($ \kappa_{ 4}=0$ or $ \kappa_{ 4}>0$). The analysis of the three simpler cases where the system is not fully-connected (one among $ \kappa_{ 2}$ or $ \kappa_{ 4}$ is zero) is detailed in Section~\ref{sec:nonfully_coupled} below.  Some features are commonly observed in all three cases, but we have chosen a separate exposition for purpose of clarity. 
The more complicated case with full connectivity (that is $ \kappa_{ 2}>0$ and $ \kappa_{ 4}>0$),  reveals richer dynamical patterns and  is analyzed separately in Section~\ref{sec:fully_coupled}.

\subsection{Hypotheses}
In addition to Assumptions~\ref{ass:main_h} and~\ref{ass:model_intro}, we require the following:
\begin{assumption}
\label{ass:Phis}
Suppose that $ \Phi_{ B\to A}$ is non-increasing with $ \Phi_{ B\to A} (x) \to0$ as $x\to\infty$ and, without loss of generality, that $ \Phi_{ B\to A}(0)=1$. Similarly, suppose  that $ \Phi_{ A\to B}$ is nondecreasing with $ \Phi_{ A\to B}(0)=0$ and $ \Phi_{ A\to B} (x) \to+\infty$ as $x\to\infty$. Finally, assume that
\begin{equation}
\label{hyp:h_to_0}
h_{ i}(u)\geq 0,\ \text{ and } h_{ i}(u) \xrightarrow[ u\to \infty]{}0,\ u\geq0,\ i=1, \ldots, 4. 
\end{equation} 
\end{assumption}  

\subsection{Decoupled cases}
\label{sec:nonfully_coupled}
We analyze in this section the long-time behavior of \eqref{eq:lambdaAB} and \eqref{eq:lambdaAB_lin} in the simpler cases where $ \kappa_{ 2}=0$ or $ \kappa_{ 4}=0$.

\subsubsection{ The fully decoupled case}
\label{sec:fully_decoupled}
Suppose  that  $ \kappa_{ 2}= \kappa_{ 4}=0$, there is no inhibition and no retroaction. The dynamics of \eqref{eq:lambdaAB} is somehow trivial since populations $A$ and $B$ behave independently from one another. Nonetheless consider this case as a control model for the other cases. 
\begin{proposition}
\label{prop:k2_0_k4_0}
Suppose  that Assumptions~\ref{ass:main_h},~\ref{ass:model_intro} and~\ref{ass:Phis} are satisfied with $ \kappa_{ 2}= \kappa_{ 4}=0$. If $ \left\Vert \Phi_{ A} \right\Vert_{ L} \kappa_{ 1}<1$ (resp. $ \left\Vert \Phi_{ B} \right\Vert_{ L} \kappa_{ 3}<1$), then population $A$ (resp. $B$) is subcritical: $\sup_{ t\geq0} \lambda_{  t}^{A}<\infty$ (resp. $\sup_{ t\geq0}  \lambda_{ t}^{B}<\infty$). Moreover, in the linear case \eqref{eq:lambdaAB_lin}, 
\begin{enumerate}
\item if $ \kappa_{ 1}<1$ (resp.  $ \kappa_{ 3}<1$), then population $A$ (resp. $B$) is subcritical and $ \lambda_{  t}^{A} \xrightarrow[ t\to\infty]{} \frac{ \mu_{ A}}{ 1- \kappa_{ 1}}$ (resp. $  \lambda_{ t}^{B} \xrightarrow[ t\to\infty]{} \frac{ \mu_{ B}}{ 1- \kappa_{ 3}}$),
\item if $ \kappa_{ 1}>1$ and $ \mu_{ A}>0$ (resp.  $ \kappa_{ 3}>1$ and $ \mu_{ B}>0$), then population $A$ (resp. $B$) is supercritical and $ \lambda_{  t}^{A} \xrightarrow[ t\to\infty]{} +\infty$ (resp. $  \lambda_{ t}^{B} \xrightarrow[ t\to\infty]{} +\infty$).
\end{enumerate}

\end{proposition}
\begin{proof}[Proof of Proposition~\ref{prop:k2_0_k4_0}]
 This case is an immediate consequence of \cite{MR3449317}, Th.~10 and 11. Note also that this result can be retrieved easily from the estimates given in Section~\ref{sec:proofs_longtime} below.
 \end{proof}

\subsubsection{The case with inhibition and no retroaction}
\label{sec:k2_pos_k4_0}
Suppose  that  $ \kappa_{ 2}>0$ and $\kappa_{ 4}=0$: there is   inhibition from $B$ to $A$ but no retroaction from $A$ to $B$. In this case, population $B$ behaves independently from $A$. 
\begin{proposition}
\label{prop:k2_pos_k4_0}
Suppose  that Assumptions~\ref{ass:main_h},~\ref{ass:model_intro} and~\ref{ass:Phis} are satisfied with $ \kappa_{ 2}>0$ and $\kappa_{ 4}=0$.
If $ \left\Vert \Phi_{ B} \right\Vert_{ L} \kappa_{ 3}<1$, then population $B$ is subcritical: $\sup_{ t\geq0}  \lambda_{ t}^{B}<\infty$. Moreover, in the particular case of \eqref{eq:lambdaAB_lin}, 
\begin{enumerate}
\item If $ \kappa_{ 3}<1$, then population $B$ is subcritical: $  \lambda_{ t}^{B} \xrightarrow[ t\to\infty]{} \frac{ \mu_{ B}}{ 1- \kappa_{ 3}}$. Moreover,
\begin{enumerate}
\item if $ \kappa_{ 1} \Phi_{ B\to A} \left( \frac{ \kappa_{ 2} \mu_{ B}}{ 1- \kappa_{ 3}}\right)<1$, then population $A$ is subcritical and
$
\lambda_{  t}^{A} \xrightarrow[ t\to\infty]{} \frac{ \mu_{ A}\Phi_{ B\to A} \left( \frac{ \kappa_{ 2} \mu_{ B}}{ 1- \kappa_{ 3}}\right)}{1- \kappa_{ 1}\Phi_{ B\to A} \left( \frac{ \kappa_{ 2} \mu_{ B}}{ 1- \kappa_{ 3}}\right)},
$
\item if $ \kappa_{ 1} \Phi_{ B\to A} \left( \frac{ \kappa_{ 2} \mu_{ B}}{ 1- \kappa_{ 3}}\right)>1$ and $ \mu_{ A}>0$, then population $A$ is supercritical: $ \lambda_{  t}^{A} \xrightarrow[ t\to\infty]{}+\infty$.
\end{enumerate}
\item If $ \kappa_{ 3}>1$ and $ \mu_{ B}>0$: population $B$ is supercritical and $  \lambda_{ t}^{B} \xrightarrow[ t\to\infty]{}+\infty$. Regardless of the values of $ \mu_{ A}, \kappa_{ 1}, \kappa_{ 2}>0$, we have that $ \lambda_{  t}^{A} \xrightarrow[ t\to \infty]{}0$.
\end{enumerate}
\end{proposition}
Proposition~\ref{prop:k2_pos_k4_0} reveals a first effect of inhibition: when $ \kappa_{ 2}=0$, (no inhibition, Proposition~\ref{prop:k2_0_k4_0}) the critical parameter for $ \kappa_{ 1}$ is $ \kappa_{ 1, c}^{ \text{ isolated}}:=1$. With a subcritical inhibition ($ \kappa_{ 3}<1$), the critical parameter for $ \kappa_{ 1}$ becomes (since $ \Phi_{ B\to A}\leq1$) $ \kappa_{ 1, c}^{ \text{ inhib}}:= { 1}/{ \Phi_{ B\to A} \left( \frac{ \kappa_{ 2} \mu_{ B}}{ 1- \kappa_{ 3}}\right)} \geq \kappa_{ 1, c}^{ \text{ isolated}}$, the inequality being strict when $ \Phi_{ B\to A}$ strictly decreases. Note that $ \kappa_{ 1, c}^{ \text{ inhib}}$ increases as either the intrinsic activity of population $B$ increases ($ \mu_{ B}$ increases), or population $B$ gets closer to criticality ($\kappa_{3}\nearrow 1$), or in presence of a longer memory effect of population $B$ onto population $A$ ($\kappa_{2}$ increases).
\begin{remark}
\label{rem:effect_inhib}
This allows in particular for regimes where $ \kappa_{ 1, c}^{ \text{ inhib}} > \kappa_{ 1} > \kappa_{ 1, c}^{ \text{ isolated}}=1$: in absence of inhibition, population $A$ may be supercritical, but inhibition brings back population $A$ into subcriticality. This situation is all the more spectacular when population $B$ is supercritical ($ \kappa_{ 3}>1$): supercritical inhibition simply kills the dynamics of $A$. This result reflects biological observations: \textit{epilepsy} is a chronic condition characterized by recurrent seizures. Episodes of \textit{seizures}, involving abnormal synchronous firing of large groups of neurons \cite{luo:15}, constitute an actual example of population $A$ becoming supercritical when recurrent inhibition from population $B$ is altered or suppressed. \end{remark}

Proof of Proposition~\ref{prop:k2_pos_k4_0} is given in Section~\ref{sec:proof_k2_pos_k4_0}.
 \subsubsection{The case with retroaction and no inhibition}
 \label{sec:k2_0_k4_pos}
Suppose  $ \kappa_{ 2}=0$ and $\kappa_{ 4}>0$: there is no inhibition from $B$ to $A$ but presence of a retroaction from $A$ to $B$. In this case, population $A$ behaves independently from $B$ and has a sole excitatory influence on $B$.  
 \begin{proposition}
\label{prop:k2_0_k4_pos}
Suppose  that Assumptions~\ref{ass:main_h},~\ref{ass:model_intro} and~\ref{ass:Phis} are satisfied with $ \kappa_{ 2}=0$ and $\kappa_{ 4}>0$. If $ \left\Vert \Phi_{ A} \right\Vert_{ L} \kappa_{ 1}<1$, then population $A$ is subcritical: $\sup_{ t\geq0} \lambda_{  t}^{A}<\infty$. Moreover, in the particular case of \eqref{eq:lambdaAB_lin}, 
\begin{enumerate}
\item If $ \kappa_{ 1}<1$, then population $A$ is subcritical: $ \lambda_{  t}^{A} \xrightarrow[ t\to\infty]{} \frac{ \mu_{ A}}{ 1- \kappa_{ 1}}$. Moreover,
\begin{enumerate}
\item if $ \kappa_{ 3} <1$, then population $B$ is subcritical and $ \lambda_{ t}^{B} \xrightarrow[ t\to\infty]{} \frac{ \mu_{ B}}{ 1- \kappa_{ 3}} +\frac{1}{1- \kappa_{ 3}} \Phi_{ A\to B} \left(\frac{ \kappa_{ 4} \mu_{ A}}{1- \kappa_{ 1}}\right)$,
\item if $ \kappa_{ 3}>1$ and $ \mu_{ B}>0$, then population $B$ is supercritical: $  \lambda_{ t}^{B} \xrightarrow[ t\to\infty]{}+\infty$.
\end{enumerate}
\item If $ \kappa_{ 1}>1$ and $ \mu_{ A}>0$: both populations $A$ and $B$ are supercritical, $ \lambda_{  t}^{A} \xrightarrow[ t\to\infty]{}+\infty$ and $  \lambda_{ t}^{B} \xrightarrow[ t\to\infty]{}+\infty$.
\end{enumerate}
\end{proposition}
As anticipated, whenever $A$ is supercritical, so is population $B$ whatever its own intrinsic behavior. Proposition~\ref{prop:k2_0_k4_pos} is proven in Section~\ref{sec:proof_k2_0_k4_pos}.

\subsection{The fully-coupled case}
\label{sec:fully_coupled}
Consider now the more interesting case with full connectivity, i.e. $\kappa_{2}>0$ and $\kappa_{4}>0$: both inhibition from $B$ to $A$ and retroaction from $A$ to $B$ are present. We first propose a general result for Model  \eqref{eq:lambdaAB} in Theorem \ref{th:A_subcritical} and then specialize to the linear Model \eqref{eq:lambdaAB_lin}.

\subsubsection{A general result of subcriticality}
\label{sec:gen_res_sub}
\begin{theorem}
\label{th:A_subcritical}
Consider Model \eqref{eq:lambdaAB} satisfying Assumptions~\ref{ass:main_h},~\ref{ass:model_intro} and~\ref{ass:Phis} with $\kappa_{2} \kappa_{ 4}>0$. Then, the following holds true:
\begin{enumerate}
\item Whatever the values of $ \kappa_{ i}$, $i=1, \ldots, 4$, $ \lambda_{  t}^{A}$ does not tend to $+\infty$ as $t\to \infty$.
 \item If $  \lambda_{ t}^{B} \xrightarrow[ t\to \infty]{}+\infty$, it holds that
$\lambda_{  t}^{A} \xrightarrow[ t\to \infty]{} 0.$
\end{enumerate} 
\end{theorem}
We retrieve the same phenomenon described in Remark~\ref{rem:effect_inhib} in an even more universal way: in the regime of full connectivity $ \kappa_{ 2}>0$ and $ \kappa_{ 4}>0$, population $A$ cannot be supercritical, even for arbitrary large values of $ \kappa_{ 1}$ (possibly larger than $1$, when population $A$, isolated from $B$, is supercritical). Note that Theorem~\ref{th:A_subcritical} only states a weak form of subcriticality: we only prove that $ \lim\inf _{t\to\infty} \lambda_{ t}^{A}<\infty$. In  the whole generality of the hypotheses of Theorem~\ref{th:A_subcritical}, one cannot rule out the possibility of having $ \lim\sup_{t\to\infty}\lambda_{  t}^{A}=+\infty$, although this has not been observed numerically. 

Proof of Theorem~\ref{th:A_subcritical} is given in Section~\ref{sec:th_A_subcritical}.

 \subsubsection{ Specifying to the linear model \label{sec:B_supercritical}}
 We now provide  more precise estimates in the case of the linear Model \eqref{eq:lambdaAB_lin} of Example~\ref{ex:linear}. The large time  asymptotic of \eqref{eq:lambdaAB_lin} depends crucially on the behavior of the population $B$. The critical parameter here is $ \kappa_{ 3}$: $ \kappa_{ 3}<1$ (resp. $ \kappa_{ 3}>1$) means that population $B$,  isolated from population $A$, is subcritical (resp. supercritical). Let us mention briefly the easiest second case $ \kappa_{ 3}>1$:
\begin{proposition}
\label{prop:B_sup}
Consider Model \eqref{eq:lambdaAB_lin} satisfying Assumptions~\ref{ass:main_h},~\ref{ass:model_intro} and~\ref{ass:Phis} with $\kappa_{2}>0$ and $\kappa_{4}>0$. Suppose also that $ \mu_{ B}>0$ and that population $B$, in isolation, is supercritical, i.e. 
$\kappa_{ 3}>1.$
Then, the following is true:
$ \lambda_{ t}^{B} \xrightarrow[ t\to \infty]{} +\infty$ and 
$\lambda_{  t}^{A} \xrightarrow[ t\to \infty]{}0.$
\end{proposition}
Proposition~\ref{prop:B_sup} is a straightforward consequence of Theorem~\ref{th:A_subcritical}: one has simply to check that  $ \lambda_{ t}^{ B} \xrightarrow[ t\to\infty]{}+\infty$ in this case. This follows directly from $ \lambda_{ t}^{B}\geq \mu_{ B} + (1- \alpha) \int_{ 0}^{t} h_{ 3}(t-u)  \lambda_{ u}^{B} {\rm d}u$ (recall that $ \Phi_{ A\to B}\geq0$). Details are given in Section~\ref{sec:th_B_sup} for the sake of completeness. 

\bigskip

From now on, we restrict to the case $\kappa_{ 3}<1$. Note that the arguments leading to Theorem~\ref{th:full_conv} in the fully coupled case $ \kappa_{ 2} \kappa_{ 4}>0$ can easily be adapted to the decoupled cases $ \kappa_{ 2} \kappa_{ 4}=0$ mentioned earlier (with a considerably simpler analysis in these cases). Nonetheless, we have chosen a separate exposition for purpose of clarity. Before stating  the main result of this paragraph, we introduce some notations. We are interested here in a priori bounds concerning
\begin{equation}
\label{eq:def_ells}
\begin{cases}
\underline\ell_{ A}&:= \liminf_{ t\to\infty} \lambda_{  t}^{A}\in [0, \infty],\hspace{1cm} \bar \ell_{ A}:= \limsup_{ t\to\infty} \lambda_{  t}^{A}\in [0, \infty],\\
\underline\ell_{ B}&:= \liminf_{ t\to\infty}  \lambda_{ t}^{B}\in [0, \infty],\hspace{1cm} \bar \ell_{ B}:= \limsup_{ t\to\infty}  \lambda_{ t}^{B}\in [0, \infty].
\end{cases}
\end{equation} 
For fixed $ \kappa_{ 1}, \kappa_{ 2}$, define the interval
\begin{equation}
\label{eq:I_kappas}
 \mathcal{ I}_{ \kappa_{ 1}, \kappa_{ 2}}:= \left\lbrace x\geq0,\ \kappa_{ 1} \Phi_{ B\to A} \left( \kappa_{ 2}x\right)< 1\right\rbrace.
\end{equation}
Under Assumption~\ref{ass:Phis}, $ \mathcal{ I}_{ \kappa_{ 1}, \kappa_{ 2}}=[0, +\infty)$ if $ \kappa_{ 1}<1$ and if $ \kappa_{ 1}\geq 1$ $ \mathcal{ I}_{  \kappa_{ 1}, \kappa_{ 2}}= (x^{ \ast}_{ \kappa_{ 1}, \kappa_{ 2}}, +\infty)$ for some $x^{ \ast}_{ \kappa_{ 1}, \kappa_{ 2}}>0$. Define 
the following functions 
\begin{align}
\Psi_{ 1}&:= x\in \mathcal{ I}_{ \kappa_{ 1}, \kappa_{ 2}} \mapsto \frac{ \mu_{ A}\Phi_{ B\to A} \left( \kappa_{ 2}x\right)}{1- \kappa_{ 1}\Phi_{ B\to A} \left( \kappa_{ 2}x\right)},\label{eq:Psi1}\\
\Psi_{ 2}&:= x\geq0 \mapsto\frac{ \mu_{ B}}{ 1- \kappa_{ 3}} + \frac{ \Phi_{ A\to B} \left( \kappa_{ 4} x\right)}{1- \kappa_{ 3}},\label{eq:Psi2}
\end{align}
as well as
\begin{align}
\Phi:= \Psi_{ 2}\circ \Psi_{ 1}:& x \in \mathcal{ I}_{ \kappa_{ 1}, \kappa_{ 2}}\mapsto \frac{ 1}{ 1- \kappa_{ 3}} \left( \mu_{ B}+  \Phi_{ A\to B} \left( \frac{ \kappa_{ 4}\mu_{ A}\Phi_{ B\to A} \left( \kappa_{ 2}x\right)}{1- \kappa_{ 1}\Phi_{ B\to A} \left( \kappa_{ 2}x\right)}\right)\right).\label{eq:Phi}
\end{align}
\begin{lemma}
\label{lem:Phi}
The function $ \Phi$ given by \eqref{eq:Phi} is continuous, nonincreasing and has a unique fixed-point $\ell \in \mathcal{ I}_{  \kappa_{ 1}, \kappa_{ 2}}$. This fixed-point satisfies 
\begin{equation}
\label{eq:lower_bound_ell}
\frac{ \mu_{ B}}{ 1- \kappa_{ 3}}\leq \ell.
\end{equation}
\end{lemma}
\begin{proof}[Proof of Lemma~\ref{lem:Phi}]
It is a direct consequence of Assumption~\ref{ass:Phis}, since $ \Psi_{ 1}$ is nonincreasing and $ \Psi_{ 2}$ is nondecreasing. If $ \kappa_{ 1}\geq1$, $ \Phi(x) \xrightarrow[ x\to x^{ \ast}_{ \kappa_{ 1}, \kappa_{ 2}}]{}+ \infty$, as $ \lim_{ +\infty} \Phi_{ A\to B}= +\infty$. Since $ \Phi_{ A\to B}\geq0$, the a priori bound \eqref{eq:lower_bound_ell} is straightforward.
\end{proof}
Finally, define  the following domain
\begin{align}
\mathcal{ U}_{ \Phi}&:=  \left\lbrace (u, v)\in \left(\mathcal{ I}_{ \kappa_{ 1}, \kappa_{ 2}}\right)^{ 2},\, \Phi(v)\leq u \leq \ell\leq v \leq \Phi(u),\ \frac{ \mu_{ B}}{ 1- \kappa_{ 3}}\leq u\right\rbrace.\label{eq:def_UPhi}
\end{align}
The set $ \mathcal{ U}_{ \Phi}$ is nonempty, as $(\ell, \ell)\in \mathcal{ U}_{ \Phi}$, by \eqref{eq:lower_bound_ell}. 
\begin{assumption}
\label{ass:U}
The set $ \mathcal{ U}_{ \Phi}$ is a singleton: 
 $\mathcal{ U}_{ \Phi}= \left\lbrace (\ell, \ell)\right\rbrace.$
\end{assumption}

Hereafter in Proposition~\ref{prop:verify_U} we provide sufficient conditions for Assumption~\ref{ass:U} as well as examples in Section~\ref{sec:examples} where these are satisfied. The main convergence result of this paragraph is the following:
\begin{theorem}
\label{th:full_conv}
Consider Model \eqref{eq:lambdaAB_lin} satisfying Assumptions~\ref{ass:main_h},~\ref{ass:model_intro} and~\ref{ass:Phis} with  $\kappa_{2}>0$, $\kappa_{4}>0$, $\kappa_{3}<1$ and $ \mu_{ A}>0$. Then the following is true:
\begin{enumerate}
\item Population $B$ is subcritical: $\sup_{ t\geq0}  \lambda_{ t}^{B}<\infty$ and one has the a priori bounds (recall the definition of $ \mathcal{ I}_{ \kappa_{ 1}, \kappa_{ 2}}$ in \eqref{eq:I_kappas}): 
\begin{equation}
\label{eq:apriori_ellsB}
\underline{ \ell}_{ B}\geq \frac{ \mu_{ B}}{ 1- \kappa_{ 3}} \text{ and }\bar \ell_{ B}\in \mathcal{ I}_{ \kappa_{ 1}, \kappa_{ 2}}.
\end{equation}
\item If $ \lambda_{  t}^{A}$ converges as $t\to\infty$ to $\ell_{ A}$, then $  \lambda_{ t}^{B}$ converges as well, to $ \ell_{ B}= \Psi_{ 2}(\ell_{ A})$ where $ \Psi_{ 2}$ is given in \eqref{eq:Psi2}.
\item Suppose furthermore that Assumption~\ref{ass:U} is true. Then there is an equivalence between
\begin{enumerate}
\item $ \underline{ \ell}_{ B}$ verifies
\begin{equation}
\label{eq:uB_in_I}
\underline{ \ell}_{ B} \in \mathcal{ I}_{ \kappa_{ 1}, \kappa_{ 2}}
\end{equation} 
\item $  \lambda_{ t}^{B}$ as a finite limit as $t\to\infty$.
\end{enumerate}
If this is true, the limit $\lim_{ t\to \infty} \lambda_{ t}^{ B}:= \ell_{ B}=\ell$ where $ \ell$ is the fixed-point of $ \Phi$ given by \eqref{eq:Phi}. Moreover, $ \sup_{ t\geq0} \lambda_{  t}^{A}< \infty$ (population $A$ is  subcritical) and $ \lambda_{  t}^{A} \xrightarrow[ t \to \infty]{} \ell_{ A}:= \Psi_{ 1}(\ell_{ B})$.
\end{enumerate}
\end{theorem}
Proof of Theorem~\ref{th:full_conv} is given in Section~\ref{sec:th_full_conv}.

\subsubsection{Remarks on the main convergence result}
\label{sec:remarks_gen}
Theorem~\ref{th:full_conv} comes with two hypotheses: firstly, Assumption~\ref{ass:U} for which we give sufficient conditions in Proposition~\ref{prop:verify_U} and is verified for several examples in Section~\ref{sec:examples}, and secondly, a technical a priori estimate \eqref{eq:uB_in_I} on $ \underline{ \ell}_{ B}$. Several remarks concerning these two hypotheses are necessary: firstly, from $ \underline{ \ell}_{ B}\geq \frac{ \mu_{ B}}{ 1- \kappa_{ 3}}$ and the fact $ \Phi_{ B\to A}$ is nonincreasing, we obtain immediately that \eqref{eq:uB_in_I} is  satisfied under the following sufficient condition:
\begin{equation}
\label{hyp:A_sub_muB_kappa3}
\kappa_{ 1} \Phi_{ B\to A} \left(\frac{ \mu_{ B}\kappa_{ 2}}{ 1- \kappa_{ 3}}\right)<1.
\end{equation}
This condition is the same subcriticality condition met in Proposition~\ref{prop:k2_pos_k4_0}. We retrieve the phenomenon mentioned in Remark~\ref{rem:effect_inhib}: inhibition has a tendency to quench supercriticality of population $A$. Numerical results however suggest that  the convergence of Theorem~\ref{th:full_conv} may remain true (and hence so is \eqref{eq:uB_in_I}) even though \eqref{hyp:A_sub_muB_kappa3} is violated (see Figure~\ref{fig:lambdaA_pol_CondlB} below). Remark also that the situation when $ \underline{ \ell}_{ B} \notin \mathcal{ I}_{ \kappa_{ 1}, \kappa_{ 2}}$ would imply, by \eqref{eq:apriori_ellsB}, $\underline{ \ell_{ B}} < \bar \ell_{ B}$ and thus correspond to a case where neither $ \lambda^{ B}$ nor $ \lambda^{ A}$ converge, but oscillate. Giving some necessary and sufficient condition for \eqref{eq:uB_in_I} (and in particular giving some sufficient condition for $\underline{\ell}_{ B}\notin \mathcal{ I}_{ \kappa_{ 1}, \kappa_{ 2}}$, hence proving the existence of oscillations for \eqref{eq:lambdaAB_lin}) would require to have some precise a priori upper bounds on $ \underline{ \ell}_{ B}$, that we do not have for the moment (what we easily get is a lower bound on $\underline{ \ell}_{ B}$, not a upper bound). Our conjecture is that \eqref{eq:uB_in_I} remains true as long as $ \Phi_{ B\to A}$ is convex (independently of \eqref{hyp:A_sub_muB_kappa3}), although we do not have a proof for this statement. We give in Figure~\ref{fig:lA_atan_k1_5} a numerical example of $\underline{\ell}_{ B}\notin \mathcal{ I}_{ \kappa_{ 1}, \kappa_{ 2}}$, in a case when $ \Phi_{ B\to A}$ is not convex.

\medskip

A crucial observation of the proof of Theorem~\ref{th:full_conv} is that, under \eqref{eq:uB_in_I}, we always have 
\begin{equation}
\label{eq:ellsBinU}
(\underline{\ell}_{ B}, \bar{\ell}_{ B}) \in \mathcal{ U}_{ \Phi}.
\end{equation} 
This fact, which does not rely on Assumption~\ref{ass:U}, is of independent interest, especially  in cases where we do not have convergence. In presence of oscillations for \eqref{eq:lambdaAB_lin} (see Section~\ref{sec:oscillations} for further details), \eqref{eq:ellsBinU} provides lower and upper bounds on the size of the limit cycle. A first consequence of \eqref{eq:ellsBinU} is that whenever $ \lambda_{ B}(t)$ converges as $t\to \infty$, $\underline{\ell}_{ B}=\bar{\ell}_{ B}$ so that (recall the definition \eqref{eq:def_UPhi} of $ \mathcal{ U}_{ \Phi}$) necessarily $\underline{\ell}_{ B}=\bar{\ell}_{ B}=\ell$. This, together with the previous remark that $ \underline{ \ell}_{ B} \notin \mathcal{ I}_{ \kappa_{ 1}, \kappa_{ 2}}$ would lead to oscillations, tells us that Theorem~\ref{th:full_conv} gives a complete picture of the possible limit for $(\lambda_{ t}^{ A}, \lambda_{ t}^{ B})$: if one among $ (\lambda_{ t}^{ A}, \lambda_{ t}^{ B})$ converges as $t\to \infty$, both do, and their limit are necessarily given by $\ell_{ A}= \Psi_{ 1}(\ell_{ B})$ and $\ell_{ B}=\ell$, unique fixed-point of $ \Phi$ in \eqref{eq:Phi}. 

Furthermore, we understand from \eqref{eq:ellsBinU} the point of Assumption~\ref{ass:U}: it gives immediately the convergence of $ \lambda^{ B}$. In this respect, one may find that Assumption~\ref{ass:U} is quite strong: but it is trivially verified in the uncoupled cases $ \kappa_{ 2} \kappa_{ 4}=0$, as the function $ \Phi$ is then constant, and it is satisfied for several nontrivial examples in the fully coupled case (see Section~\ref{sec:examples} below). However, we have numerical evidence (see Figure~\ref{fig:lambdaA_exp_p1}) that Assumption~\ref{ass:U} is in fact not necessary for the convergence of $( \lambda^{ A}, \lambda^{ B})$. In this context, finding a general necessary and sufficient condition for convergence of $(\lambda^{ A}, \lambda^{ B})$ seems to be a tricky question.

\begin{remark}
\label{rem:uniform_bound_lA}
From Theorem~\ref{th:A_subcritical}: under hypotheses of full connectivity, population $A$ cannot be supercritical (in the sense of $\liminf_{ t\to\infty} \lambda_{ t}^{ A}< \infty$), regardless the parameters of the model. Theorem~\ref{th:full_conv} provides a stronger version of this statement:  $\sup_{ t\geq 0} \lambda_{ t}^{ A}<\infty$, although we do not provide an explicit control on $\sup_{ t\geq 0} \lambda_{ t}^{ A}$ (explicit estimates are only on $\limsup_{ t\to\infty} \lambda^{ A}_{ t}$). Recall however what we have said in Remark~\ref{rem:maj_cst_conv}: using  $ \Phi_{ B\to A}\leq 1$, population $A$ is  bounded by a linear Hawkes process with memory kernel $ \alpha h_{ 1}$. Therefore, in the particular case of $ \kappa_{ 1}<1$ (which is in particular stricter than \eqref{hyp:A_sub_muB_kappa3}), we have directly that $ \sup_{ t\geq 0}\lambda^{ A}_{ t} \leq \frac{ \mu_{ A}}{ 1- \kappa_{ 1}}$ (see Lemma 23 of \cite{MR3449317}). Deriving similarly an explicit bound in the general case (for example under \eqref{hyp:A_sub_muB_kappa3} only) is unclear.
\end{remark}

\section{Fluctuations results and a test of inhibition}
\label{sec:fluctuations}
In this section, we are interested in fluctuation results corresponding to the convergence given in Proposition~\ref{prop:conv} and the derivation of corresponding statistical tests for inhibition.
\subsection{Hierarchy of intensities under inhibition and/or retroaction}
\label{sec:fluct_hierarchy}
We restrict to the framework of Model \eqref{eq:lambdaAB_lin} satisfying Assumptions~\ref{ass:main_h},~\ref{ass:model_intro} and~\ref{ass:Phis}. Suppose also that Assumption~\ref{ass:U} as well as \eqref{hyp:A_sub_muB_kappa3} are satisfied. 

The previous analysis shows that  under these hypotheses convergence of $ \lambda_{ t}^{ A}$ as $t\to \infty$ holds in all cases (i.e. both when $ \kappa_{ 2} \kappa_{ 4}=0 $ and when $ \kappa_{ 2} \kappa_{ 4}>0$). 
Recall that Assumption~\ref{ass:U} holds trivially in the decoupled cases $ \kappa_{ 2} \kappa_{ 4}=0$ since $ \Phi$ defined in \eqref{eq:Phi} is then constant (and in such cases, the conditions $\kappa_{3}<1$ and \eqref{hyp:A_sub_muB_kappa3} for subcriticality are sharp).

\bigskip

The first main observation of this paragraph is that there is a natural ordering between the different limits $ \lim_{ t\to\infty} \lambda_{ t}^{ A} $, depending on the presence or absence of inhibition and retroaction: when $A$ is isolated ($\kappa_{2}=0$), we obtain from Proposition~\ref{prop:k2_0_k4_0} that
\begin{equation*}
\lambda_{  t}^{A} \xrightarrow[ t\to\infty]{} \frac{ \mu_{ A}}{ 1- \kappa_{ 1}}:= \ell_{ A}^{ \text{(I)}}.
\end{equation*}
When  inhibition from $B\to A$ is added, without feedback from $A\to B$ ($ \kappa_{ 2}>0, \kappa_{ 4}=0$), the conclusion of Proposition~\ref{prop:k2_pos_k4_0} is
\begin{equation*}
\lambda_{  t}^{A} \xrightarrow[ t\to\infty]{} \frac{ \mu_{ A}\Phi_{ B\to A} \left(\frac{\kappa_{2}\mu_{B}}{1-\kappa_{3}}\right)}{1- \kappa_{ 1}\Phi_{ B\to A} \left(\frac{\kappa_{2}\mu_{B}}{1-\kappa_{3}}\right)} := \ell_{ A}^{ \text{(II)}}.\end{equation*}Finally, in the fully coupled case $ \kappa_{ 2} \kappa_{ 4}>0$, one obtains from Theorem~\ref{th:full_conv} that
$$\lambda_{  t}^{A} \xrightarrow[ t\to\infty]{}\frac{ \mu_{ A}\Phi_{ B\to A} \left(\kappa_{2}\ell\right)}{1- \kappa_{ 1}\Phi_{ B\to A} \left( \kappa_{2}\ell\right)} := \ell_{ A}^{ \text{(III)}},$$
for $ \ell$ the unique fixed-point to the function $\Phi$ given in \eqref{eq:Phi}. Using that $\Phi_{B\to A}$ is non increasing and bounded by $1$ together with \eqref{eq:lower_bound_ell}, for $ \kappa_{ 4}>0$, the following natural order between the previous limits holds true
\begin{equation}\label{eq:lim_lbdA}
\ell_{ A}^{ \text{(I)}} > \ell_{ A}^{ \text{(II)}} > \ell_{ A}^{ \text{(III)}}.
\end{equation} 
Note that one has $ \ell_{ A}^{ \text{(III)}}= \ell_{ A}^{ \text{(III)}}(\kappa_{ 1}, \kappa_{ 2}, \kappa_{ 3}, \kappa_{ 4}) \xrightarrow[ \kappa_{ 4}\to 0]{} \ell_{ A}^{ \text{(II)}}$. 

\subsection{Central limit results}
\label{sec:CLT}
We take advantage of the strict ordering observed in \eqref{eq:lim_lbdA} in order to experimentally test the presence/absence of inhibition in a neuronal network modeled by \eqref{eq:lambdaAB_lin}. For \eqref{eq:lim_lbdA} to be useful, the convergence results seen before must be accompanied with fluctuation results. Fluctuation theorems for Hawkes processes (both in large population $N\to\infty$ and in large time $t\to \infty$) have already been considered in the literature \cite{MR3449317,DITLEVSEN20171840} and some of the existing results  directly apply to our case: for instance, in the setting of Example \ref{ex:linear} with $\kappa_{1}<1$, $\kappa_{3}<1$ and $\kappa_{2}=0$ (resp. $\kappa_{4}=0$), Theorem 10 of  \cite{MR3449317} provides Gaussian limits for the quantities $ \left(\sqrt{m_{t}^{A}}\left(\frac{Z^{i,A}}{m_{t}^{A}}-1\right)\right)_{i\in  I^{A}}$ (resp. $\big(\sqrt{m_{t}^{B}}(\frac{Z^{i,B}}{m_{t}^{B}}-1)\big)_{ i\in I^{B}}$), when $(t,N)\to(\infty,\infty)$ and where $I^{A}$ (resp. $I^{B}$) is a subset of $\{1,\ldots, N_{A}\}$ (resp. $\{N_{A}+1\ldots, N\}$)  and $(m^{A},m^{B})$ is the solution of \eqref{eq:mtAB}. Note nonetheless that the proof in \cite{MR3449317} appears to be specific to the case of linear Hawkes processes so that its direct application to our model with full connectivity is doubtful. 

An existing result more useful to our concern may be found in \cite{DITLEVSEN20171840}, Th.~2, where a central limit theorem for $K$ macroscopic populations of interacting nonlinear Hawkes processes is proven. In particular, a closer look to the proof of Theorem 2 of \cite{DITLEVSEN20171840} permits to realize that it does not rely on the specific shape of the intensities of the underlying Hawkes processes but only on the existence of a propagation of chaos estimates similar to Proposition~\ref{prop:conv}. This result can easily transpose to our setting: 
\begin{proposition}[Ditlevsen, L\"ocherbach, \cite{DITLEVSEN20171840}, Th.~2]
Consider Model \eqref{eq:lambdaAB_lin} satisfying Assumptions~\ref{ass:main_h}, \ref{ass:model_intro}, \ref{ass:Phis} and~\ref{ass:U}, either in the uncoupled cases $\kappa_{2}\kappa_{4}=0$ or with full coupling $\kappa_{2}\kappa_{4}>0$. Require also that 
\begin{equation}
\label{eq:cond_CLT}
\max\{(1+\kappa_{1}\frac{\mu_{A}}{1-\kappa_{1}})(\kappa_{1}+\kappa_{2}),\ (\kappa_{3}+\kappa_{4})\}<1.
\end{equation} 
Then,
for any subsets  $I^{A}_{k_{A}} \subset \{1,\ldots, N_{A}\}$ and $I^{B}_{k_{B}} \subset \{N_{A}+1,\ldots, N\}$ of fixed cardinality $k_{A}$ and $k_{B}$ (independent of $N$) 
\begin{equation}
\label{eq:TCLAB} \left(\sqrt{m_{t}^{A}}\left(\frac{Z_{t}^{i,A}}{m_{t}^{A}}-1\right)\ ,\ \sqrt{m_{t}^{B}}\left(\frac{Z_{t}^{j,B}}{m_{t}^{B}}-1\right)\right)_{i\in I^{A}_{k_{A}},\ j\in  I^{B}_{k_{B}}}\xrightarrow[\overset{(t , N ) \to (\infty,\infty),}{\frac tN\to0}]{d}\mathcal N(0,I_{k_{A}+k_{B}}).
\end{equation}
\end{proposition}
\begin{proof}[Sketch of proof]
 This is a straightforward consequence of \cite{DITLEVSEN20171840}, Th~2: it suffices to note that in all cases,  $m^{A}_{t}/t$ and $m^{B}_{t}/t$ are bounded away from 0 and $\infty$ (recall Propositions~\ref{prop:k2_0_k4_0} and~\ref{prop:k2_pos_k4_0} and Theorem~\ref{th:full_conv}). Condition \eqref{eq:cond_CLT}  ensures that the constant appearing in Proposition \ref{prop:conv} does not increase faster than $C t$, for a positive constant $C$ (see Remark \ref{rem:maj_cst_conv}). Putting everything together one has simply to follow the lines of \cite{DITLEVSEN20171840}, Th~2 to conclude.
\end{proof}

\subsubsection{Discussion on the statistical significance of the CLT}
\label{sec:discussion_CLT}
Here we discuss the practical significance of \eqref{eq:TCLAB}. Denote by $\ell$ the limit intensity of either po\-pu\-lation $A$ or $B$ appearing in the convergence results of Propositions~\ref{prop:k2_0_k4_0},~\ref{prop:k2_pos_k4_0} and Theorem~\ref{th:full_conv}. An estimator $\widehat\ell_{N,T}$ for this quantity may be given through the following approximations:
\[ \ell\underset{T\to \infty}\approx {\lambda_{T}}\underset{{N\to\infty}}\approx\frac{Z_{t}^{i}}{T}:=\widehat\ell_{N,T},\quad 1\le i\le N.\] Therefore, $\widehat\ell_{N,T}$ can be computed from the observation of one neuron (recall that in the mean field limit, neurons within the same population are interchangeable) in a $N$-particle system of Hawkes processes observed in a large time $T$. From a statistical viewpoint, a statement as in \eqref{eq:TCLAB} cannot be exploited directly as the unknown pa\-ra\-meter $\ell$ does not explicitly appear and secondly the normalization of $Z^{i}_{t}$ by $m_{t}$ is unknown. Replacing $\frac{Z^{i}_{t}}{m_{t}}$ by  $\frac{Z^{i}_{t}}{\widehat{m_{t}}}$ where $\widehat m_{t}=Z^{i'}_{t}$ for $i\ne i'$ for instance or $\widehat m_{t}=2(Z^{i}_{t}-Z^{i}_{\frac{t}{2}})$ is not satisfactory as information on $\ell$ will not emerge.
Instead, what is required here (for instance to derive testing procedures and confidence intervals) is a result of the form 
$\sqrt{T}\left(\widehat \ell_{N,T}- \ell\right)\xrightarrow[(T,N)\to(\infty,\infty)]{d}K$
 for some distribution $K$. 
 
The simple case of an isolated subcritical population $A$ with linear intensity (Model~\eqref{eq:lambdaAB_lin} and Proposition~\ref{prop:k2_0_k4_0}, with $ \kappa_{ 2}=0$) is of particular interest in order to build the test considered in the next paragraph. In this case, $\ell=\frac{\mu_{A}}{1- \kappa_{ 1}}$, for $\kappa_{1}= \alpha\|h_{1}\|_{1}<1$ and  the following decomposition holds \begin{align*}
\sqrt{T}\left(\widehat \ell_{N,T}-\ell\right)&=\sqrt{\frac{m_{T}}T}\sqrt{m_{T}}\left(\frac{Z^{i}_{T}}{m_{T}}-1\right)
 +\sqrt{T} \left(\frac{m_{T}}{T}-\frac{\mu_{A}}{1- \kappa_{ 1}}\right):= \sqrt{\frac{m_{T}}T}I_{N,T}(1)+I_{T}(2).
\end{align*} 
From \eqref{eq:TCLAB} (see also Theorem 10 of \cite{MR3449317}), it immediately follows that $I_{N,T}(1)\xrightarrow[]{d}\mathcal{N}(0,1)$ as $(T,N)\to(\infty,\infty)$, $ \frac{ T}{ N} \to 0$. Controlling the deterministic term $I_{2}(T)$ requires to evaluate the rate at which $m_{T}/T$ converges to its limit $\mu_{A}/(1- \kappa_{ 1}).$ This has been studied, in the linear case, in Lemma 5 of \cite{MR3054533} which implies that $I_{T}(2)=o(1)$ if $\int_{0}^{\infty}\sqrt{t}h(t)\rmd t<\infty$. Then, assuming $\int_{0}^{\infty}\sqrt{t}h(t)\rmd t<\infty$ and using that Proposition~\ref{prop:k2_0_k4_0} implies that $\frac{m_{T}}{T}\xrightarrow[T\to\infty]{}\ell$, we derive 
\begin{align}\label{eq:TCLstat}
\sqrt{T}\left(\widehat \ell_{N,T}-\ell\right)\xrightarrow[{(T,N)\to(\infty,\infty)}\atop{\frac{ T}{ N} \to 0}]{d}\mathcal{N}\left(0,\ell\right)\quad \mbox{and}\quad \sqrt{\frac{T}{\widehat \ell_{N,T}}}\left(\widehat \ell_{N,T}-\ell\right)\xrightarrow[(T,N)\to(\infty,\infty)\atop{\frac{ T}{ N} \to 0}]{d}\mathcal{N}\left(0,1\right).
\end{align}

\subsubsection{A test of inhibition} 
\label{sec:test_CLT}
Neurophysiologists studying synaptic function routinely isolate one type of synaptic coupling with specific compounds (very often toxins synthesized by plants, bacteria, spiders, etc. \cite{luo:15}). Using picrotoxin or bicuculline (two plant toxins) they can for instance suppress inhibition, that is, force our kernel $h_2$ to become uniformly zero. Based on  \eqref{eq:TCLstat}, we can then propose the following procedure in order to test if inhibition on a population $A$ is present. In a experimental preparation, neurons from the $A$ population can be continuously recorded in three successive conditions: control (no toxin applied); inhibition blocked (with picrotoxin or bicuculline); 'wash' (the physiologists' term to designate a return to the control condition after removal of a toxin). Collecting for a duration  $T$ the activities of the neurons of population $A$ and estimate their intensities with $\widehat\ell^{Control}_{N,T}$, $\widehat\ell^{Toxin}_{N,T}$ and $\widehat \ell^{Wash}_{N,T}$ respectively.

Denote by $(\ell^{{Control}},\ell^{Toxin})$ the \emph{true} intensities of population $A$ with and without inhibition blocked. If population $B$ has no effect onto population $A$, one should have $\ell^{Control}=\ell^{Toxin}$, otherwise following \eqref{eq:lim_lbdA} one should observe $\ell^{Control}>\ell^{Toxin}$. This suggests to consider the following testing procedure:
\begin{align*}
\begin{cases}
H_{0}\ :\ \mbox{Population $B$ has no inhibitive effect on population $A$},\\
H_{1}\ :\ \mbox{Population $B$ has an inhibitive effect on population $A$},
\end{cases}
\end{align*} or equivalently
\begin{align}
\label{eq:test}\begin{cases}
H_{0}\ :\ \ell^{Control}=\ell^{Toxin},\\
H_{1}\ :\ \ell^{Control}>\ell^{Toxin}.
\end{cases}
\end{align} Consider the test statistics $R_{N,T}=\widehat\ell^{Control}_{N,T}-\widehat \ell ^{Toxin}_{N,T},$ where $\widehat\ell^{Control}_{N,T}$ and $\widehat \ell ^{Toxin}_{N,T}$ are independent. Using \eqref{eq:TCLstat} and that under $H_{0}$ it holds $\ell^{Control}=\ell^{Toxin}$, we write:
\begin{align*}
\sqrt{T}(\widehat\ell^{Control}_{N,T}-\widehat \ell ^{Toxin}_{N,T})&=\sqrt{T}(\widehat\ell^{Control}_{N,T}- \ell ^{Control})-\sqrt{T}(\widehat \ell ^{Toxin}_{N,T}-\ell^{Toxin})\xrightarrow[(T,N)\to(\infty,\infty)\atop{\frac{ T}{ N} \to 0}]{d,\ H_{0}}\mathcal{N}(0,2\ell),
\end{align*}  from which we derive
\begin{align*}
\sqrt{\frac{T}{\widehat\ell^{Control}_{N,T}+\widehat \ell ^{Toxin}_{N,T}}}(\widehat\ell^{Control}_{N,T}-\widehat \ell ^{Toxin}_{N,T})&\xrightarrow[(T,N)\to(\infty,\infty)\atop{\frac{ T}{ N} \to 0}]{d,\ H_{0}}\mathcal{N}(0,1).
\end{align*}
Therefore, we get a test for the set of assumptions \eqref{eq:test} with asymptotic level $a\in(0,1)$ $$\Psi^{(a)}_{N,T}=\mathbf{1}_{\left\{\widehat\ell^{Control}_{N,T}-\widehat \ell ^{Toxin}_{N,T}\ge s_{a} \right\}}\quad \mbox{where}\quad s_{a}=\sqrt{\frac{\widehat\ell^{Control}_{N,T}+\widehat \ell ^{Toxin}_{N,T}}T}q^{1-a}_{\mathcal{N}(0,1)},$$ where $q^{r}_{\mathcal{N}(0,1)}$ the quantile of order $r$ of the distribution $\mathcal{N}(0,1)$.
Note that under $H_{1}$ the same decomposition holds if one adds  the additional diverging term $\sqrt{T}(\ell^{Control}-\ell^{Toxin})$ that ensures the consistency of the test $\Psi^{(a)}_{N,T}$.

\section{Examples in the fully-coupled subcritical linear case}
\label{sec:examples}
We place ourselves here in the context of Theorem~\ref{th:full_conv} concerning Model \eqref{eq:lambdaAB_lin}. Recall in particular that we assume full connectivity $\kappa_{2} \kappa_{ 4}>0$ with population $B$ subcritical: $ \kappa_{ 3}<1$. The aim of this paragraph is twofold: first, to give general sufficient conditions for Assumption~\ref{ass:U} used in Theorem~\ref{th:full_conv} and second, to give instances of kernels $ \Phi_{ A\to B}$ and $ \Phi_{ B\to A}$ that fulfil the hypotheses of Theorem~\ref{th:full_conv} (and this for which the convergence result described in Theorem~\ref{th:full_conv} is valid). Focus will be made on particular choices of  inhibition kernels $ \Phi_{ B\to A}$: in all the examples  below, we have simply chosen $\Phi_{ A\to B}(x)=x$. To simplify notations in the sequel, define the following parameters
\begin{equation}
\label{eq:ab}
a:=\frac{ \mu_{ B} \kappa_{ 2}}{ 1- \kappa_{ 3}}\quad \mbox{and}\quad b:= \frac{\kappa_{ 2} \kappa_{ 4}}{ 1- \kappa_{ 3}}.
\end{equation} 
Informally, the parameter $a$ measures the influence of population $B$ onto population $A$ whereas $b$ assesses the impact of the coupling between both populations. 
\subsection{Verifying Assumption~\ref{ass:U}}
\label{sec:verify_assU}
First, we provide general sufficient conditions for Assumption~\ref{ass:U}. Recall the definitions of $ \Phi$ and $ \mathcal{ U}_{ \Phi}$ in \eqref{eq:Phi} and \eqref{eq:def_UPhi}. 
\begin{proposition}
\label{prop:verify_U}
Suppose that Assumption~\ref{ass:Phis} is satisfied together with $\kappa_{3}<1$. Then, 
\begin{enumerate}
\item General sufficient conditions: Assumption~\ref{ass:U} is true under one of the following conditions:
\begin{enumerate}
\item \label{it:contraction}$ \Phi$ is a contraction on $ \mathcal{ I}_{ \kappa_{ 1}, \kappa_{ 2}} \cap \left[ \frac{ \mu_{ B}}{ 1- \kappa_{ 3}}, +\infty\right)$: \begin{equation}
\label{eq:Psi_contraction}
\left\vert \Phi(x) - \Phi(y)\right\vert < \left\vert x-y \right\vert,\quad\quad x\neq y.
\end{equation}
\item \label{it:Phi2} $ \Phi^{ 2}:= \Phi \circ \Phi$ has a unique fixed-point in $ \mathcal{ I}_{ \kappa_{ 1}, \kappa_{ 2}} \cap \left[ \frac{ \mu_{ B}}{ 1- \kappa_{ 3}}, +\infty\right)$.
\end{enumerate}
\item Suppose now that $ \Phi$ is strictly decreasing (that is the case when e.g. $ \Phi_{ A\to B}$ and $ \Phi_{ B\to A}$ are strictly increasing and decreasing respectively). Then Assumption~\ref{ass:U} is satisfied if and only if
\begin{align}
\label{eq:def_UPhi3}
u< \Phi\circ \Phi(u), \text{ on }  \left(\max \left(x_{ \kappa_{ 1}, \kappa_{ 2}}^{ \ast}, \frac{ \mu_{ B}}{ 1- \kappa_{ 3}} \right), \ell\right).
\end{align}

\end{enumerate}

\end{proposition}
\begin{remark}\label{rmk:PhiPhi}
Condition \eqref{eq:def_UPhi3} is true in particular if $ \Phi\circ \Phi$ is strictly concave on  $ \left(\max \left(x_{ \kappa_{ 1}, \kappa_{ 2}}^{ \ast}, \frac{ \mu_{ B}}{ 1- \kappa_{ 3}} \right), \ell\right)$ or if $ \Phi$ is differentiable in $\ell$, $ \left\vert \Phi^{ \prime}(\ell) \right\vert \leq1$ and $ \Phi^{ 2}$ is strictly convex on $ \left(\max \left(x_{ \kappa_{ 1}, \kappa_{ 2}}^{ \ast}, \frac{ \mu_{ B}}{ 1- \kappa_{ 3}} \right), \ell\right)$.\end{remark}

\begin{proof}[Proof of Proposition~\ref{prop:verify_U}]
The case~\ref{it:contraction}. is a straightforward consequence of the definition of $ \mathcal{ U}_{ \Phi}$ in \eqref{eq:def_UPhi} together with \eqref{eq:Psi_contraction}. Turn now to case~\ref{it:Phi2}.: suppose that $\Phi$ is such that $ \Phi^{ 2}$ has a unique fixed-point, which is necessarily $\ell$. From the definition \eqref{eq:def_UPhi} of $ \mathcal{ U}_{ \Phi}$ and the fact that $ \Phi$ is nonincreasing it follows that, for any $(u, v)\in \mathcal{ U}_{ \Phi}$,
\begin{equation*}
\Phi^{ 2}(u)\leq \Phi \left(v\right) \leq u \leq \ell \leq v \leq \Phi \left(u\right) \leq \Phi^{ 2}(v).
\end{equation*}
An immediate recursion from the latter inequality shows that $ \left( \Phi^{2n}(u)\right)_{ n\geq0}$ is (positive) non-increasing and hence converges to a fixed-point of $ \Phi^{ 2}$, which by uniqueness is $\ell$. Thus, $ \ell \leq \Phi^{ 2n}(u) \leq u \leq \ell$, which gives in particular $ u=\ell$ and since $ \ell \leq v \leq \Phi( u)= \Phi \left(\ell\right)=\ell$, this gives the result. Turn now to the proof when $ \Phi$ strictly decreases:  the domain $ \mathcal{ U}_{ \Phi}$ boils down to
\begin{align}
\mathcal{ U}_{ \Phi}&:=  \left\lbrace \max \left(x_{ \kappa_{ 1}, \kappa_{ 2}}^{ \ast}, \frac{ \mu_{ B}}{ 1- \kappa_{ 3}} \right)\leq u \leq \ell,\  \Phi^{ -1}(u) \leq v \leq \Phi(u)\right\rbrace.\label{eq:def_UPhi2}
\end{align}
So that Assumption~\ref{ass:U} is satisfied if and only if $ \Phi^{ -1}(u)> \Phi(u)$ on $ \left(\max \left(x_{ \kappa_{ 1}, \kappa_{ 2}}^{ \ast}, \frac{ \mu_{ B}}{ 1- \kappa_{ 3}} \right), \ell\right)$, which gives the result.
\end{proof}

\subsection{Polynomial inhibition}
\label{sec:pol_inhib}
Let $ \beta, \tau>0$ and consider in this paragraph
\begin{equation}
\label{eq:PhiBtoA_pol}
\Phi_{ B\to A}(x)= \frac{ 1}{ 1+ \tau x^{ \beta}}\quad \mbox{and}\quad \Phi_{ A\to B}(x)=x,\quad x\geq0.
\end{equation}
Thenn the function $ \Phi$ writes
\begin{equation}
\label{eq:Phi_pol}
\Phi(x)= \frac{ a}{ \kappa_{ 2}} + \frac{ \mu_{ A}b}{ \kappa_{ 2} \left(1- \kappa_{ 1}+ \tau \left(\kappa_{ 2}x\right)^{ \beta}\right)},\quad x\in  \mathcal{ I}_{ \kappa_{ 1}, \kappa_{ 2}},
\end{equation}
where the domain $ \mathcal{ I}_{ \kappa_{ 1}, \kappa_{ 2}}$ is given by
\begin{equation*}
\mathcal{ I}_{ \kappa_{ 1}, \kappa_{ 2}}= \begin{cases}
[0, +\infty) & \text{ if } \kappa_{ 1}<1,\\
\left( x_{ \kappa_{ 1}, \kappa_{ 2}}^{ \ast}, +\infty\right) = \left( \frac{ 1}{ \kappa_{ 2}} \left(\frac{ \kappa_{ 1}-1}{ \tau}\right)^{ \frac{ 1}{ \beta}}, +\infty\right) & \text{ if } \kappa_{ 1}\geq 1.
\end{cases}
\end{equation*}
\begin{proposition}
\label{prop:conv_pol}
Consider Model \eqref{eq:lambdaAB_lin} with $ \Phi_{ B\to A}$ and $ \Phi_{ A\to B}$ given by \eqref{eq:PhiBtoA_pol}, satisfying $ \kappa_{ 3}<1$, $\kappa_{ 2} \kappa_{ 4}>0$, $ \mu_{ A}>0$, $ \kappa_{ 1}< 1 + \tau a^{ \beta} $ and Assumptions~\ref{ass:main_h} and \eqref{hyp:h_to_0}. Then,   in the following cases:
\begin{enumerate}
\item when $ \beta\in (0, 1]$, 
\item when $ \beta>1$:
\begin{enumerate}
\item if $ \kappa_{ 1}<1$, at least when $ \tau$ is sufficiently small (weak inhibition) or when $ \tau$ is sufficiently large (strong inhibition).
\item if $ \kappa_{ 1}\geq 1$, at least when $ \tau$ is sufficiently large (strong inhibition),
\end{enumerate}
\end{enumerate} the conclusions of Theorem~\ref{th:full_conv} are valid, i.e. $(\lambda^{ A}_{ t}, \lambda^{ B}_{ t}) \xrightarrow[ t\to\infty]{} \left(\ell_{ A}, \ell_{ B}\right):=\left( \frac{ \mu_{ A}}{ 1- \kappa_{ 1}+ \tau \left(\kappa_{ 2} \ell\right)^{ \beta} }, \ell\right)$, where $\ell$ is the unique fixed-point of $ \Phi$ given by \eqref{eq:Phi_pol} on $\left[ \frac{ \mu_{ B}}{ 1- \kappa_{ 3}}, +\infty\right)$.
In the case $ \beta=1$, $\ell_{ B}=\ell$ can be explicitly computed as 
\begin{equation}
\label{eq:theoretical_l_beta1}
\ell= \frac{ -(1- \kappa_{ 1} - \tau a) + \left( \left(1- \kappa_{ 1} - \tau a\right)^{ 2} + 4 \tau \left((1- \kappa_{ 1})a + \mu_{ A}b\right)\right)^{ \frac{ 1}{ 2}}}{ 2 \tau \kappa_{ 2}}.
\end{equation}
\end{proposition}
\begin{remark}
Note that $ \beta\in [0, 1]$ corresponds to the case where $ \Phi_{ B\to A}$ (and hence $ \Phi$ itself) is convex. We stress the fact that, from a numerical point of view, convergence is actually valid beyond the hypotheses of Proposition~\ref{prop:conv_pol}, especially in cases when the hypothesis \eqref{hyp:A_sub_muB_kappa3} $ \kappa_{ 1}< 1+ \tau a^{ \beta}$ is no longer valid, see Figure~\ref{fig:lambdaA_pol_CondlB}.

The case $ \beta>1$ seems to reveal richer dynamical patterns, especially in the regime where $ \beta\gg 1$, where \eqref{eq:PhiBtoA_pol} is of sigmoid type and approaches (for fixed $ \tau$) the indicator function $ \mathbf{ 1}_{ [0,1]}$ as $ \beta\to \infty$, see Section~\ref{sec:oscillations} for further details. 
\end{remark}

\begin{proof}[Proof of Proposition~\ref{prop:conv_pol}]Note that the subcriticality condition \eqref{hyp:A_sub_muB_kappa3} writes here
$\kappa_{ 1}< 1+ \tau a^{ \beta}$, which is equivalent to $ \frac{ \mu_{ B}}{ 1- \kappa_{ 3}} \geq x^{ \ast}_{ \kappa_{ 1}, \kappa_{ 2}}$ and hence \eqref{eq:uB_in_I} is true. Suppose first $ \beta\in (0, 1]$ and define $ \rho(x):= 1- \kappa_{ 1} + \tau x^{ \beta}$. We show that $\Phi\circ \Phi$ is concave, from which we derive by Proposition \ref{prop:verify_U} point 2 that Assumption \ref{ass:U} is valid (see Remark \ref{rmk:PhiPhi}). Showing that $\Phi\circ\Phi$ is concave is equivalent to showing that $(\Phi\circ\Phi)'$ is decreasing. Simple computations show that $$(\Phi\circ\Phi)'(x)=\mu^{2}\frac{\rho'(\kappa_{2}x)}{\rho^{2}(\kappa_{2}x)}\frac{\rho'\left(a+\frac{\mu_A b}{\rho(\kappa_2 x)}\right)}{\rho^{2}\left(a+\frac{\mu_A b}{\rho(\kappa_2 x)}\right)},$$ therefore $(\Phi\circ\Phi)'$ has the same monotonicity as $F(y)G(y),\ y\ge a,$ where $F(y):=\frac{\rho'(y)}{\rho^{2}(y)}$ is positive and $G(y):=\frac{\rho'\left(a+\frac{\mu}{\rho(y)}\right)}{\rho^{2}\left(a+\frac{\mu}{\rho(y)}\right)}$  is positive, where we set $\mu=\mu_{A}b$. A sufficient condition for $H$ to be decreasing is  that $\frac{F'}{F}+\frac{G'}{G}<0$. First, tedious computations enable to write for $y\ge a$,
\begin{align*}
H(y):&=y\rho(y)\left(\frac{F'(y)}{F(y)}+\frac{G'(y)}{G(y)}\right),\\
&=(\beta-1)\rho(y)-\tau \beta y^{\beta}\left(2+\frac{\mu}{a\rho(y)+\mu}(\beta-1)-2\beta\frac{\mu}{a\rho(y)+\mu}\frac{\tau\left(a+\frac{\mu}{\rho(y)}\right)^{\beta}}{\rho\left(a+\frac{\mu}{\rho(y)}\right)}\right),\\
&= T_{1}(y)+T_{2}(y)
\end{align*} 
  where 
\begin{align*}
T_{1}(y):&=-(1-\beta)\left(\rho(y)+\frac{\beta\mu\tau y^\beta }{a\rho(y)+1}\right),\\
T_{2}(y) :&=-2\tau \beta y^{\beta}\left(1+\frac{\mu}{a\rho(y)+\mu}(\beta-1)-\beta\frac{\mu}{a\rho(y)+\mu}\frac{\tau\left(a+\frac{\mu}{\rho(y)}\right)^{\beta}}{\rho\left(a+\frac{\mu}{\rho(y)}\right)}\right).\end{align*}
As $\beta\le 1$, we immediately have that $T_{1}\le 0.$ To control the second term $T_{2}$, we write it as follows, using the definition of $\rho$,
\begin{align*}
T_{2}(y)&=-2\tau \beta y^{\beta}\frac{\mu}{a\rho(y)+\mu}\left(1+\frac{a\rho(y)}{\mu}+(\beta-1)-\beta\frac{\rho\left(a+\frac{\mu}{\rho(y)}\right)+\kappa_{1}-1}{\rho\left(a+\frac{\mu}{\rho(y)}\right)}\right),\\
&=-2\tau \beta y^{\beta}\frac{\mu}{a\rho(y)+\mu}\left(\frac{a\rho(y)}{\mu}+\beta\frac{1-\kappa_{1}}{\rho\left(a+\frac{\mu}{\rho(y)}\right)}\right),\\ 
&=-2\tau \beta y^{\beta}\frac{\mu}{a\rho(y)+\mu}\frac{1}{\rho\left(a+\frac{\mu}{\rho(y)}\right)}\left(\frac{a\rho(y)}{\mu}{\rho\left(a+\frac{\mu}{\rho(y)}\right)}+\beta(1-{\kappa_{1}})\right).
\end{align*}
Note that $T_{2}\le 0$ holds true whenever $\kappa_{1}\le 1$ but requires an additional justification  if $\kappa_{1}\in]1,1+\tau a^{\beta}[$. For that, we study the function $T_{3}(y):=\frac{a\rho(y)}{\mu}\rho\left(a+\frac{\mu}{\rho(y)}\right)+\beta
(1-\kappa_{1}),\ y\ge a$, which, as $\rho$ is increasing and $\mu\ge0$, has the same monotonicity as the function  $g:z\mapsto z\rho\left(a+\frac{1}{z}\right)$, $z>0$. Using that $\kappa_{1}< 1+\tau a^{\beta}$ and $\beta\le 1$, it holds that
 \begin{align*}
 g'(z)&=\frac{1}{z^{\beta}}\left(z^{\beta}(1-\kappa_{1})+\tau(1+az)^{\beta-1}(1-\beta+az)\right)
 \ge \frac{\tau}{z^{\beta}}\left(-(az)^{\beta}+(1+az)^{\beta-1}(1-\beta+az)\right)\ge 0.
 \end{align*} It follows that $g$ and $T_{3}$ are increasing and that $T_{3}$ attains its minimum  at $y=a$. We conclude using that  for all $y\ge a$ and $\kappa_{1}< 1+\tau a^{\beta}$
  \begin{align*}T_{3}(y)\ge T_{3}(a)&=\frac{a\rho(a)}{\mu}\rho\left(a+\frac{\mu}{\rho(a)}\right)+\beta
(1-\kappa_{1})\ge  \frac{a\rho(a)}{\mu}\left(-\tau a^{\beta}+\tau\left(a+\frac{\mu}{\rho(a)}\right)^{\beta}\right)-\beta\tau a^{\beta}\\
&=\tau \frac{a\rho(a)^{1-\beta}}{\mu}\left((a\rho(a)+\mu)^{\beta}-({a\rho(a)})^{\beta}-\beta ({a\rho(a)})^{\beta-1}\right)\ge 0
\end{align*} as $z\mapsto (z+\mu)^{\beta}-z^{\beta}-\beta z^{\beta-1}$ is increasing (as $\beta\le 1$ and $\mu\ge 0$) and is positive if $z=0$.
Gathering everything  together, we conclude that $T_{2}$ is negative, therefore $H$ is negative and $\Phi\circ\Phi$ is concave.  Assumption \ref{ass:U} is fulfilled and
 the conclusions of Theorem~\ref{th:full_conv} are valid.

Turn now to the case $ \beta>1$: we have $ \Phi^{ \prime}(x) = - \mu_{ A} b \frac{ \tau \beta (\kappa_{ 2}x)^{ \beta-1}}{ \left(1- \kappa_{ 1} + \tau (\kappa_{ 2} x)^{ \beta}\right)^{ 2}}$, so that for $x\in \left[\frac{ \mu_{ B}}{ 1- \kappa_{ 3}}, \ell\right)$, $ \left\vert \Phi^{ \prime}(x) \right\vert \leq \frac{  \mu_{ A}b \beta \left(\kappa_{ 2} \ell \right)^{ \beta-1} \tau}{ \left(1- \kappa_{ 1} + \tau a^{ \beta}\right)^{ 2}}$. Suppose first that $ \kappa_{ 1}< 1$. Note that since $ \ell \geq \frac{ \mu_{ B}}{ 1- \kappa_{ 3}}$, one obtains from \eqref{eq:Phi_pol} the following a priori bound on $\ell$: $ \ell \leq \frac{ a}{ \kappa_{ 2}} + \frac{ \mu_{ A} b }{ \kappa_{ 2} \left( 1- \kappa_{ 1} + \tau a^{ \beta}\right)}$, quantity that remains bounded as $ \tau\to 0$ and $ \tau\to \infty$. Putting everything together, this yields that $\sup_{ x \in \left[\frac{ \mu_{ B}}{ 1- \kappa_{ 3}}, \ell\right)} \left\vert \Phi^{ \prime}(x) \right\vert$ can be made strictly smaller than $1$ provided $ \tau \to 0$ or $\tau \to \infty$. In these regimes, $ \Phi$ is a contraction and Assumption~\ref{ass:U} holds, by Proposition~\ref{prop:verify_U}, which gives the result. The same argument works in the case $ \kappa_{ 1}\geq 1$, considering now only the limit of strong inhibition $\tau\to\infty$. Computing the fixed-point $\ell$ in the case $ \beta=1$ is basic analysis from \eqref{eq:Phi_pol}.
\end{proof}
\begin{figure}[h]
\centering
\includegraphics[width=\textwidth]{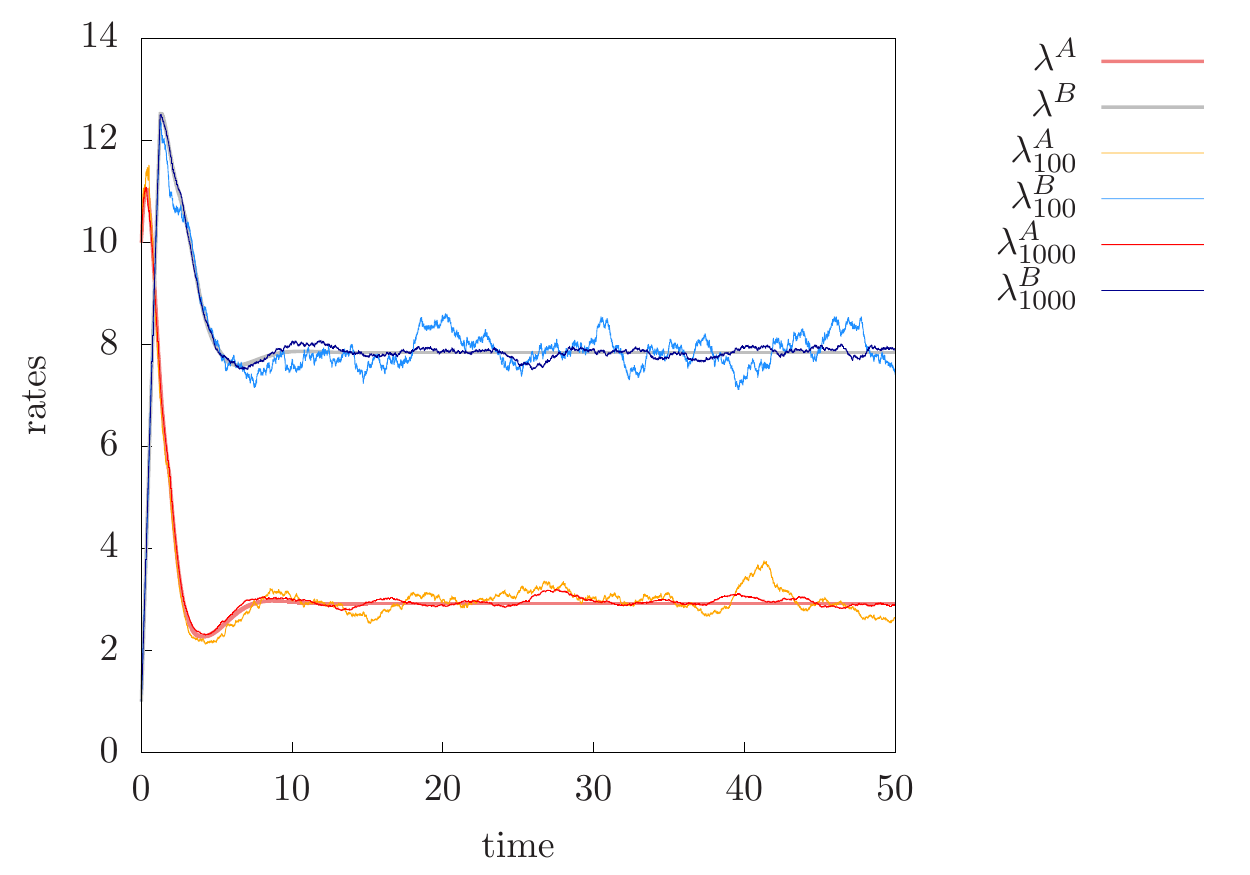}
\caption{Trajectories of the mean-field $ (\lambda^{ A}, \lambda^{ B})$ and their microscopic counterparts $( \lambda_{ N}^{ A}, \lambda_{ N}^{ B})$,  when $h_{ i}= \mathbf{ 1}_{ [0, \theta_{ i}]}$ for the Model \eqref{eq:PhiBtoA_pol} for $ \alpha=0.8$, $ \beta=1$, $ \kappa_{ 1}=1.5$, $ \kappa_{ 2}= 0.5$, $ \kappa_{ 3}=0.5$, $ \kappa_{ 4}=1$, $ \mu_{ A}=10$, $ \mu_{ B}=1$, $ \tau=1$, on $[0, T]$ with $T=50$: the conditions of Proposition~\ref{prop:conv_pol} are satisfied (note in particular that  $ 1<\kappa_{ 1}< 1+ \tau a=2$). The limit of $(\lambda^{ A}, \lambda^{ B})$ match their theoretical counterparts (recall \eqref{eq:theoretical_l_beta1}) $\ell_{ B}\approx 7.844$ and $\ell_{ A}= \Psi_{ 1}(\ell_{ B}) \approx 2.922$.}
\label{fig:lambdaA_pol_thOK}
\end{figure}

\begin{figure}[h]
\centering
\includegraphics[width=\textwidth]{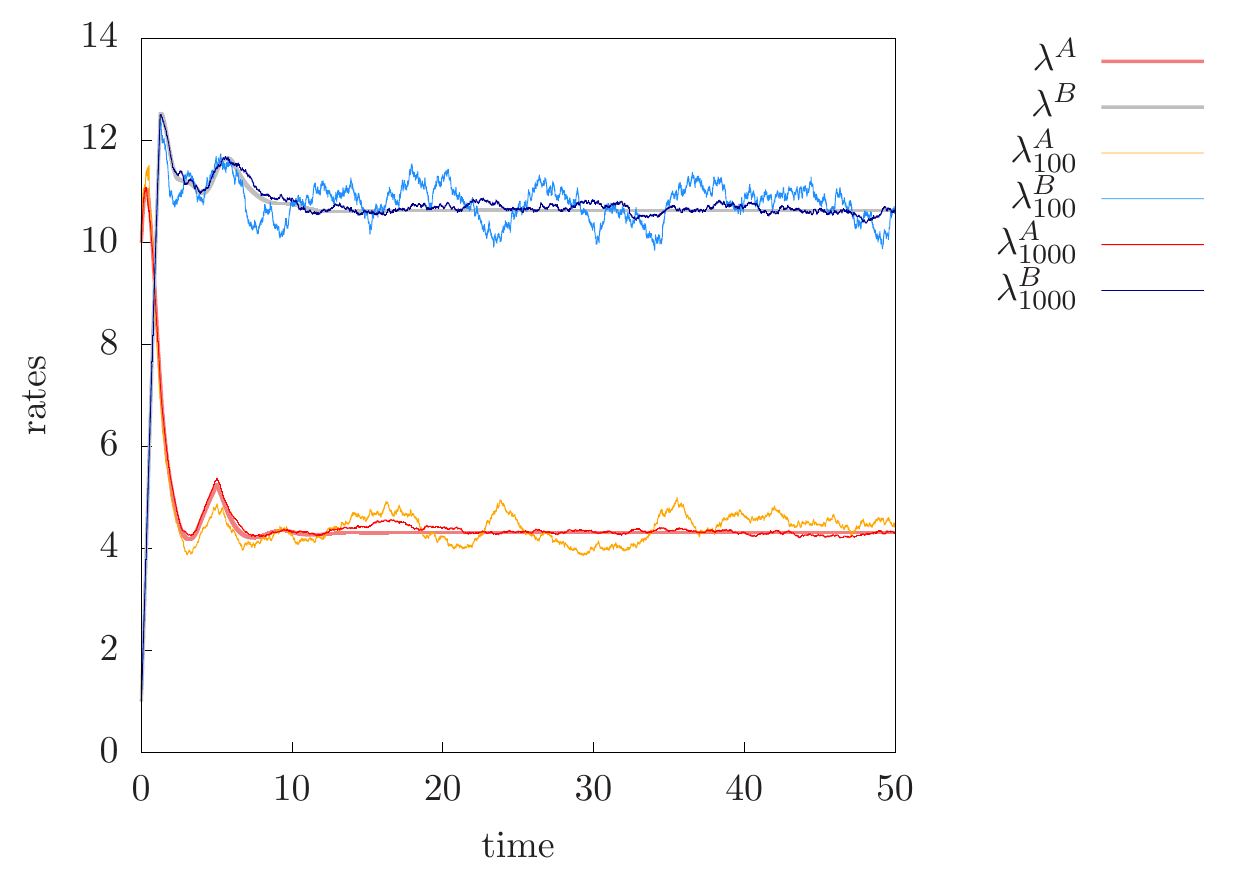}
\caption{Trajectories of the mean-field $ (\lambda^{ A}, \lambda^{ B})$ and their microscopic counterparts $( \lambda_{ N}^{ A}, \lambda_{ N}^{ B})$, when $h_{ i}= \mathbf{ 1}_{ [0, \theta_{ i}]}$ for the Model \eqref{eq:PhiBtoA_pol} for the same parameters as Figure~\ref{fig:lambdaA_pol_thOK}, except for $ \kappa_{ 1}$ where now $ \kappa_{ 1}=4$, so that condition \eqref{hyp:A_sub_muB_kappa3} no longer holds. We still observe convergence to the correct theoretical value (and \eqref{eq:uB_in_I} is nonetheless true).}
\label{fig:lambdaA_pol_CondlB}
\end{figure}
\subsection{Exponential inhibition} 
\label{sec:case_exp}
Consider in this paragraph the case ($ \tau>0$)
\begin{equation}
\label{eq:example_PhiBtoA_exp}
\Phi_{ B\to A} (x)= \exp \left(- \tau x\right)\quad \mbox{and}\quad\Phi_{ A\to B}(x)=x,\quad x\geq 0.
\end{equation}
The function $ \Phi$ rewrites in this case as
\begin{equation}
\label{eq:Phi_exp}
\Phi(x)= \frac{ a}{ \kappa_{ 2}} + \frac{ b\mu_{ A} \exp \left(- \tau \kappa_{ 2}x\right)}{ \kappa_{ 2} \left(1- \kappa_{ 1}\exp \left(- \tau \kappa_{ 2}x\right)\right)}
\end{equation}
defined on
\begin{equation*}
\mathcal{ I}_{ \kappa_{ 1}, \kappa_{ 2}}= \begin{cases}
[0, +\infty) & \text{ if } \kappa_{ 1}<1,\\
\left( \frac{ \ln \left(\kappa_{ 1}\right)}{ \tau \kappa_{ 2}}, +\infty\right) & \text{ if } \kappa_{ 1}>1.
\end{cases}
\end{equation*}
The sufficient subcriticality condition \eqref{hyp:A_sub_muB_kappa3} becomes here
\begin{equation}
\label{hyp:subcriticality_exp}
\kappa_{ 1} e^{ - \tau a}<1.
\end{equation}

\begin{proposition}
\label{prop:example_PhiBtoA_exp_contraction}
Consider Model \eqref{eq:lambdaAB_lin} with $ \Phi_{ B\to A}$ and $ \Phi_{ A\to B}$ given by \eqref{eq:example_PhiBtoA_exp}, satisfying $ \kappa_{ 3}<1$, $\kappa_{ 2} \kappa_{ 4}>0$, $ \mu_{ A}>0$, $ \kappa_{ 1} e^{ - \tau a}<1$ and Assumptions~\ref{ass:main_h} and \eqref{hyp:h_to_0}. Then, the conclusions of Theorem~\ref{th:full_conv} are valid, i.e. $(\lambda^{ A}_{ t}, \lambda^{ B}_{ t}) \xrightarrow[ t\to\infty]{} \left(\ell_{ A}, \ell_{ B}\right):=\left( \frac{ \mu_{ A} e^{ - \tau \kappa_{ 2} \ell}}{ 1- \kappa_{ 1} e^{ - \tau \kappa_{ 2} \ell}}, \ell\right)$, where $\ell$ is the unique fixed-point of $ \Phi$ given by \eqref{eq:Phi_exp} on $\left[ \frac{ \mu_{ B}}{ 1- \kappa_{ 3}}, +\infty\right)$, at least in the following cases:
\begin{enumerate}
\item when $ \kappa_{ 1}<1$, at least when $ \tau>0$ is sufficiently small (weak inhibition) and when $ \tau$ is sufficiently large (strong inhibition),
\item when $ \kappa_{ 1}\geq 1$, at least when $ \tau> \tau_{ \ast}$, for some $ \tau_{ \ast}> \frac{ \ln(\kappa_{ 1})}{ a}>0$ depending on the parameters of the system (strong inhibition).
\end{enumerate}
\end{proposition}
\begin{proof}[Proof of Proposition~\ref{prop:example_PhiBtoA_exp_contraction}]
The derivative of $ \Phi$ on $ \mathcal{ I}_{ \kappa_{ 1}, \kappa_{ 2}} \cap \left[ \frac{ \mu_{ B}}{ 1- \kappa_{ 3}}, +\infty\right)$ writes
$\Phi^{ \prime}(x)= \frac{ -b\mu_{ A} \tau\exp \left(- \tau \kappa_{ 2}x\right)}{ \left(1- \kappa_{ 1}\exp \left(- \tau \kappa_{ 2}x\right)\right)^{ 2}}.$
If $ \kappa_{ 1}<1$, $ \mathcal{ I}_{ \kappa_{ 1}, \kappa_{ 2}}=[0, +\infty)$, the maximum of $x \mapsto \left\vert \Phi^{ \prime}(x) \right\vert$ on $\left[ \frac{ \mu_{ B}}{ 1- \kappa_{ 3}}, +\infty\right)$ is attained for $x= \frac{ \mu_{ B}}{ 1- \kappa_{ 3}}$ and is given by
\begin{equation}
\label{eq:r}
r(\tau):= \frac{b\mu_{ A}\tau \exp(- \tau a)}{ \left(1- \kappa_{ 1}\exp(- \tau a)\right)^{ 2}}
\end{equation}
In particular, $ r(\tau)<1$ as $ \tau\to 0$ and $ \tau\to +\infty$, imply that $ \Phi$ is a contraction, and hence Item 2 of Proposition~\ref{prop:verify_U} is true. More precisely, since $r(\tau)\leq \frac{b\mu_{ A}\tau}{ \left(1- \kappa_{ 1}\right)^{ 2}}$, $r(\tau)<1$ as least when $ \tau< \frac{ \left(1- \kappa_{ 1}\right)^{ 2}}{ b \mu_{ A}}$. Secondly, since $r(\tau) \leq \frac{ 2b \mu_{ A} }{ ea(1- \kappa_{ 1})^{ 2}} \exp \left(- \frac{ \tau a}{ 2}\right)$, $r(\tau)<1$ at least when $ \frac{ 2}{ a}\ln \left(\frac{ 2b \mu_{ A} }{ ea(1- \kappa_{ 1})^{ 2}}\right) <\tau$.
\medskip

\noindent
In the case $ \kappa_{ 1}>1$, suppose that \eqref{hyp:subcriticality_exp} holds, so that $ \frac{ \mu_{ B}}{ 1- \kappa_{ 3}} > \frac{ \ln(\kappa_{ 1})}{ \tau \kappa_{ 2}}=x_{ \kappa_{ 1}, \kappa_{ 2}}^{ \ast}$. Once again, the maximum of $ \left\vert \Phi^{ \prime}(x) \right\vert$ on $ \mathcal{ I}_{ \kappa_{ 1}, \kappa_{ 2}}\cap \left[ \frac{ \mu_{ B}}{ 1- \kappa_{ 3}}, +\infty\right)$ is attained for $x= \frac{ \mu_{ B}}{ 1- \kappa_{ 3}}$ and given by \eqref{eq:r}, which goes to $0$ as $ \tau\to\infty$ and hence $r(\tau)<1$ for some $ \tau_{ \ast} > \frac{ \ln(\kappa_{ 1})}{ a}$.
\end{proof}

\begin{remark}
As already underlined, Assumption~\ref{ass:U} is only sufficient for the convergence of Theorem~\ref{th:full_conv}, not necessary: in the exponential case of Section~\ref{sec:case_exp}, Figure~\ref{fig:lambdaA_exp_p1} shows a situation where $ \Phi\circ \Phi(u) < u$ for $u< \ell$ (and hence $ \mathcal{ U}_{ \Phi}$ is nontrivial). Nonetheless, convergence of $(\lambda^{ A}, \lambda^{ B})$ still occurs. The conjecture one may draw here is that convergence mentioned in Theorem~\ref{th:full_conv} holds as long as $ \Phi$ is convex, although we do not have a proof of this statement.
\end{remark}

\begin{figure}[h]
\centering
\subfloat[Graph of $ \Phi\circ \Phi$, where $ \Phi$ is given by \eqref{eq:Phi_exp}. Here, $\ell\approx 7.567$ and $ \Phi\circ \Phi(u)< u$ for $u\leq \ell$: Assumption~\ref{ass:U} does not hold.]{\includegraphics[width=0.45\textwidth]{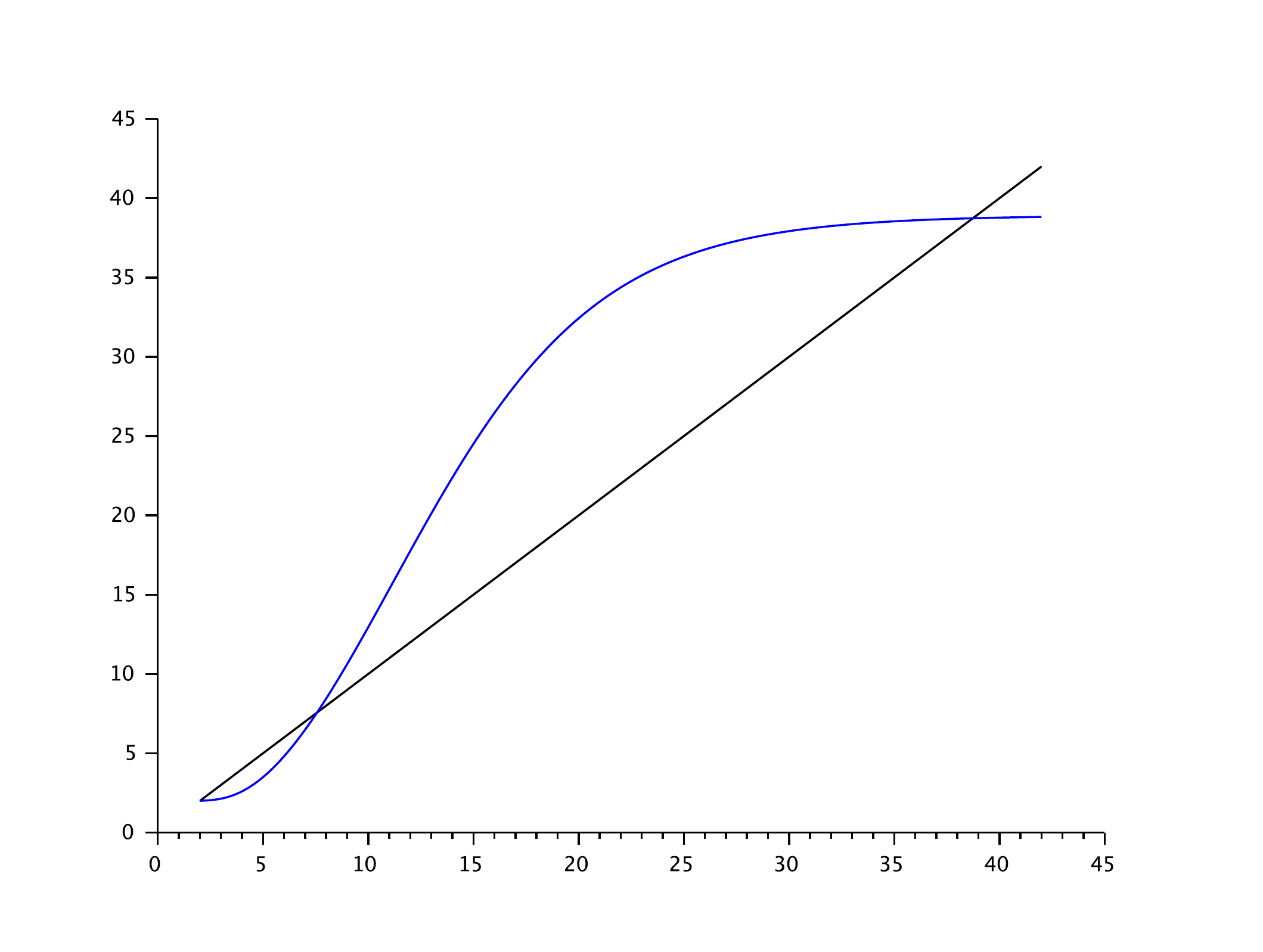}\label{subfig:Psi2_p1}}
\quad\subfloat[Trajectories of the mean-field $ (\lambda^{ A}, \lambda^{ B})$ and their microscopic counterparts $( \lambda_{ N}^{ A}, \lambda_{ N}^{ B})$.]{\includegraphics[width=0.5\textwidth]{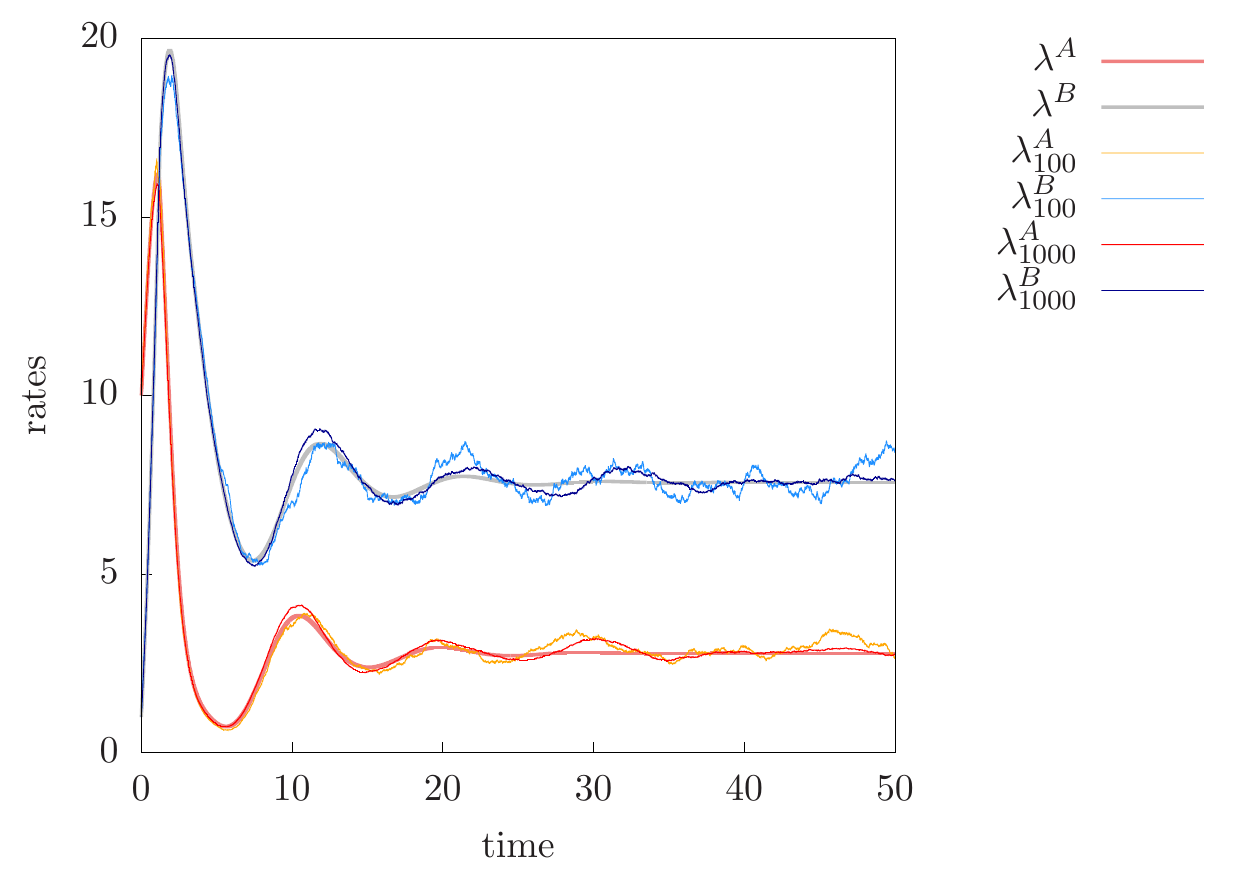}\label{subfig:lambdaA_exp_p1}}
\caption{The case of Model \eqref{eq:example_PhiBtoA_exp} with $h_{ i}= \mathbf{ 1}_{ [0, \theta_{ i}]}$ and $ \alpha=0.8$, $ \kappa_{ 1}=0.95$, $ \kappa_{ 2}= 1$, $ \kappa_{ 3}=0.5$, $ \kappa_{ 4}=1$, $ \mu_{ A}=10$, $ \mu_{ B}=1$, $ \tau=0.2$, on $[0, T]$ with $T=50$. Note that $ \kappa_{ 1}<1$ so that \eqref{eq:uB_in_I} is satisfied. The intensities converge (Figure~\ref{subfig:lambdaA_exp_p1}) whereas Assumption~\ref{ass:U} does not hold (Figure~\ref{subfig:Psi2_p1}): Assumption~\ref{ass:U} is sufficient for convergence, not necessary.}%
\label{fig:lambdaA_exp_p1}%
\end{figure}

\section{Towards oscillations: inhibition through a sigmoidal kernel}
\label{sec:oscillations}
Consider again the Model \eqref{eq:lambdaAB_lin} with full connectivity $ \kappa_{ 2} \kappa_{ 4}>0$. In most of the examples considered in Section~\ref{sec:examples}, $ \Phi_{ B\to A}$ (and hence $ \Phi$ itself) is convex. As already said, although we do not have a rigorous proof, we do not expect anything else but convergence of the intensities $( \lambda^{ A}, \lambda^{ B})$ in the convex case (even for large values of $ \kappa_{ 1}$ where \eqref{hyp:A_sub_muB_kappa3} no longer holds or in cases when Assumption~\ref{ass:U} is no longer valid, recall Figure~\ref{fig:lambdaA_exp_p1}). 

The case where $ \Phi_{ B\to A}$ is of sigmoid type (hence no longer convex) is more intriguing. In the numerical examples below, the following two examples will be considered: for fixed $R>0$ and $ \beta>0$ a possibly large parameter,
\begin{align}
\Phi_{ B\to A}(x)&=  \frac{ 1}{ 1+ \left( \frac{ x}{ R}\right)^{ \beta}},\label{eq:sigmoid_pol}\\
\Phi_{ B\to A}(x)&=  \frac{ 1}{ 2}-\frac{ 1}{ \pi}\arctan \left( \beta \left(x-R\right)\right).\label{eq:sigmoid_atan}
\end{align}
The kernel \eqref{eq:sigmoid_pol} is the same example as \eqref{eq:PhiBtoA_pol} (take $ \tau= 1/{ R^{ \beta}}$). Both are smooth approximations as $ \beta\to\infty$ of the indicator function
\begin{equation}
\label{eq:PhiBA_ind}
\Phi_{ B\to A} (x) = \mathbf{ 1}_{ [0, R]}(x).
\end{equation}
However, There is  every reasons to believe that the main features developed in this section should not depend specifically on the choice of the sigmoid kernel, as long as it is sufficiently close to \eqref{eq:PhiBA_ind}. In this paragraph we illustrate, both with non-rigorous intuition and numerical examples that the lack of convexity of $ \Phi_{ B\to A}$ is likely to induce oscillations for the solution $(\lambda^{ A}, \lambda^{ B})$ to \eqref{eq:lambdaAB_lin} (recall Remark~\ref{rem:effect_inhib}).  The existence of limit cycles for \eqref{eq:lambdaAB_lin} is only made explicit based on numerical simulations. A rigorous proof is lacking and will be the object of a future work.

\subsection{The case of an indicator}
\label{sec:phase_transition_intuition}
In order to give some intuition about the possible emergence of limit cycles for \eqref{eq:lambdaAB_lin}, replace for the moment the kernels $ \Phi_{ B\to A}$ in \eqref{eq:sigmoid_pol} or \eqref{eq:sigmoid_atan} by their common limit \eqref{eq:PhiBA_ind} as $ \beta\to \infty$. Obviously, none of the rigorous results developed in the previous sections should apply (at least easily) to the case \eqref{eq:PhiBA_ind}, as this kernel is not even continuous. Even the well-posedness of the particle system \eqref{eq:Hawkes} or its mean-field limit \eqref{eq:barZ_gen} is unclear in this case, not to mention the analysis made in Section~\ref{sec:MFL} to derive the limit of the mean-field system as $t\to\infty$. This being said, suppose anyway that one would be able to consider \eqref{eq:PhiBA_ind} as a proper candidate for the inhibition kernel $ \Phi_{ B\to A}$. We place ourselves in the framework where $ \mu_{ A}$ is large and $ \mu_{ B}$ small (possibly even $0$). In 'real life', our $\mu_{ A}$ and $\mu_{ B}$ account for the following two features: i) the weaker one, the intrinsic properties of the neurons (that make them more or less prone to spike spontaneously) \cite{luo:15}; ii) the stronger one, the excitatory inputs from the other brain regions (other neocortical area and \emph{mostly} thalamic inputs \cite{white:89,braitenberg.schuz:98}). There is strong anatomical evidence that the latter are stronger onto excitatory neurons than onto inhibitory ones (reviewed in \cite{white:89}), justifying the relevance  of considering $\mu_{ A} \gg \mu_{ B}$.

The intuition for the emergence of oscillations is the following: considering \eqref{eq:lambdaAB_lin} at $t\approx 0$, as  $ \mu_{ A} \gg \mu_{ B}\approx 0$, the intensity of population $A$ (resp. $B$) should be high (resp. small). Then, in presence of feedback from $A$ onto $B$ ($ \kappa_{ 4}>0$), the spiking activity of population $A$ propagates to population $B$ so that the intensity of $B$ increases. But the higher the activity of $B$ is, the larger $X_{ B}(t):=(1- \alpha) \int_{ 0}^{t} h_{ 2}(t-u)  \lambda_{ u}^{B} {\rm d}u$ (that is the term within the inhibition kernel $ \Phi_{ B\to A}(\cdot)$ in the first equation of \eqref{eq:lambdaAB_lin}) becomes. Once $X_{ B}(t)$ has reached the threshold $R>0$ of \eqref{eq:PhiBA_ind}, the activity of $A$ gets killed (recall that $ \kappa_{ 2}>0$). Hence, supposing for instance that the memory kernel $h_{ 4}$ has compact support with length $ \theta_{ 4}$, after a time $ \theta_{ 4}$, population $B$ will no longer feel the influence of population $A$: intensity of $B$ will go back to $ \mu_{ B}\approx 0$. Thus, $X_{ B}(t)$ crosses back the threshold $R$: population $A$ goes back to its normal high activity, and so on: oscillations should appear.

For this to happen, the threshold $R$ should not be too large, otherwise the term $X_{ B}(t)$ will not be able to reach $R$. We can actually pursue the intuition given by \eqref{eq:PhiBA_ind} in deriving informally a phase transition. Suppose again that the formalism developed in Section~\ref{sec:fully_coupled} is valid for \eqref{eq:PhiBA_ind} (which is not). Note first that the subcriticality condition \eqref{hyp:A_sub_muB_kappa3} simply boils down here to
\begin{equation}
\kappa_{ 1}<1
\end{equation}
that is population $A$ alone is subcritical. The fixed-point function $ \Phi$ becomes
\begin{equation}
\Phi(x)= \frac{ a}{ \kappa_{ 2}} +  \frac{ \mu_{ A} b}{ \kappa_{ 2}(1- \kappa_{ 1})}\mathbf{ 1}_{ [0, R]}(\kappa_{ 2}x).
\end{equation}
Recall that the whole point of Section~\ref{sec:fully_coupled} is to compute the fixed-point of $ \Phi$. Straightforward computations enable to distinguish two cases:
\begin{itemize}
\item Either $ \frac{ R-a}{ b}\geq \frac{ \mu_{ A}}{ 1- \kappa_{ 1}}$: Then $ \Phi$ has a unique fixed-point $\ell=  \frac{ a}{ \kappa_{ 2}} + \frac{ \mu_{ A} b}{ \kappa_{ 2}(1- \kappa_{ 1})}$ (which should be the limit of $ \lambda^{ B}$ if the results of Section~\ref{sec:fully_coupled} would apply to \eqref{eq:PhiBA_ind}). In this case, the limit of $ \lambda^{ A}$ should be $\ell_{ A}= \Psi_{ 1}(\ell)= \frac{ \mu_{ A} \mathbf{ 1}_{ [0, R]}( \kappa_{ 2}\ell)}{ 1- \kappa_{ 1} \mathbf{ 1}_{ [0, R]}(\kappa_{ 2}\ell)}= \frac{ \mu_{ A}}{ 1- \kappa_{ 1}}$.
\item Either $ \frac{ R-a}{ b}< \frac{ \mu_{ A}}{ 1- \kappa_{ 1}}$: in this case, there is no fixed-point for $ \Phi$. But the main conclusion of Theorem~\ref{th:full_conv} that any possible limit for $ \lambda^{ B}$ is necessarily given by $\ell$ fixed-point of $ \Phi$. Hence, if one believes in the conclusion of Theorem~\ref{th:full_conv} for \eqref{eq:PhiBA_ind}, we would necessarily have $ \underline{ \ell}_{ B}< \bar \ell_{ B}$: this is the regime where one expects oscillations.
\end{itemize}
Of course, once again, this phase transition is only formal, as \eqref{eq:PhiBA_ind} does enter into the hypotheses used in the paper. What is more, this formal argument is not stable by perturbation: if one replaces \eqref{eq:PhiBA_ind} by any continuous sigmoid function with the same threshold $R$, Theorem~\ref{th:full_conv} does apply in this case and $ \Phi$ has always a fixed-point and the formal argument above becomes meaningless. It appears nonetheless that this formal phase transition is actually meaningful and visible in simulations in the sigmoid case, see Section~\ref{sec:phase_transition_num}.

\subsection{ The case of a sigmoid kernel}
\label{sec:case_sigmoid}
Having in mind the intuition provided by the previous paragraph, let us go back to the case where $ \Phi_{ B\to A}$ is defined as a smooth Lipschitz sigmoid function, such as in \eqref{eq:sigmoid_pol} of \eqref{eq:sigmoid_atan}. Here, every results of the present paper apply. In view of Theorem~\ref{th:full_conv}, there are only two ways to capture oscillations: 
\begin{enumerate}
\item either to ensure that \eqref{eq:uB_in_I} no longer holds, i.e. situations where $ \kappa_{ 1}\gg 1$ so that (with the notations of the previous sections, recall \eqref{eq:I_kappas}) 
\begin{equation}
\label{eq:oscill_ellB}
\underline{ \ell}_{ B} < x_{ \kappa_{ 1}, \kappa_{ 2}}^{ \ast} < \bar \ell_{ B}.
\end{equation} 
This corresponds to a situation where population $A$, in isolation from $B$ would be highly supercritical (but still subcritical with inhibition and retroaction, see Theorem~\ref{th:A_subcritical}). The technical difficulty is that we lack some a priori upper bounds on $ \underline{ \ell}_{ B}$ in order to give some rigorous sufficient conditions for \eqref{eq:oscill_ellB}. We give nonetheless a numerical illustration of this case in Figure~\ref{fig:lA_atan_k1_5}.
\begin{figure}[ht]
\centering
\includegraphics[width=\textwidth]{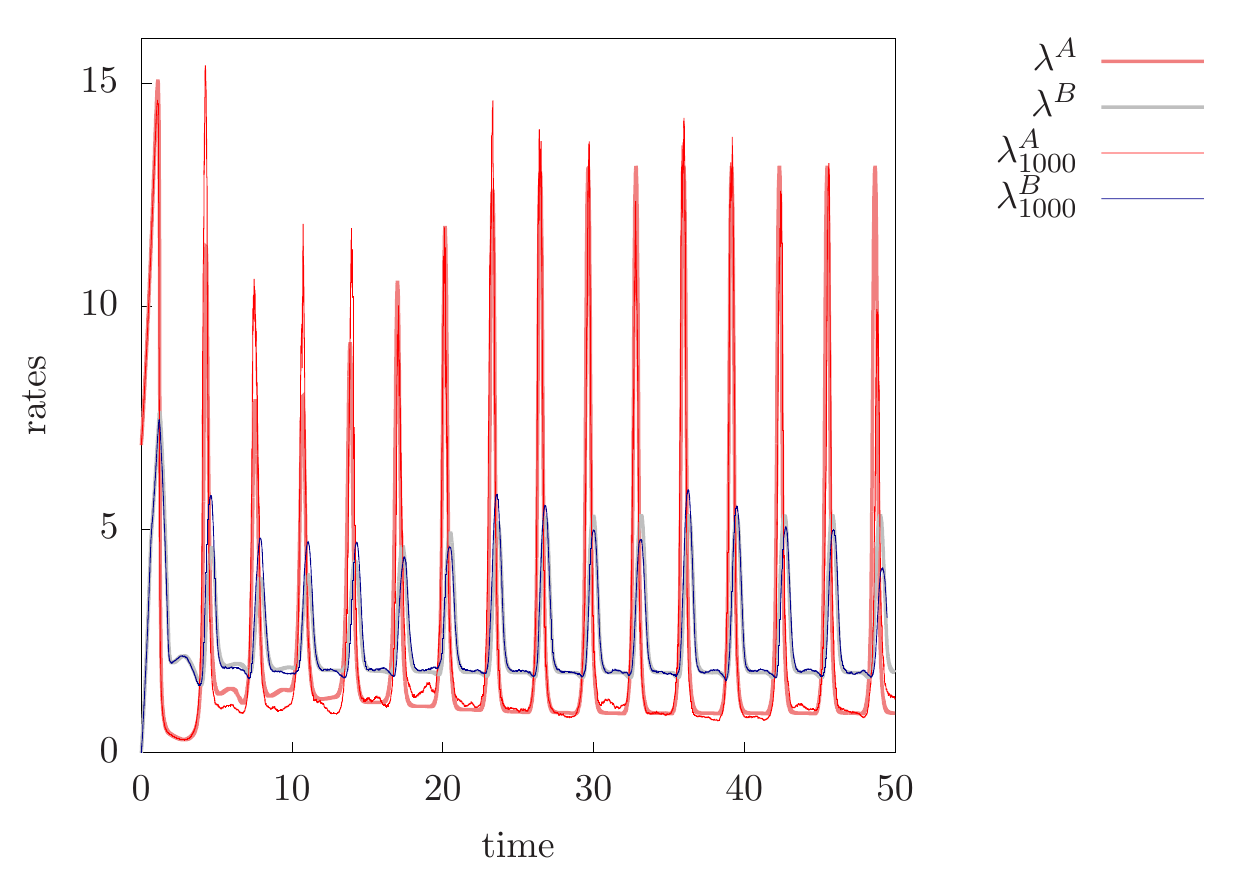}
\caption{Trajectories of the mean-field $ (\lambda^{ A}, \lambda^{ B})$ and their microscopic counterparts $( \lambda_{ N}^{ A}, \lambda_{ N}^{ B})$, when $ \Phi_{ B\to A}$ is given by \eqref{eq:sigmoid_atan}, with $R=1$, $ \beta=20$, $h_{ i}= \mathbf{ 1}_{ [0, \theta_{ i}]}$, $ \alpha=0.8$, $ \kappa_{ 1}=5$, $\kappa_{ 2}= \kappa_{ 3}= \kappa_{ 4}=0.5$, $ \mu_{ A}=7$ and $ \mu_{ B}=0$, on $[0, T]$ with $T=50$. Note that $\kappa_{ 1} \Phi_{ B\to A} \left(\kappa_{ 2} \underline{ \ell}_{ B}\right)\approx 4.81>1$ so that \eqref{eq:uB_in_I} does not hold: we have oscillations.}
\label{fig:lA_atan_k1_5}
\end{figure}
\item or either in situations where \eqref{eq:uB_in_I} is valid (e.g. by supposing \eqref{hyp:A_sub_muB_kappa3}; note that in simulations below, we take $ \kappa_{ 1}<1$ so that \eqref{eq:uB_in_I} is verified immediately), but where Assumption~\ref{ass:U} no longer holds (recall that Assumption~\ref{ass:U} is sufficient for convergence, not necessary, but oscillations necessarily occur when Assumption~\ref{ass:U} is not true). The difficulty (already met in the proof of Proposition~\ref{prop:conv_pol}) is that checking Assumption~\ref{ass:U} requires nontrivial algebraic estimates that may be, in general, out of reach from any rigorous treatment. Therefore, we only check Assumption~\ref{ass:U} numerically here. 
\end{enumerate}

\subsubsection{Phase transition and emergence of oscillations}
\label{sec:phase_transition_num}
The first observation of this paragraph is that the phase transition derived formally in the previous section is (numerically) accurate: take $ h_{ i}= \mathbf{ 1}_{ [0, \theta_{ i}]}$, $ \theta_{ i}>0$, $i=1, \ldots, 4$, $ \alpha=0.8$ and for simplicity $ \kappa_{ i}=0.5$ for $i=1, \ldots, 4$ and $ \mu_{ B}=0$. Consider $ \beta$ in \eqref{eq:sigmoid_pol} (or in \eqref{eq:sigmoid_atan}) sufficiently large so that $ \Phi_{ B\to A}$ is close enough to the indicator function \eqref{eq:PhiBA_ind} (in pratice, interesting dynamical features already appears for $ \beta\approx 10$). With these parameters, the phase transition described in Section~\ref{sec:phase_transition_intuition} simply boils down to $ R\geq \mu_{ A}$ or $R< \mu_{ A}$. The corresponding simulations are displayed in Figures~\ref{fig:muA099} and \ref{fig:muA101} below, made for the choice of \eqref{eq:sigmoid_pol} with $ \beta=1000$ and $R=1$: for $ \mu_{ A}=0.99$, observe numerically in Figure~\ref{fig:muA099} that $ \mathcal{ U}_{ \Phi}$ is trivial (Assumption~\ref{ass:U} is true) and therefore Theorem~\ref{th:full_conv} applies: we have convergence of $(\lambda^{ A}, \lambda^{ B})$. On the contrary, for $ \mu_{ A}= 1.01$, $ \mathcal{ U}_{ \Phi}$ is no longer reduced to $(\ell, \ell)$ (Assumption~\ref{ass:U} is not valid) and oscillations appear (Figure~\ref{fig:muA101}). These oscillations are all the more effective when $ \mu_{ A}$ gets larger, see Figure~\ref{fig:lA_3_sig}. 
\begin{figure}[h]
\centering
\subfloat[Verifying the hypotheses of Theorem~\ref{th:full_conv}: representation of $ \Phi$ (in blue) and $ \Phi^{ -1}$ (in red) on the interval $ \left(0, \ell\right)$. We see here that $ \Phi< \Phi^{ -1}$ so that, by \eqref{eq:def_UPhi3}, $ \mathcal{ U}_{ \Phi}$ is trivial so that Assumption~\ref{ass:U} is true. Here $ \kappa_{ 1}<1$ and \eqref{eq:uB_in_I} is trivial.]{\includegraphics[width=0.45\textwidth]{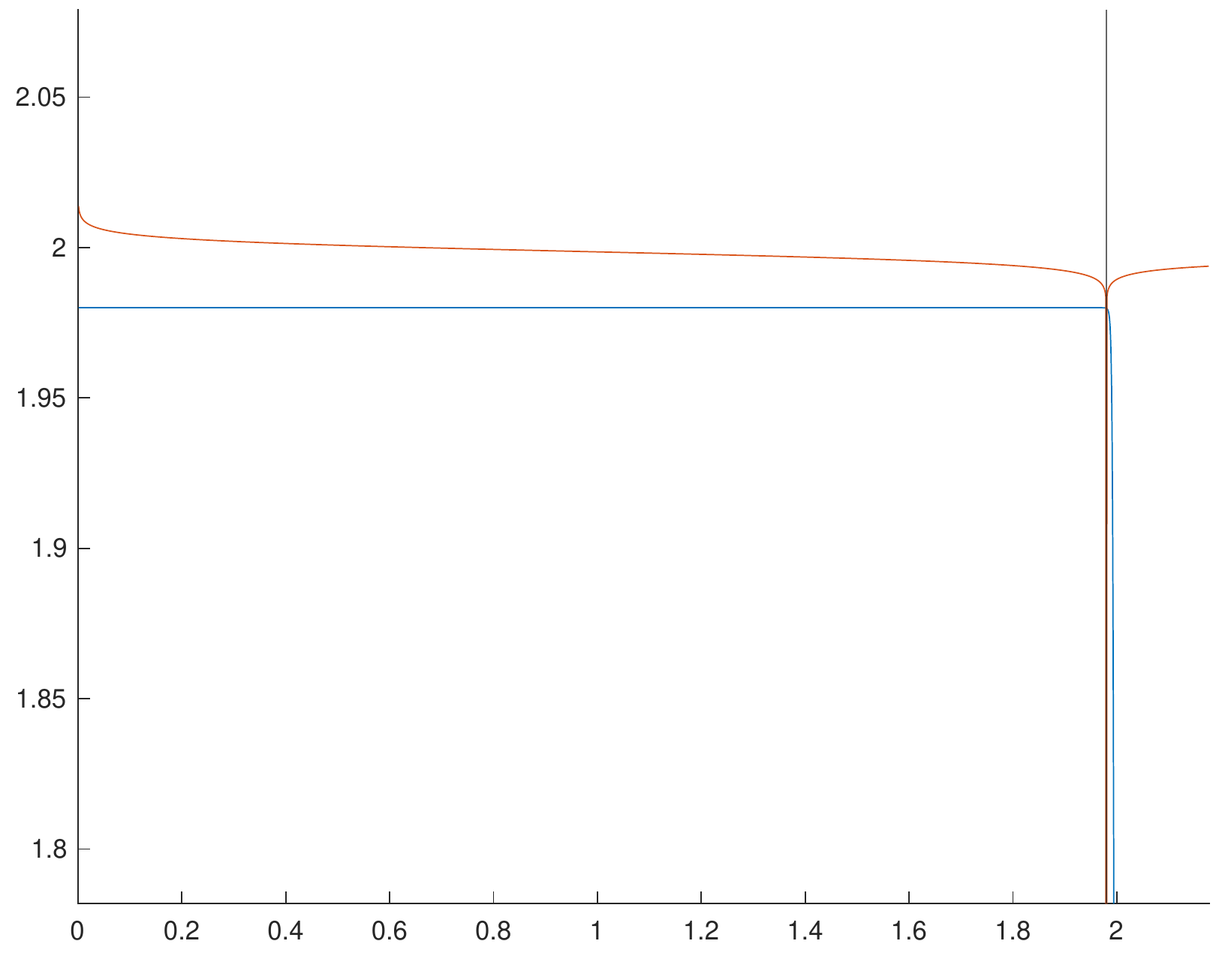}
\label{subfig:UPhi099}}
\quad\subfloat[Trajectories of the mean-field $ (\lambda^{ A}, \lambda^{ B})$ and their microscopic counterparts $( \lambda_{ N}^{ A}, \lambda_{ N}^{ B})$.]{\includegraphics[width=0.5\textwidth]{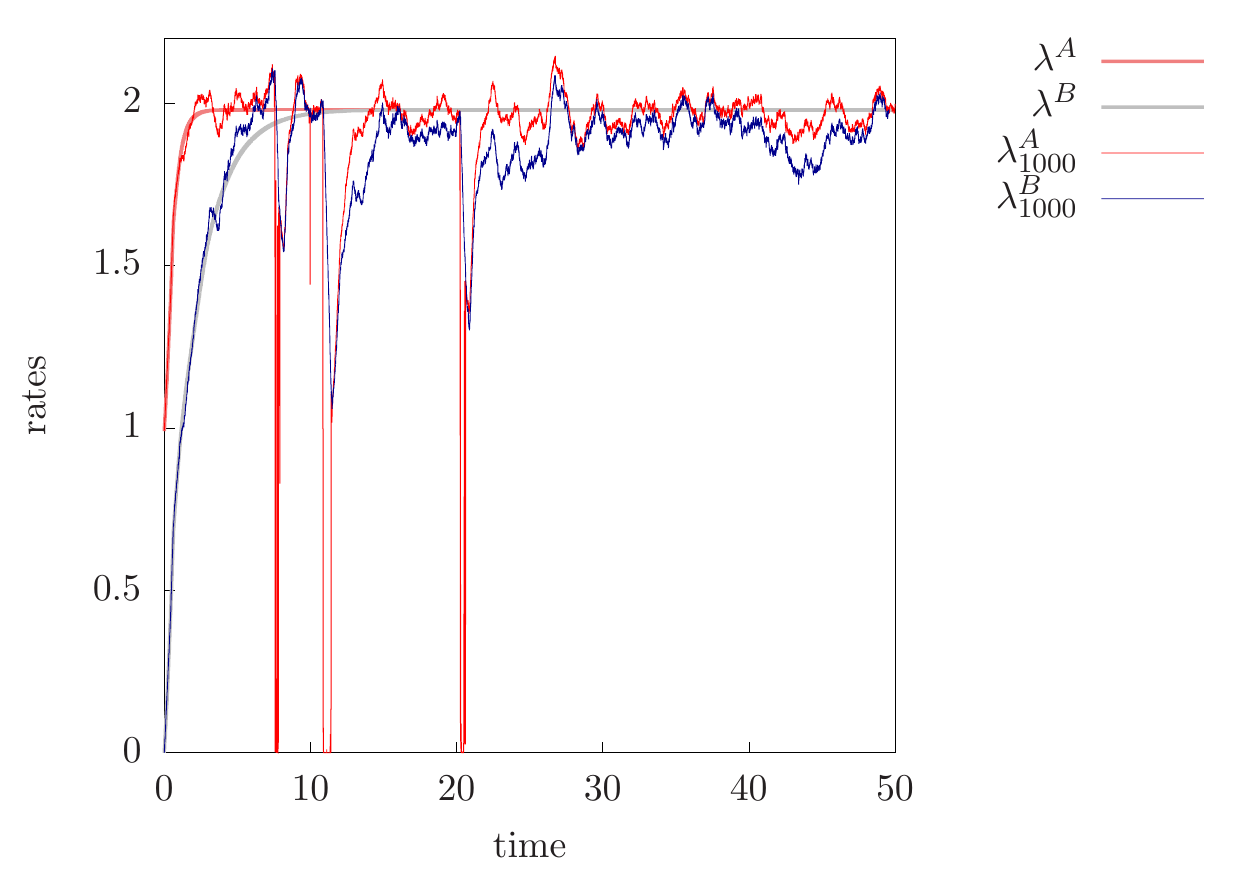}
\label{subfig:traj099}}
\caption{The case $ \Phi_{ B\to A}$ given by \eqref{eq:sigmoid_pol}, with $R=1$, $ \beta=1000$, $h_{ i}= \mathbf{ 1}_{ [0, \theta_{ i}]}$, $ \alpha=0.8$, $ \kappa_{ 1}= \kappa_{ 2}= \kappa_{ 3}= \kappa_{ 4}=0.5$ and $ \mu_{ B}=0$. Here we take $ \mu_{ A}=0.99$.  The conclusions of Theorem~\ref{th:full_conv} hold in this case (Figure~\ref{subfig:UPhi099}) and we have convergence of the intensities as $t\to\infty$ (Figure~\ref{subfig:traj099}).}%
\label{fig:muA099}%
\end{figure}
\begin{figure}[h]
\centering
\subfloat[Representation of $ \Phi$ (in blue) and $ \Phi^{ -1}$ (in red) on the interval $ \left(0, \ell\right)$. We see here that $ \mathcal{ U}_{ \Phi}$ is nontrivial (recall \eqref{eq:def_UPhi2}): Assumption~\ref{ass:U} no longer holds.]{\includegraphics[width=0.45\textwidth]{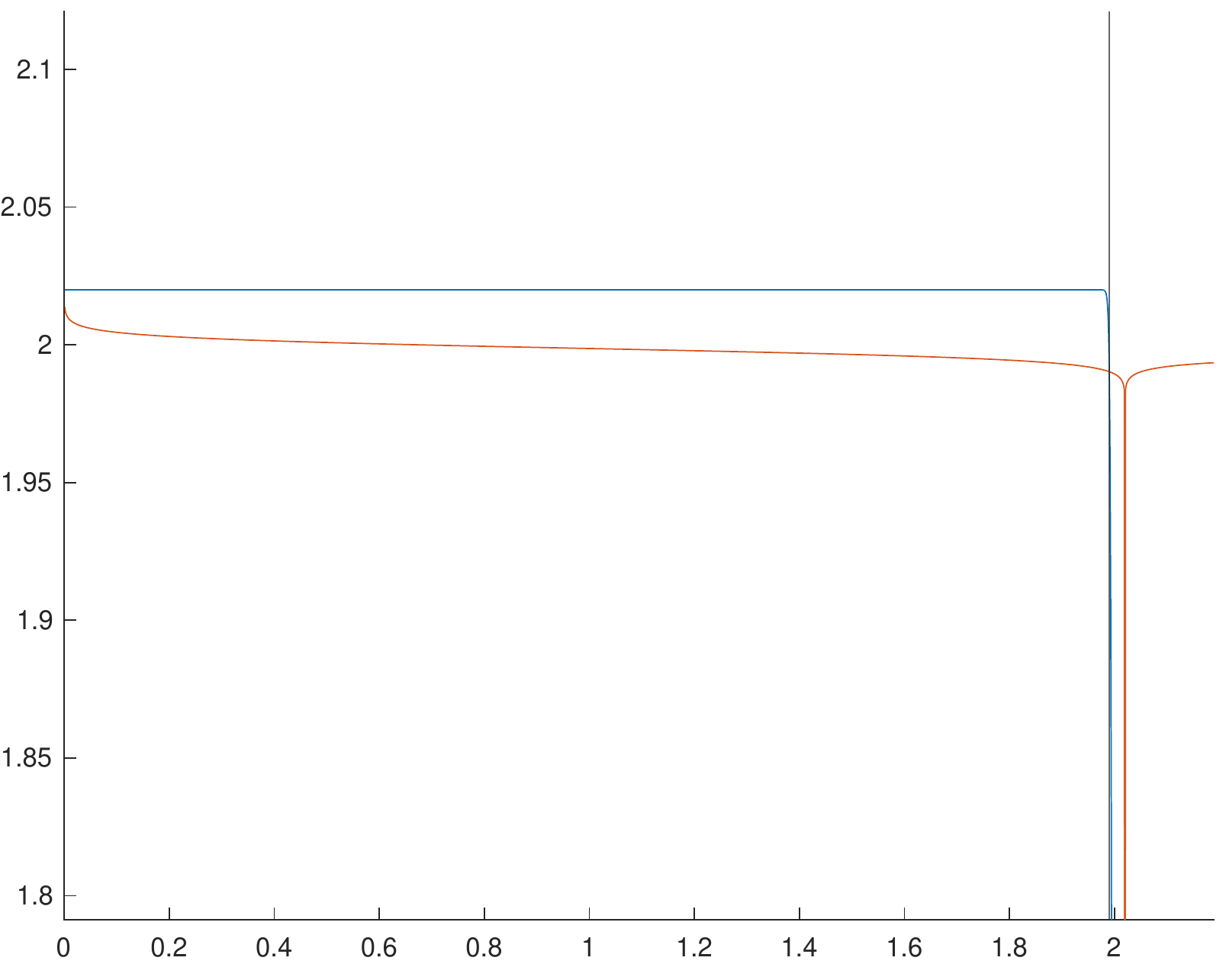}
\label{subfig:UPhi101}}
\quad\subfloat[Trajectories of the mean-field $ (\lambda^{ A}, \lambda^{ B})$ and their microscopic counterparts $( \lambda_{ N}^{ A}, \lambda_{ N}^{ B})$.]{\includegraphics[width=0.5\textwidth]{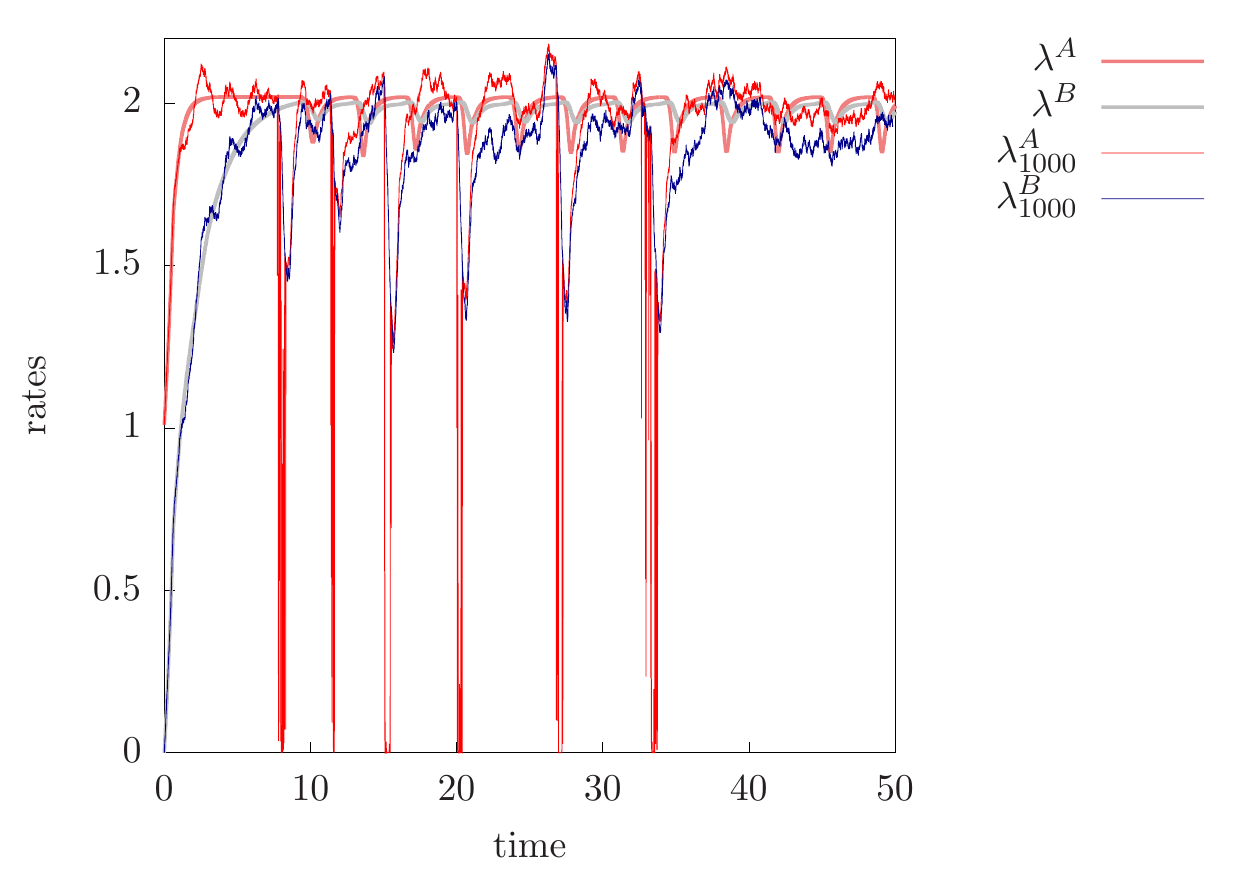}
\label{subfig:traj101}}
\caption{The case $ \Phi_{ B\to A}$ given by \eqref{eq:sigmoid_pol}, with $R=1$, $ \beta=1000$, $h_{ i}= \mathbf{ 1}_{ [0, \theta_{ i}]}$, $ \alpha=0.8$, $ \kappa_{ 1}= \kappa_{ 2}= \kappa_{ 3}= \kappa_{ 4}=0.5$ and $ \mu_{ B}=0$. Here we take $ \mu_{ A}=1.01$. Assumption~\ref{ass:U} is no longer valid (Figure~\ref{subfig:UPhi101}) and we observe oscillations (Figure~\ref{subfig:traj101}).}%
\label{fig:muA101}%
\end{figure}
\begin{figure}[ht]
\centering
\includegraphics[width=\textwidth]{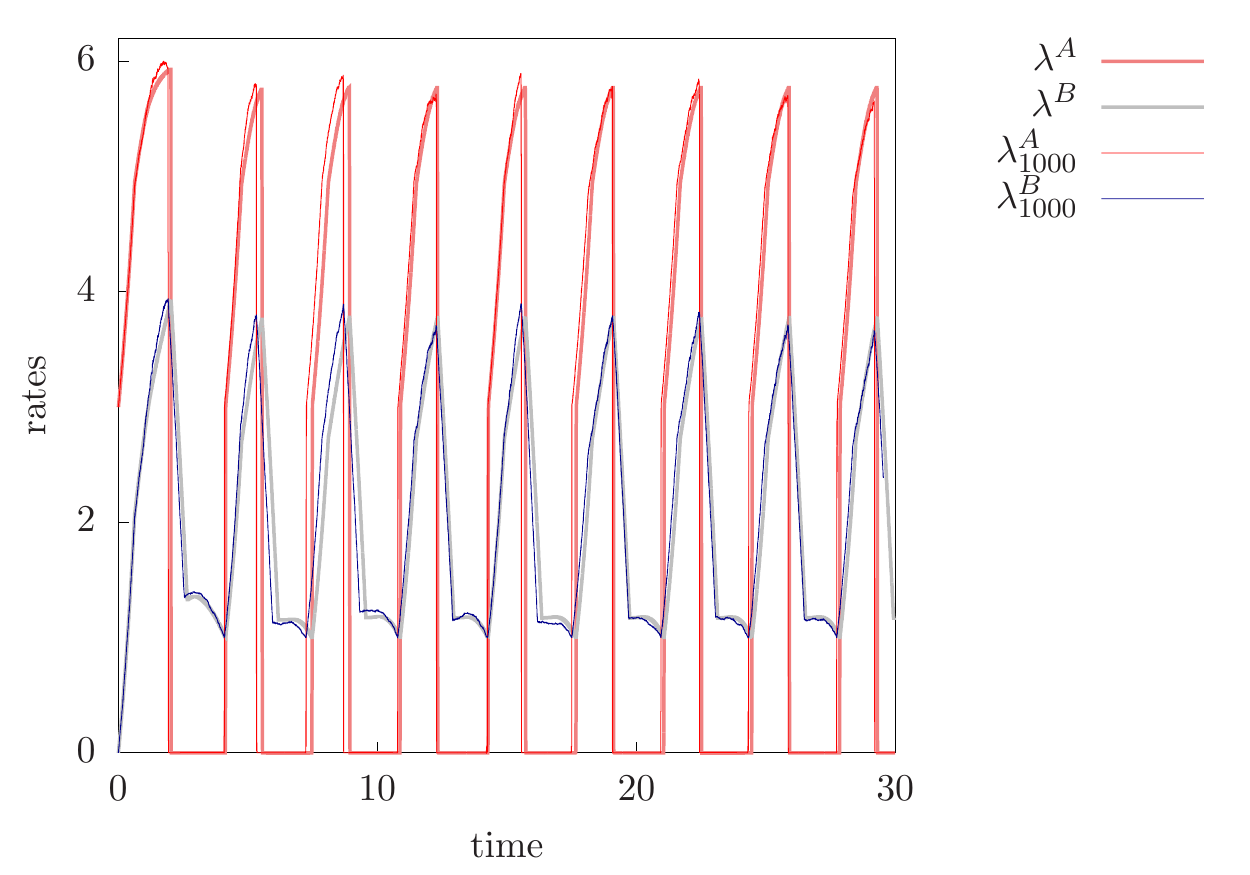}
\caption{Trajectories of the mean-field $ (\lambda^{ A}, \lambda^{ B})$ and their microscopic counterparts $( \lambda_{ N}^{ A}, \lambda_{ N}^{ B})$,  when $ \Phi_{ B\to A}$ is given by \eqref{eq:sigmoid_pol}, with $R=1$, $ \beta=1000$, $h_{ i}= \mathbf{ 1}_{ [0, \theta_{ i}]}$, $ \alpha=0.8$, $ \kappa_{ 1}= \kappa_{ 2}= \kappa_{ 3}= \kappa_{ 4}=0.5$, $ \mu_{ A}=3$ and $ \mu_{ B}=0$, on $[0, T]$ with $T=30$.}
\label{fig:lA_3_sig}
\end{figure}
Note that this phase transition does not intrinsically depend on the specific choice of the sigmoid kernel $ \Phi_{ B\to A}$. Replacing \eqref{eq:sigmoid_pol} by \eqref{eq:sigmoid_atan} leads to similar patterns, even for smaller values of $ \beta$, see Figure~\ref{fig:lA_atan_k1petit}. However the precision of the phase transition around $ \mu_{ A}^{ c}:=R$ degrades as $ \beta$ is not too large.

\subsubsection{The role of memory kernels in oscillations}
\label{sec:role_memory}
\begin{figure}[!h]
\centering
\includegraphics[width=\textwidth]{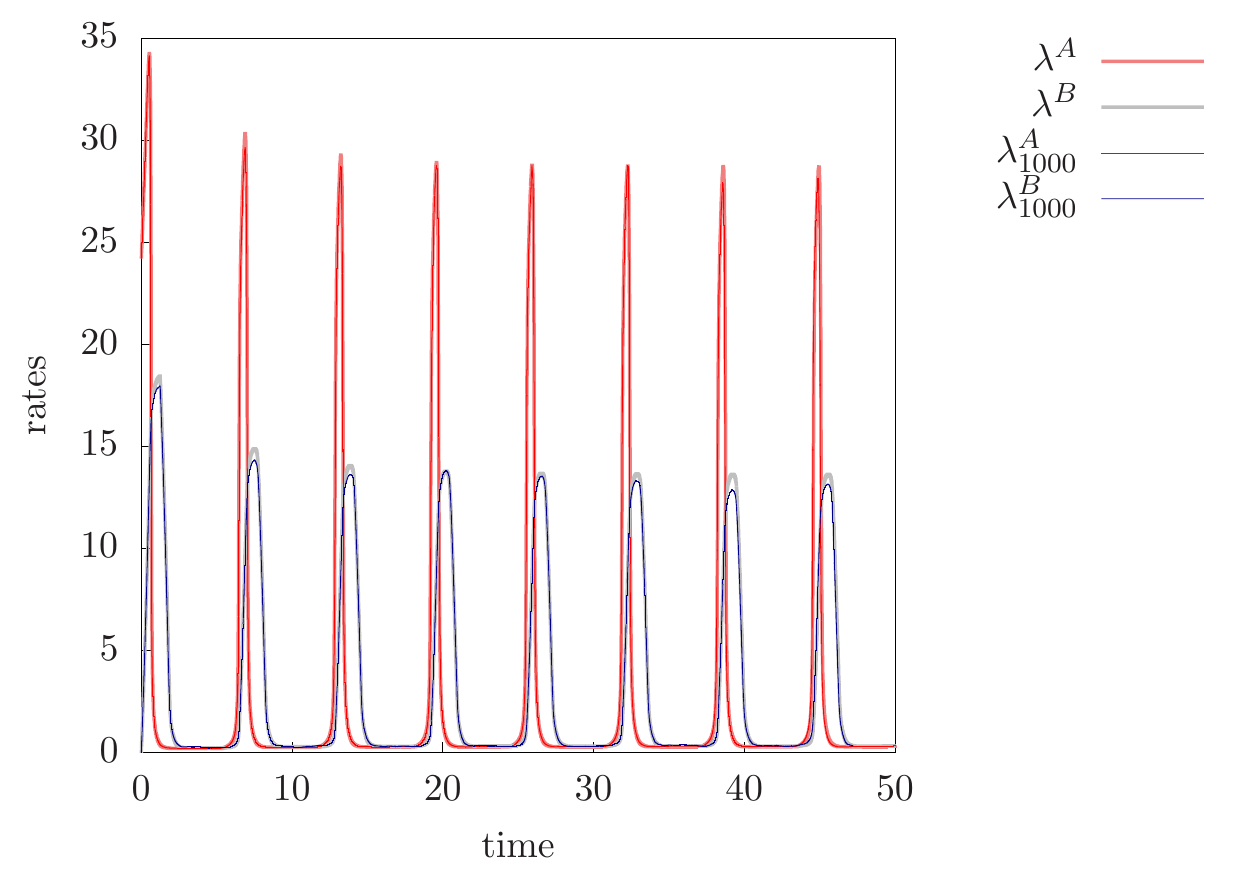}
\caption{Trajectories of the mean-field $ (\lambda^{ A}, \lambda^{ B})$ and their microscopic counterparts $( \lambda_{ N}^{ A}, \lambda_{ N}^{ B})$,  when $ \Phi_{ B\to A}$ is given by \eqref{eq:sigmoid_atan}, with $R=1$, $ \beta=10$, $h_{ i}= \mathbf{ 1}_{ [0, \theta_{ i}]}$, $ \alpha=0.8$, $ \kappa_{ 1}=0.5$, $\kappa_{ 2}= 1$, $\kappa_{ 3}=0.1$, $\kappa_{ 4}=1.0$, $ \mu_{ A}=25$ and $ \mu_{ B}=0$, on $[0, T]$ with $T=50$.}
\label{fig:lA_atan_k1petit}
\end{figure}

The point we want to raise here is the importance of the choice of the memory kernels $h_{ i}$ in the emergence of oscillations. Previous simulations were made 
 with compactly supported $ h_{ i}= \mathbf{ 1}_{ [0, \theta_{ i}]}$. Several other choices for the $h_{ i}$ can be made, such as e.g. exponential $h_{ i}(u)= e^{ - \frac{ u}{ \theta_{ i}}}$ or Erlang kernels $h_{ i}(u)= e^{ - \frac{ u}{ \theta_{ i}}} \frac{ u^{ n}}{ n!}$ \cite{Duarte:2016aa}. This last choice of exponential kernels has the nice ability to transform the convolution equation governing the mean-field intensity $ \lambda(t)$ of a single population of Hawkes processes into an ODE governing the mean activity $X(t)= \int_{ 0}^{t} e^{ - \frac{ t-s}{ \theta}} \lambda(s) {\rm d}s$ (see e.g. \cite{Duarte:2016aa} for further details). The same procedure can be applied here: set $ h_{ i}(u)= e^{ - \frac{ u}{ \theta_{ i}}}$, $i=1, \ldots, 4$ and write $X_{1}(t):= \alpha \int_{ 0}^{t} h_{1}(t-s) \lambda^{A}(s) {\rm d}s$, $X_{2}(t):= (1-\alpha) \int_{ 0}^{t} h_{2}(t-s) \lambda^{B}(s) {\rm d}s$, $X_{3}(t):= (1-\alpha) \int_{ 0}^{t} h_{3}(t-s) \lambda^{B}(s) {\rm d}s$, $X_{4}(t):= \alpha \int_{ 0}^{t} h_{4}(t-s) \lambda^{A}(s) {\rm d}s$, one easily sees that
the vector $(X_{1}, X_{2}, X_{3}, X_{4})$ solves the system of coupled ODEs
\begin{equation}
\label{eq:ode_X}
\begin{cases}
\frac{ {\rm d}}{ {\rm d}t} X_{1}(t) &= - \frac{ 1}{ \theta_{1}} X_{1}(t) + \alpha \left( \mu_{A} + X_{1}(t)\right) \Phi_{B\to A} \left(X_{2}(t)\right),\\
\frac{ {\rm d}}{ {\rm d}t} X_{2}(t)&= - \frac{ 1}{ \theta_{2}} X_{2}(t) + (1- \alpha) \left(\mu_{B} + X_{3}(t) + X_{4}(t)\right),\\
\frac{ {\rm d}}{ {\rm d}t} X_{3}(t)&= - \frac{ 1}{ \theta_{3}} X_{3}(t) + (1- \alpha) \left(\mu_{B} + X_{3}(t) + X_{4}(t)\right),\\
\frac{ {\rm d}}{ {\rm d}t} X_{4}(t) &= - \frac{ 1}{ \theta_{4}} X_{4}(t) + \alpha \left( \mu_{A} + X_{1}(t)\right) \Phi_{B\to A} \left(X_{2}(t)\right).
\end{cases}
\end{equation}
Nonetheless, no oscillations have been numerically observed for \eqref{eq:ode_X}: for the same set of parameters, limit cycles that are present in the compact case $h_{ i}= \mathbf{ 1}_{ [0, \theta_{ i}]}$ disappear with exponential kernels, replaced by damped oscillations. To be more precise, oscillations disappear whenever both $h_{ 2}$ and $h_{ 4}$ are exponential. We interpret this situation as follows: for the intuition described in Section~\ref{sec:phase_transition_intuition} to work, it is necessary for population $B$ that it forgets rapidly the influence of population $A$ or that inhibition of $B$ onto $A$ has finite memory, that is one among $h_{ 2}$ or $h_{4}$ has compact support. Concerning $h_{ 4}$, the influence of population $A$ onto $B$ is made through $X_{4}= \alpha\int_{ 0}^{t} h_{4}(t-s) \lambda^{A}(s) {\rm d}s$ which consists in integrating on the whole history of population $A$. Therefore, if $h_{ 4}$ has infinite support, $X_{4}$ conserves the influence of nontrivial activity of $A$ for arbitrary long times and even if the activity of $A$ may sometimes vanish, it may not be sufficient for a significant decrease in the feedback from $A$ to $B$. The same argument applies for $h_{2}$: if $h_{ 2}$ has long memory, $X_{2}$ keeps the trace of periods when activity of $B$ was significant,  therefore $X_{2}$ cannot go back to its initial state where $ \mu_{B}\approx0$.

\section{Proofs of well-posedness results and propagation of chaos}
\label{sec:proofs_part_MF}
We prove here the results of Section~\ref{sec:modgen}.
\subsection{Well-posedness of the particle system}
\label{sec:proof_WP_part}
We prove Proposition~\ref{prop:ExistUniq}. Similarly to the  proof of Proposition~2.1 \cite{costa2020renewal},  we proceed by an inductive thinning procedure  on the Poisson point processes $ \left( \left(\pi_{ i}\right)_{ i=1, \ldots, N_{ A}}, \left(\tilde{ \pi}_{ i}\right)_{ i=N_{ A}+1, \ldots, N}\right)$.
\paragraph{Initialization} Define $ \Lambda_{ 0}^{ A}:= F(0,0)$ and $ \Lambda_{ 0}^{ B}:= G(0,0)$ and 
\begin{align*}
U_{ 1}^{ (i)}&:= \inf \left\lbrace u>0,\ \int_{ (0, u]} \int_{(0, \Lambda_{ 0}^{ A}]} \pi_{ i}({\rm d}s, {\rm d}z)>0\right\rbrace,\ i=1, \ldots, N_{ A},\\
U_{ 1}^{ (i)}&:= \inf \left\lbrace u>0,\ \int_{ (0, u]} \int_{(0, \Lambda_{ 0}^{ B}]} \tilde{ \pi}_{ i}({\rm d}s, {\rm d}z)>0\right\rbrace,\ i=N_{ A}+1, \ldots, N
\end{align*}
as well as
\begin{equation}
\label{eq:U1}
U_{ 1} := \min(U_{ 1}^{ (i)}, i=1, \ldots, N).
\end{equation}
Each $U_{ 1}^{ (i)}$ for $i=1, \ldots, N_{ A}$ (resp. for $i=N_{ A}+1, \ldots, N$), $U_{ 1}^{ (i)}$ is the first atom of an homogeneous Poisson process with parameter $ \Lambda_{ 0}^{ A}$ (resp. $ \Lambda_{ 0}^{ B}$). Hence, almost surely, $U_{ 1}^{ (i)}$ and thus $U_{ 1}$ is strictly positive and finite. Moreover, by independence of the Poisson point processes $ \left( \left(\pi_{ i}\right)_{ i=1, \ldots, N_{ A}}, \left(\tilde{ \pi}_{ i}\right)_{ i=N_{ A}+1, \ldots, N}\right)$, the variables $U_{ 1}^{ (i)}, i=1, \ldots, N$ are almost surely distinct so that 
\begin{equation}
\label{eq:I1}
I_{ 1}:= \left\lbrace i\in \left\lbrace1, \ldots, N\right\rbrace, U_{ 1}^{ (i)}= U_{ 1}\right\rbrace
\end{equation}
is defined without ambiguity. Since any solution to \eqref{eq:Hawkes} has constant intensities $F(0,0)$ and $G(0,0)$ until the first atom, the above construction yields that $U_{ 1}$ given by \eqref{eq:U1} is necessarily the first atom of $(Z_{ t}^{i})_{ i=1, \ldots, N}$ and that it concerns the process $Z^{i}$ for $i=I_{ 1}$.
\paragraph{Recursion} Suppose we have constructed $0< U_{ 1}< \ldots< U_{ k}$ and $(I_{ 1}, \ldots, I_{ k})\in \left\lbrace1, \ldots, N\right\rbrace^{ k}$ such that on the event $ \left\lbrace U_{ k}<\infty\right\rbrace$, $0< U_{ 1}< \ldots< U_{ k}$ are the $k$ first atoms of the aggregated process 
\begin{equation}
\label{eq:Z}
Z= \sum_{ i=1}^{ N} Z^{i}
\end{equation} and for all $j=1, \ldots, k$, $I_{ j}$ is the index in $ \left\lbrace1, \ldots, N\right\rbrace$ responsible for the $j$th atom. Let 
\begin{equation*}
Z_{ A}^{ (k)}:= \sum_{\substack{ j=1, \ldots, k,\\ I_{ j}\in \left\lbrace1, \ldots, N_{ A}\right\rbrace}} \delta_{ U_{ j}}\ \text{ and }Z_{ B}^{ (k)}:= \sum_{\substack{ j=1, \ldots, k,\\ I_{ j}\in \left\lbrace N_{ A}+1, \ldots, N\right\rbrace}} \delta_{ U_{ j}}
\end{equation*}
be the point measures corresponding to atoms which concern population $A$ and $B$ respectively. Define
\begin{align*}
\Lambda_{ k}^{ A}(t)&:= F \left(\frac{ 1}{ N} \int_{ (0, U_{ k}]} h_{ 1}(t-s) {\rm d}Z_{ A, s}^{ (k)}, \frac{ 1}{ N}\int_{ (0, U_{ k}]} h_{ 2}(t-s) {\rm d}Z_{ B, s}^{ (k)}\right), \label{eq:Lambda_Ak}\\
\Lambda_{ k}^{ B}(t)&:= G \left(\frac{ 1}{ N} \int_{ (0, U_{ k}]} h_{ 3}(t-s) {\rm d}Z_{ B, s}^{ (k)}, \frac{ 1}{ N}\int_{ (0, U_{ k}]} h_{ 4}(t-s) {\rm d}Z_{ A, s}^{ (k)}\right)
\end{align*}
as well as
\begin{align*}
U_{ k+1}^{ (i)}&:= \inf \left\lbrace u>0,\ \int_{ (U_{ k}, u]} \int_{(0, \Lambda_{ k}^{ A}(s)]} \pi_{ i}({\rm d}s, {\rm d}z)>0\right\rbrace,\ i=1, \ldots, N_{ A},\\
U_{ k+1}^{ (i)}&:= \inf \left\lbrace u>0,\ \int_{ (U_{ k}, u]} \int_{(0, \Lambda_{ k}^{ B}(s)]} \tilde{ \pi}_{ i}({\rm d}s, {\rm d}z)>0\right\rbrace,\ i=N_{ A}+1, \ldots, N
\end{align*}
Define finally
$U_{ k+1}:= \min \left(U_{ k+1}^{ (i)}, i=1, \ldots, N\right).$
For $ \varepsilon>0$, denote by
\begin{align*}
\mathcal{ R}_{ \varepsilon}^{ A} &:= \left\lbrace (s, z), s\in (U_{ k}, U_{ k}+ \varepsilon],\ z\in(0, \Lambda_{ k}^{ A}(s)]\right\rbrace,\\
\mathcal{ R}_{ \varepsilon}^{ B} &:= \left\lbrace (s, z), s\in (U_{ k}, U_{ k}+ \varepsilon],\ z\in(0, \Lambda_{ k}^{ B}(s)]\right\rbrace
\end{align*}
For all $i=1, \ldots, N_{ A}$, conditionally on $ \mathcal{ F}_{ U_{ k}}$, $ \pi_{ i}(\mathcal{ R}_{ \varepsilon}^{ A})$ follows a Poisson law with parameter $ \int_{ U_{ k}}^{U_{ k}+ \varepsilon} \Lambda_{ k}^{ A}(t) {\rm d}t$, so that
\[ \mathbf{ P} \left( \pi_{ i}(\mathcal{ R}_{ \varepsilon}^{ A})<\infty\right)= \mathbf{ E} \left( \mathbf{ P} \left( \int_{ U_{ k}}^{U_{ k}+ \varepsilon} \Lambda_{ k}^{ A}(t) {\rm d}t <\infty \vert \mathcal{ F}_{ U_{ k}}\right)\right).\]
Then, using \eqref{hyp:F_bound} and Fubini Theorem
\begin{align*}
\int_{ U_{ k}}^{U_{ k}+ \varepsilon} \Lambda_{ k}^{ A}(t) {\rm d}t &\leq c_{ 2} \varepsilon +  \frac{ c_{ 1}}{ N}  \left\lbrace\left\Vert h_{ 1} \right\Vert_{ 1}  + \left\Vert h_{ 2} \right\Vert_{ 1}  \right\rbrace k < \infty.
\end{align*}
Hence, almost surely for all $i=1, \ldots, N_{ A}$, $\pi_{ i}(\mathcal{ R}_{ \varepsilon}^{ A})<\infty$ and 
the same reasoning applies to $\pi_{ i}(\mathcal{ R}_{ \varepsilon}^{ B})$ for all $i=N_{ A}+1, \ldots, N$. This gives that almost surely for all $i=1, \ldots, N$, $U_{ k+1}^{ (i)}>U_{ k}$ and thus $U_{ k+1}>U_{ k}$. Again by independence of the $(\pi_{ i})_{ i=1, \ldots, N}$, the $U_{ k+1}^{ (i)}, i=1, \ldots, N$ are all distinct and 
$I_{ k+1}:= \left\lbrace i\in \left\lbrace1, \ldots, N\right\rbrace, U_{ k+1}^{ (i)}= U_{ k+1}\right\rbrace$
is uniquely defined. By definition \eqref{eq:Hawkes}, the above construction yields that $U_{ k+1}$ is necessarily the $k$th atom of $(Z_{ t}^{i})_{ i=1, \ldots, N}$ and that it concerns the process $Z^{i}$ for $i=I_{ k+1}$. This completes the proof of the existence of a pathwise unique Hawkes process as defined in \eqref{eq:Hawkes}. We now establish absence of explosion of this process.
\paragraph{Stochastic domination.} Introduce the following linear processes:
\begin{equation}
\label{eq:Zlin}
\begin{cases}
Z_{t}^{i, \text{lin}}&= \int_{0}^{t}\int_{0}^{\infty}\mathbf{1}_{z\le \lambda_{  s}^{i, \text{lin}}}\pi_{i}(\rmd s,\rmd z),\ i=1, \ldots, N_{ A},\\
Z_{t}^{i, \text{lin}}&= \int_{0}^{t}\int_{0}^{\infty}\mathbf{1}_{z\le \lambda_{  s}^{i, \text{lin}}} \tilde{ \pi}_{i}(\rmd s,\rmd z), \ i=N_{ A}+1,\ldots,N,
\end{cases}
\end{equation}
where (recall \eqref{hyp:F_bound} and \eqref{hyp:G_bound})
{\small\begin{equation}
\begin{cases}
\lambda_{  s}^{i, \text{lin}}=\lambda_{ s}^{A, \text{lin}}&:=  c_{ 1} \left(\frac{ 1}{ N} \displaystyle{\sum_{ j=1}^{ N_{ A}}}\int_{ 0}^{s^{ -}} h_{ 1}(s-u) {\rmd Z_{u}^{j, \text{lin}}} + \frac{ 1}{ N} \displaystyle{\sum_{ j=N_{ A}+1}^{ N}}\int_{ 0}^{s^{ -}} h_{ 2}(s-u) {\rmd Z_{u}^{j, \text{lin}}}\right) + c_{ 2},\ 1\le i\le  N_{ A},\\
\lambda_{  s}^{i, \text{lin}}=\lambda_{ s}^{B, \text{lin}}&:= c_{ 6} \left(\frac{ 1}{ N}  \underset{N_{A}+1\le j\le N}{\sum} \int_{ 0}^{s^{ -}} h_{ 3}(s-u) {\rmd Z_{u}^{j, \text{lin}}} + \frac{ 1}{ N}  \underset{1\le j\le N_{A}}{\sum}\int_{ 0}^{s^{ -}} h_{ 4}(s-u) {\rmd Z_{u}^{j, \text{lin}}}\right)+ c_{ 7}, \  N_{ A}+1\le i\le N.
\end{cases}
\end{equation}}
Existence and pathwise uniqueness of a process $ \left(Z^{i, \text{lin}}\right)_{ i=1, \ldots, N}$ such that $\sum_{ i=1}^{ N} \mathbf{ E} \left[ \sup_{ t\in [0, T]}Z_{ t}^{i,\text{lin}}\right]<\infty$ for all $T>0$ is a direct consequence of well-known existing results (see e.g. \cite{MR3449317}).

We prove here that the point processes $Z^{i}, i=1, \ldots, N$ are stochastically dominated by $Z^{i,\text{lin}}, i=1, \ldots, N$ given in \eqref{eq:Zlin}. More precisely, we show that for all $i=1, \ldots, N$, each atom of $Z^{i}$ is also an atom of $Z^{i, \text{lin}}$. We proceed by recursion: we claim that (recall \eqref{eq:model})
\begin{equation}
\label{eq:dom_lambda1}
\begin{cases}
\lambda_{  t}^{A} &\leq \lambda_{  t}^{A,\text{lin}},\\
\lambda_{  t}^{B} &\leq \lambda_{  t}^{B, \text{lin}},
\end{cases}\ t\in (0, U_{ 1}].
\end{equation}
Indeed, by construction of $Z^{i}, i=1, \ldots, N$, there are no atoms before $U_{ 1}$ so that, for $t\in \left(0, U_{ 1}\right)$, $\lambda_{  t}^{A} = F(0, 0) \leq c_{ 2} \leq \lambda_{  t}^{A, \text{lin}}$ and $\lambda_{  t}^{B} = G(0, 0) \leq c_{ 7} \leq \lambda_{  t}^{B, \text{lin}}$. Suppose that $I_{ 1}\in \left\lbrace1, \ldots, N_{ A}\right\rbrace$ (recall the definition of $I_{ 1}$ in \eqref{eq:I1}), the other case $I_{ 1}\in \left\lbrace N_{ A}+1, \ldots, N\right\rbrace$ is treated similarly. By definition of $U_{ 1}$, there exists an atom $(U_{ 1}, z)$ of $ \pi_{ I_{ 1}}$  with $0\leq z \leq \lambda_{ U_{ 1}}^{A}$ so that by \eqref{eq:dom_lambda1}, $0\leq z \leq \lambda_{ U_{ 1}}^{A,\text{lin}}$ so that by \eqref{eq:Zlin}, $U_{ 1}$ is also an atom of $Z_{ I_{ 1}}^{ \text{lin}}$.

Suppose now that $U_{ 1}, \ldots , U_{ k}$ (which are by definition atoms of $Z_{ I_{ 1}}, \ldots, Z_{ I_{ k}}$ respectively) are also atoms of $Z_{ I_{ 1}}^{ \text{lin}}, \ldots, Z_{ I_{ k}}^{ \text{lin}}$.  In particular, for $t\in [U_{ k}, U_{ k+1})$, recalling
\begin{align*}
\lambda_{  t}^{A}&= \Lambda_{ k}^{ A}(t) 
\leq c_{ 1} \left\lbrace \frac{ 1}{ N} \int_{ (0, U_{ k}]} h_{ 1}(t-s) {\rm d}Z_{ A, s}^{ (k)} + \frac{ 1}{ N}\int_{ (0, U_{ k}]} h_{ 2}(t-s) {\rm d}Z_{ B, s}^{ (k)}\right\rbrace + c_{ 2} \leq \lambda_{  t}^{A,\text{lin}},
\end{align*}
where we used \eqref{hyp:F_bound} and the fact that the number of atoms among $U_{ 1},\ldots, U_{ k}$ that concern population $A$ (resp. $B$) is the same for both processes $(Z^{i})_{ i=1, \ldots, N}$ and $(Z^{i, \text{lin}})_{ i=1, \ldots, N}$. The same reasoning proves that $ \lambda_{  t}^{B} \leq \lambda_{  t}^{B, \text{lin}}$ for $t\in [U_{ k}, U_{ k+1})$. Hence, this proves again that $U_{ k+1}$ (which is an atom of $Z_{ I_{ k+1}}$) is also an atom of $Z_{ I_{ k+1}}^{ \text{lin}}$, which proves the claim.
\paragraph{Absence of explosion}A direct consequence of the previous claim is  $ \mathbf{ E} \left[\sup_{ t\in[0, T]}Z_{ t}^{i}\right] \leq \mathbf{ E} \left[\sup_{ t\in[0, T]}Z_{ t}^{i, \text{lin}}\right]$ so that (recall \eqref{eq:Z})
\begin{equation}
\label{eq:apriori_control_Z}
 \mathbf{ E} \left[\sup_{ t\in [0, T]} Z_{t}\right] \leq C_{ T}, T>0
\end{equation}
for some constant $C_{ T}>0$. This yields in particular that almost surely $ U_{ \infty}:=\lim_{ k\to\infty} U_{ k}= +\infty$. Indeed, set $ \Omega_{ M}:= \left\lbrace U_{ \infty} \leq M\right\rbrace$ for all $M>0$. Then, $ \mathbf{ E} \left[Z_{ U_{ k}} \mathbf{ 1}_{ \Omega_{ M}}\right]=k \mathbf{ P} \left(\Omega_{ M}\right)$ on the one hand and $ \mathbf{ E} \left[Z_{ U_{ k}} \mathbf{ 1}_{ \Omega_{ M}}\right]\leq \mathbf{ E} \left[Z_{ M}\right]\leq C_{ M}$ on the other hand, by \eqref{eq:apriori_control_Z}. Hence, $ \mathbf{ P} \left( \Omega_{ M}\right)\leq \frac{ C_{ M}}{ k}$, for all $k\geq0$. Letting $k\to\infty$, $ \mathbf{ P} \left(\Omega_{ M}\right)=0$ for all $M>0$ which gives that $ \mathbf{ P} \left( U_{ \infty}=\infty\right)=1$.

Uniqueness of a trajectory follows from the inductive procedure. This concludes the proof of Proposition~\ref{prop:ExistUniq}.

\subsection{Proof of well-posedness of the mean-field limit}
\label{sec:proof_WP_MF}
We prove here Proposition~\ref{prop:mAB}. 
\subsubsection{Well-posedness of \eqref{eq:mtAB}}
\label{sec:proof_WP_mean}
Consider first the dynamics of the mean-value $(m_{ t}^{ A}, m_{ t}^{ B})$ given by \eqref{eq:mtAB}. Let us begin with an a priori estimate: for any solution $(m_{ A}, m_{ B})$ that is $ \mathcal{ C}^{ 1}$ on $[0, T]$ (if it exists), such a solution satisfies, for $ \lambda_{  t}^{A}:= (m_{  t}^{  A})^{\prime}$, $  \lambda_{ t}^{B}:= (m_{  t}^{  B})^{\prime}$,
\begin{equation*}
\begin{cases}
\lambda_{  t}^{A}&=F \left( \alpha\int_{ 0}^{t} h_{ 1}(s-u)  \lambda_{ u}^{A} {\rm d}u,(1- \alpha) \int_{ 0}^{t} h_{ 2}(t-u)  \lambda_{ u}^{B} {\rm d}u\right),\\
 \lambda_{ t}^{B}&=G \left( (1-\alpha)\int_{ 0}^{t} h_{ 3}(s-u)  \lambda_{ u}^{B} {\rm d}u, \alpha \int_{ 0}^{t} h_{ 4}(t-u)  \lambda_{ u}^{A} {\rm d}u\right).
\end{cases}
\end{equation*}
In particular $  \lambda_{ u}^{A}\geq0$ and $  \lambda_{ u}^{B}\geq0$ for all $u\in [0, T]$. Using \eqref{hyp:F_bound} and \eqref{hyp:G_bound}, we  obtain that, for $t\in [0, T]$
\begin{align*}
\lambda_{  t}^{A}&\leq  c_{ 2}+ c_{ 1} \left\lbrace\alpha\int_{ 0}^{t} h_{ 1}(t-v) \lambda_{ v}^{A} {\rm d}v + (1-\alpha)\int_{ 0}^{t} h_{ 2}(t-v) \lambda_{ v}^{B} {\rm d}v\right\rbrace,
\\
 \lambda_{ t}^{B}&\leq  c_{ 7}+ c_{ 6} \left\lbrace (1-\alpha)\int_{ 0}^{t} h_{ 3}(t-v) \lambda_{ v}^{B} {\rm d}v + \alpha\int_{ 0}^{t} h_{ 4}(t-v) \lambda_{ v}^{A} {\rm d}v\right\rbrace.
\end{align*}
By summation, we obtain, for $ \Lambda(t):= \lambda_{  t}^{A}+  \lambda_{ t}^{B}$, that for some constant $C>0$
\begin{equation}
\Lambda(t) \leq C \left(1+ \int_{ 0}^{t} H(t-v) \Lambda(v) {\rm d}v\right).
\end{equation}
An application of \cite{MR3449317}, Lemma~23 (i) leads to the following a priori bound: there exists of a constant $C_{ 0}>0$, depending on $T, F, G$ and $H$ such that
\begin{equation}
\label{eq:uniform_bound_lambdaA}
\sup_{ t\in[0, T]} \left\lbrace \lambda_{  t}^{A} +  \lambda_{ t}^{B}\right\rbrace= \sup_{ t\in[0, T]} \left\lbrace (m_{  t}^{  A})^{\prime} + (m_{  t}^{  B})^{\prime}\right\rbrace \leq C_{ 0}.
\end{equation}

We now turn to the proof of uniqueness of a solution to \eqref{eq:mtAB}: consider two solutions $(m_{ A}, m_{ B})$ and $(\tilde{ m}^{ A}, \tilde{ m}^{ B})$. Set $(v_{ A}(t), v_{ B}(t)):= \left( \int_{ 0}^{t} \left\vert {\rm d}(m_{ u}^{A} - \tilde{ m}_u^{ A}) \right\vert, \int_{ 0}^{t} \left\vert {\rm d}(m_{ u}^{B} - \tilde{ m}_{u}^{B}) \right\vert\right)$ and $v(t):= v_{ A}(t) + v_{ B}(t)$. Then,
\begin{align*}
v_{ A}(t)&= \int_{ 0}^{t} \left\vert {\rm d}(m_{ u}^{A} - \tilde{ m}_u^{ A}) \right\vert
\leq  \int_{ 0}^{t} \bigg\vert F \left( \alpha\int_{ 0}^{u} h_{ 1}(u-v) {\rm d} m_{  v}^{A},(1- \alpha) \int_{ 0}^{u} h_{ 2}(u-v) {\rm d} m_{  v}^{B} \right)\\
& \hspace{2cm}- F \left( \alpha\int_{ 0}^{u} h_{ 1}(u-v) {\rm d} m_{  v}^{A},(1- \alpha) \int_{ 0}^{u} h_{ 2}(u-v) {\rm d} \tilde{m}_{  v}^{B} \right) \bigg\vert {\rm d}u\\
&+ \int_{ 0}^{t} \bigg\vert F \left( \alpha\int_{ 0}^{u} h_{ 1}(u-v) {\rm d} m_{  v}^{A},(1- \alpha) \int_{ 0}^{u} h_{ 2}(u-v) {\rm d} \tilde{m}_{  v}^{B} \right)\\
& \hspace{2cm}- F \left( \alpha\int_{ 0}^{u} h_{ 1}(u-v) {\rm d} \tilde{m}_{  v}^{A},(1- \alpha) \int_{ 0}^{u} h_{ 2}(u-v) {\rm d} \tilde{m}_{  v}^{B} \right) \bigg\vert {\rm d}u
\end{align*}
Using \eqref{hyp:F_Lip_x} and \eqref{hyp:F_Lip_y}, we obtain
\begin{align*}
v_{ A}(t)&\leq  (1- \alpha) \int_{ 0}^{t} \left(c_{ 4}\alpha\int_{ 0}^{u} h_{ 1}(u-v) {\rm d} m_{  v}^{A}+c_{ 5}\right) \left\vert  \int_{ 0}^{u} h_{ 2}(u-v) ({\rm d} m_{  v}^{B}- {\rm d} \tilde{m}_{  v}^{B}) \right\vert {\rm d}u\\
&\hspace{2cm}+ c_{ 3} \alpha\int_{ 0}^{t} \left\vert \int_{ 0}^{u} h_{ 1}(u-v) ({\rm d} m_{  v}^{A}- {\rm d} \tilde{m}_{  v}^{A}) \right\vert {\rm d}u.
\end{align*}
Using here the uniform bound \eqref{eq:uniform_bound_lambdaA} to control the first term above, we obtain
\begin{align}
v_{ A}(t)&\leq  (1- \alpha) \left(c_{ 4}\alpha C_{ 0} \left\Vert h_{ 1} \right\Vert_{ 1, T}+c_{ 5}\right) \int_{ 0}^{t} h_{ 2}(t-u) v_{ B}(u) {\rm d}u + c_{ 3} \alpha \int_{ 0}^{t} h_{ 1}(t-u) v_{A}(u) {\rm d}u. \label{aux:gron_vA}
\end{align}
Similarly, using the Lipschitz continuity of $G$ \eqref{hyp:G_Lip}, we obtain secondly
\begin{align}
v_{ B}(t)&= \int_{ 0}^{t} \left\vert {\rm d}(m_{ u}^{B} - \tilde{ m}_u^{ B}) \right\vert
\leq c_{ 6}(1-\alpha)\int_{ 0}^{t} h_{ 3}(t-u) v_{ B}(u) {\rm d}u  + c_{ 6}\alpha \int_{ 0}^{t} h_{ 4}(t-u) v_{ A}(u) {\rm d}u. \label{aux:gron_vB}
\end{align}
Gathering \eqref{aux:gron_vA} and \eqref{aux:gron_vB} and applying Gronwall's inequality (recall \cite{MR3449317}, Lemma~23, (i)) gives uniqueness of a solution to \eqref{eq:mtAB}.

We now turn to the existence part. We follow here closely the argument of \cite{MR3449317}, Lemmas~23 and~24: consider the sequence $(m_{ A}^{ 0}, m_{ B}^{ 0}) \equiv (0, 0)$ and for $n\geq0$, 
\begin{equation*}
\begin{cases}
m_{  t}^{A, n+1}&= \int_{0}^{t}{F \left( \alpha\int_{ 0}^{s} h_{ 1}(s-u) {\rm d} m_{  u}^{A, n},(1- \alpha) \int_{ 0}^{s} h_{ 2}(s-u) {\rm d} m_{  u}^{B, n} \right)}\rmd s,\\
m_{  t}^{B, n+1}&= \int_{0}^{t}  G \left((1- \alpha) \int_{ 0}^{s} h_{ 3}(s-u) {\rm d} m_{  u}^{B, n}, \alpha\int_{ 0}^{s} h_{ 4}(s-u) {\rm d} m_{  u}^{A, n}\right) \rmd s,
\end{cases}
\end{equation*}
Using \eqref{hyp:F_bound} and \eqref{hyp:G_bound}, we obtain that
\begin{align*}
m_{  t}^{A, n+1}&\leq c_{ 2}t+ c_{ 1} \left\lbrace \alpha\int_{ 0}^{t} h_{ 1}(t-u)m_{  u}^{A,n} {\rm d}u + (1- \alpha) \int_{ 0}^{t} h_{ 2}(t-u) m_{  u}^{B, n} {\rm d}u\right\rbrace,\\
m_{  t}^{B, n+1}&\leq c_{ 7}t + c_{ 6}   \left\lbrace(1- \alpha) \int_{ 0}^{t} h_{ 3}(t-u) m_{  u}^{B, n} {\rm d}u + \alpha\int_{ 0}^{t} h_{ 4}(t-u) m_{  u}^{A, n} {\rm d}u\right\rbrace,
\end{align*}
so that $M^{ n}(t):= m_{  t}^{A, n} + m_{  t}^{B, n}$ satisfies, for some constant $C>0$
\begin{align*}
M^{ n+1}(t) \leq Ct + C \int_{ 0}^{t} H(t-u) M^{ n}(u) {\rm d}u.
\end{align*}
Applying \cite{MR3449317}, Lemma~23, (iii), we obtain that $\sup_{ t\in [0, T]}\sup_{ n\geq0} M^{ n}(t)$ is bounded by some constant $C>0$. Using this uniform estimate and the hypotheses on $F$ and $G$ gives, in a similar way to the uniqueness part, that for $ \delta_{ t}^{ n}:= \int_{ 0}^{t} \left\vert {\rm d}(m_{ u}^{A, n+1}- m_{ u}^{A,n}) \right\vert + \int_{ 0}^{t} \left\vert {\rm d}(m_{ u}^{B, n+1}- m_{ u}^{B,n}) \right\vert$,
\begin{equation}
\delta_{ t}^{ n+1} \leq C \int_{ 0}^{t} H(t-u) \delta_{ u}^{ n} {\rm d}u.
\end{equation}
By \cite{MR3449317}, Lemma~23, (ii), this gives that $\sum_{ n\geq0} \delta_{ t}^{ n}< \infty$, which implies the existence of a locally bounded trajectory $(m_{  t}^{A}, m_{  t}^{B})$ such that $\lim_{ n\to\infty} \int_{ 0}^{t} \left\vert {\rm d}(m_{  u}^{A} - m_{  u}^{A, n}) \right\vert + \int_{ 0}^{t} \left\vert {\rm d}(m_{  u}^{B} - m_{  u}^{B,n}) \right\vert=0$. This provides existence of a solution to \eqref{eq:mtAB}. It remains to check that it is $ \mathcal{ C}^{ 1}$. We easily see that $(m_{ A}^{ n}, m_{ B}^{ n})$ is $ \mathcal{ C}^{ 1}$, by induction, with a bound on $ \sup_{ u\in [0, T]} \left\vert (m_{ u}^{ A, n})^{\prime} \right\vert + \left\vert (m_{ u}^{ B, n } )^{\prime}\right\vert $ that is uniform in $n$. Using this uniform bound and the properties of $F$ and $G$, we directly see that $((m^{  A,n, })^{\prime}, (m^{  B, n})^{\prime})$ is Cauchy for the uniform convergence on $[0, T]$, which gives the desired conclusion.

\subsubsection{Well-posedness of the mean-field process}
\label{sec:proof_WP_Zbar}
Now turn to the mean-field limit \eqref{eq:barZ_gen}.
Suppose that such a process $(\bar Z_{t}^{A}, \bar Z_{t}^{B})$ exists. Then, taking expectation in \eqref{eq:barZ_gen}, we see that $(\mathbb{ E}(\bar Z_{t}^{A}), \mathbb{ E}(\bar Z_{t}^{B}))$ solves \eqref{eq:mtAB} and is thus uniquely defined, by Proposition~\ref{prop:mAB}. Hence, any solution to \eqref{eq:barZ_gen} solves necessarily
\begin{equation}
\label{eq:barZm}
\begin{cases}
\bar Z_{t}^{A}&= \int_{0}^{t}\int_{0}^{\infty}\mathbf{1}_{z\le F \left( \alpha\int_{ 0}^{s} h_{ 1}(s-u) {\rm d} m_{  u}^{A}, (1- \alpha) \int_{ 0}^{s} h_{ 2}(s-u) {\rm d} m_{ u}^{B}\right)}\pi(\rmd s,\rmd z),\\
\bar Z_{t}^{B}&= \int_{0}^{t}\int_{0}^{\infty}\mathbf{1}_{z\le  G\left((1- \alpha) \int_{ 0}^{s} h_{ 3}(s-u) {\rm d} m_{B,u}, \alpha\int_{ 0}^{s} h_{ 4}(s-u) {\rm d} m_{ u}^{A}\right)}\tilde{ \pi}(\rmd s,\rmd z),
\end{cases}
\end{equation}
which gives uniqueness. Concerning existence, consider the solution $(\bar{Z}_{ A,t}, \bar Z_{t}^{B})$ to \eqref{eq:barZm}. Then $(\mathbb{ E}(\bar Z_{t}^{A}), \mathbb{ E}(\bar Z_{t}^{B}))$ verifies
\begin{equation*}
\label{eq:mtABZ}
\begin{cases}
\mathbb{ E} \left(\bar Z_{ t}^{A}\right)&= \int_{0}^{t}{F\left( \alpha\int_{ 0}^{s} h_{ 1}(s-u) {\rm d} m_{  u}^{A}, (1- \alpha) \int_{ 0}^{s} h_{ 2}(s-u) {\rm d} m_{  u}^{B} \right)}\rmd s,\\
\mathbb{ E}(\bar Z_{t}^{B})&= \int_{0}^{t} \left\lbrace G \left((1- \alpha) \int_{ 0}^{s} h_{ 3}(s-u) {\rm d} m_{  u}^{B}, \alpha\int_{ 0}^{s} h_{ 4}(s-u) {\rm d} m_{  u}^{A}\right)\right\rbrace \rmd s,
\end{cases}
\end{equation*}
so that $(\mathbb{ E}(\bar Z_{t}^{A}), \mathbb{ E}(\bar Z_{t}^{B}))=(m_{  t}^{A}, m_{  t}^{B})$ since $(m^{ A}, m^{ B})$ solves \eqref{eq:mtAB}. This gives existence and Proposition~\ref{prop:mAB} is proven.

\subsection{Proof of propagation of chaos}
\label{sec:proof_prop_chaos}
We prove here Proposition~\ref{prop:conv}. Denote by
\begin{align*}
\Lambda_{ N, t}^{ (A, 1)}&:= \frac{ 1}{ N}  \underset{1\le j\le N_{A}}{\sum}\int_{ 0}^{t^{ -}} h_{ 1}(t-u) {\rmd Z_{u}^{j}},\ &\Lambda_{ N, t}^{ (B, 2)}&:=\frac{ 1}{ N}  \underset{N_{A}+1\le j\le N}{\sum} \int_{ 0}^{t^{ -}} h_{ 2}(t-u) {\rmd Z_{u}^{j}},\\
\Lambda_{ N, t}^{ (B, 3)}&:=\frac{ 1}{ N}  \underset{N_{A}+1\le j\le N}{\sum} \int_{ 0}^{t^{ -}} h_{ 3}(t-u) {\rmd Z_{u}^{j}},\ &\Lambda_{ N, t}^{ (A, 4)}&:=\frac{ 1}{ N}  \underset{1\le j\le N_{A}}{\sum}\int_{ 0}^{t^{ -}} h_{ 4}(t-u) {\rmd Z_{u}^{j}},
\end{align*}
so that \eqref{eq:Hawkes} becomes
\begin{equation*}
\begin{cases}
Z_{t}^{i}&= \int_{0}^{t}\int_{0}^{\infty}\mathbf{1}_{z\le F \left(\Lambda_{ n, s}^{ (A, 1)}, \Lambda_{ N, s}^{ (B, 2)}\right)}\pi_{i}(\rmd s,\rmd z),\ 1\le i\le N_{ A},\\
Z_{t}^{i}&= \int_{0}^{t}\int_{0}^{\infty}\mathbf{1}_{z\le G \left(\Lambda_{ N, s}^{ (B, 3)}, \Lambda_{ N, s}^{ (A, 4)}\right)} \tilde{ \pi}_{i}(\rmd s,\rmd z), \ N_{ A}+1\le i\le N.
\end{cases}
\end{equation*}
Denote also by 
\begin{align*}
\Lambda_{t}^{ (A, 1)}&:= \alpha\int_{ 0}^{t} h_{ 1}(t-u) {\rm d} \mathbb{ E} (\bar Z_{u}^{A}),\ 
&\Lambda_{t}^{ (B, 2)}&:=(1- \alpha) \int_{ 0}^{t} h_{ 2}(t-u) {\rm d} \mathbb{ E} \left(\bar Z_{u}^{B}\right),\\
\Lambda_{t}^{ (B, 3)}&:= (1- \alpha) \int_{ 0}^{t} h_{ 3}(t-u) {\rm d} \mathbb{ E} \left(\bar Z_{u}^{B}\right),\
&\Lambda_{t}^{ (A, 4)}&:= \alpha\int_{ 0}^{t} h_{ 4}(t-u) {\rm d} \mathbb{ E} \left(\bar Z_{ u}^{A}\right),
\end{align*}
so that the limiting system \eqref{eq:barZ_gen} becomes
\begin{equation*}
\begin{cases}
\bar Z_{t}^{A}&= \int_{0}^{t}\int_{0}^{\infty}\mathbf{1}_{z\le F\left(\Lambda_{s}^{ (A, 1)}, \Lambda_{s}^{ (B, 2)}\right)}\pi(\rmd s,\rmd z),\\
\bar Z_{t}^{B}&= \int_{0}^{t}\int_{0}^{\infty}\mathbf{1}_{z\le G \left(\Lambda_{s}^{ (B, 3)}, \Lambda_{s}^{ (A, 4)}\right)}\tilde{ \pi}(\rmd s,\rmd z).
\end{cases}
\end{equation*}
Let $((\bar Z^{ i,A})_{ i=1, \ldots, N_{ A}}, (\bar Z^{ i,B})_{ i=N_{ A}+1,\ldots, N})$ be the coupling driven by the same Poisson point processes $\pi_{ i}$, $i=1, \ldots, N_{ A}$ and $ \tilde{ \pi}_{ i}$, $i=N_{ A}+1, \ldots, N$ as for \eqref{eq:Hawkes}. With a slight abuse of notations, we use the same $\bar Z^{i}$ for $\bar Z^{ i,A}$ or $ \bar Z^{ i,B}$ depending on wether $i=1, \ldots, N_{ A}$ or $i=N_{ A}+1, \ldots, N$. For all $i=1, \ldots, N$, denote by $ \Delta_{ N}^{ i}(t):= \int_{ 0}^{t} \left\vert d \left( Z^{i}_{ u} - \bar Z^{i}_{  u}\right) \right\vert$ the total variation distance between processes $Z^{i}$ and $ \bar Z^{i}$ and $\delta_{N}(t):= \mathbf{ E} \left[ \Delta_{ N}^{ i}(t)\right]$. Firstly, for $i=1, \ldots, N_{ A}$:
\begin{align*}
\Delta_{ N}^{ i}(t)&= \int_{ 0}^{t} \left\vert \int_{0}^{\infty}\mathbf{ 1}_{ z \leq F \left(\Lambda_{ n, s}^{ (A, 1)}, \Lambda_{ N, s}^{ (B, 2)}\right)} - \mathbf{ 1}_{ z\leq F\left(\Lambda_{s}^{ (A, 1)}, \Lambda_{s}^{ (B, 2)}\right)} \right\vert\pi_{ i}({\rm d}s, {\rm d}z)
\end{align*}
so that
\begin{align}
\delta_{ N}(t)&\leq \int_{ 0}^{t} \mathbf{ E} \left[\left\vert  F \left(\Lambda_{ N, s}^{ (A, 1)}, \Lambda_{ N, s}^{ (B, 2)}\right) - F\left(\Lambda_{s}^{ (A, 1)}, \Lambda_{N, s}^{ (B, 2)}\right) \right\vert\right]{\rm d}s \label{aux:deltaF_1}\\ &+ \int_{ 0}^{t} \mathbf{ E} \left[\left\vert  F \left(\Lambda_{ s}^{ (A, 1)}, \Lambda_{ N, s}^{ (B, 2)}\right) - F\left(\Lambda_{s}^{ (A, 1)}, \Lambda_{s}^{ (B, 2)}\right) \right\vert\right]{\rm d}s \label{aux:deltaF_2}.
\end{align}
Using \eqref{hyp:F_Lip_x} in \eqref{aux:deltaF_1} and \eqref{hyp:F_Lip_y} in \eqref{aux:deltaF_2}, we obtain
\begin{align}
\label{aux:deltaNi1}
\delta_{ N}(t)&\leq c_{ 3}\int_{ 0}^{t} \mathbf{ E} \left[\left\vert \Lambda_{ N, s}^{ (A, 1)} - \Lambda_{s}^{ (A, 1)} \right\vert\right]{\rm d}s+ \int_{ 0}^{t} \left(c_{ 4} \Lambda_{ s}^{ (A, 1)} + c_{ 5}\right)\mathbf{ E} \left[ \left\vert  \Lambda_{ N, s}^{ (B, 2)} - \Lambda_{s}^{ (B, 2)} \right\vert\right]{\rm d}s.
\end{align}
Concentrate on the first term in the sum above: one has
\begin{align*}
\int_{ 0}^{t} \mathbf{ E} \left[\left\vert \Lambda_{ N, s}^{ (A, 1)} - \Lambda_{s}^{ (A, 1)} \right\vert\right]{\rm d}s
&\leq\int_{ 0}^{t} \mathbf{ E} \left[\left\vert \frac{ 1}{ N}  \underset{1\le j\le N_{A}}{\sum}\int_{ 0}^{s^{ -}} h_{ 1}(s-u) ({\rmd Z_{u}^{j}}- {\rm d}\bar Z_{ u}^{j}) \right\vert\right]{\rm d}s\\
&+\int_{ 0}^{t} \mathbf{ E} \left[\left\vert \frac{ 1}{ N}  \underset{1\le j\le N_{A}}{\sum}\int_{ 0}^{s^{ -}} h_{ 1}(s-u) {\rmd \bar Z_{u}^{j}} - \alpha\int_{ 0}^{s} h_{ 1}(s-u) {\rm d} \mathbb{ E} (\bar Z_{u}^{A}) \right\vert\right]{\rm d}s,\\
&\leq\alpha\int_{ 0}^{t} h_{ 1}(t-u) \delta_{ N}(u) {\rm d}u+\int_{ 0}^{t} \mathbf{ E} \left[\left\vert \frac{ \alpha}{ N_{ A}}  \underset{1\le j\le N_{A}}{\sum} \left(X_{ j}(s)- \mathbf{ E} \left[X_{ j}(s)\right] \right) \right\vert\right]{\rm d}s,
\end{align*}
where $X_{ j}(s):= \int_{ 0}^{s^{ -}} h_{ 1}(s-u) {\rmd \bar Z_{u}^{j}}$, $j=1, \ldots, N_{ A}$ is a collection of i.i.d. random variables, with $ Var \left[X_{ j}(s)\right]= \int_{ 0}^{s} h_{ 1}(s-u)^{ 2}  \lambda_{ u}^{A} {\rm d}u$, so that the Cauchy-Schwarz inequality leads to  \begin{align*}
\int_{ 0}^{t} \mathbf{ E} \left[\left\vert \frac{ \alpha}{ N_{ A}}  \underset{1\le j\le N_{A}}{\sum} \left(X_{ j}(s)- \mathbf{ E} \left[X_{ j}(s)\right] \right) \right\vert\right]{\rm d}s &\leq \frac{ \alpha}{ N_{ A}^{ \frac{ 1}{ 2}}} \int_{ 0}^{t}\left(\int_{ 0}^{s} h_{ 1}(s-u)^{ 2}  \lambda_{ u}^{A} {\rm d}u\right)^{ \frac{ 1}{ 2}} {\rm d}s,\\
& \leq \frac{ \alpha t\left\Vert \lambda^{ A} \right\Vert_{ \infty, t}^{ \frac{ 1}{ 2}} \left\Vert h_{ 1} \right\Vert_{ 2, t}}{ N_{ A}^{ \frac{ 1}{ 2}}} := \frac{ C_{ 1}(t)}{ \sqrt{N}},
\end{align*}  where $C_{1}(t)\le { \sqrt{\alpha} \left\Vert \lambda^{ A} \right\Vert_{ \infty}^{ \frac{ 1}{ 2}} \left\Vert h_{ 1} \right\Vert_{ 2}} t$.
We obtain finally that 
\begin{align*}
\int_{ 0}^{t} \mathbf{ E} \left[\left\vert \Lambda_{ N, s}^{ (A, 1)} - \Lambda_{s}^{ (A, 1)} \right\vert\right]{\rm d}s &\leq\alpha \int_{ 0}^{t} h_{ 1}(t-u) \delta_{ N}(u) {\rm d}u+ \frac{ C_{ 1}(t)}{ \sqrt{N}}
\end{align*}
Note secondly that
\begin{align*}
\Lambda_{ s}^{ (A, 1)} = \alpha \int_{ 0}^{t} h_{ 1}(t-u)  \lambda_{ u}^{A} {\rm d}u& \leq \alpha\left\Vert \lambda^{ A} \right\Vert_{ \infty, T} \left\Vert h_{ 1} \right\Vert_{ 1, t}:= C_{ 2}(t),
\end{align*} where $C_{2}(t)\le \kappa_{1}\left\Vert \lambda^{ A} \right\Vert_{ \infty}$.
Hence, proceeding similarly, there exists $C_{ 3}(t)>0$ such that
\begin{align*}
\int_{ 0}^{t} \left(c_{ 4} \Lambda_{ s}^{ (A, 1)} + c_{ 5}\right)\mathbf{ E} \left[ \left\vert  \Lambda_{ N, s}^{ (B, 2)} - \Lambda_{s}^{ (B, 2)} \right\vert\right]{\rm d}s& \leq (c_{ 4}C_{ 2}(t)+c_{ 5}) \left\lbrace\int_{ 0}^{t} (1-\alpha)h_{ 2}(t-u) \delta_{ N}(u) {\rm d}u+ \frac{ C_{ 3}(t)}{ \sqrt{N}}\right\rbrace,
\end{align*} where $C_{3}(t)\le (1-\alpha)\|\lambda^{B}\|_{\infty}^{\frac12}\|h_{2}\|_{2}$.
Collecting everything into \eqref{aux:deltaNi1}, gives the existence of a constant $C_{ 4}(t)>0$ 
 such that, for all $i=1, \ldots, N_{ A}$, 
\begin{align*}
\delta_{ N}(t)&\leq C_{ 4}(t) \int_{ 0}^{t} \left\lbrace \alpha h_{ 1}(t-u) +(1-\alpha) h_{ 2}(t-u)\right\rbrace \delta_{ N}(u) {\rm d}u+ \frac{C_{ 5}(t)}{ \sqrt{N}},
\end{align*} where $C_{4}(t)= c_{3}+c_{4}C_{2}(t)+c_{5}$ and $C_{5}(t)= c_{3}C_{1}(t)+(c_{4}C_{2}(t))+c_{5})C_{3}(t)\le C t.$
We now turn to the similar treatment of $i=N_{ A}+1, \ldots, N$: for such $i$, using the global Lipschitz continuity of $G$ \eqref{hyp:G_Lip}
\begin{align*}
\delta_{ N}(t)
&\leq c_{ 6}\int_{ 0}^{t} \mathbf{ E} \left[\left\vert  \Lambda_{ N, s}^{ (B, 3)} - \Lambda_{s}^{ (B, 3)} \right\vert + \left\vert \Lambda_{ N, s}^{ (A, 4)}- \Lambda_{s}^{ (A, 4)}\right\vert\right]{\rm d}s.
\end{align*}
Proceeding in the same way as before, we obtain, for some constant $C_{ 6}(t)>0$ 
\begin{align*}
\delta_{ N}(t)&\leq c_{ 6}\int_{ 0}^{t} \left\lbrace (1-\alpha)h_{ 3}(t-u) +\alpha h_{ 4}(t-u)\right\rbrace \delta_{ N}(u) {\rm d}u+ \frac{C_{ 6}(t)}{ \sqrt{N}},\end{align*} where $C_{6}(t)\le ( { \sqrt{1-\alpha} \left\Vert \lambda^{ B} \right\Vert_{ \infty}^{ \frac{ 1}{ 2}} \left\Vert h_{ 3} \right\Vert_{ 2}} + { \sqrt{\alpha} \left\Vert \lambda^{ A} \right\Vert_{ \infty}^{ \frac{ 1}{ 2}} \left\Vert h_{ 4} \right\Vert_{ 2}} )t$. This gives finally that, for some constants $C>0$ and $\overline C$\begin{align*}
\delta_{ N}(t)&\leq\overline C\int_{ 0}^{t} H(t-u) \delta_{ N}(u) {\rm d}u+C \frac{t}{ \sqrt{N}},
\end{align*}
for $H$ given in \eqref{def:H} and  $\overline C\le \max\{(c_{3}+c_{4}\kappa_{1}\|\lambda^{A}\|_{\infty}+c_{5})(\kappa_{1}+\kappa_{2}), c_{6}(\kappa_{3}+\kappa_{4})\}$. Observing that $\Delta^{i}_{N}(t)\ge \sup_{[0,t]}|Z^{i}_{u}-\bar Z^{i}_{u}|$ and applying Lemma~23 of \cite{MR3449317}  gives that\begin{align}
\sup_{ i=1, \ldots, N} \mathbf{ E} \left[\sup_{ u\in [0, t]} \left\vert Z_{ t}^{i} - \bar Z_{ t}^{i}\right\vert\right] \leq \frac{  C(t)}{ \sqrt{N}}.
\end{align} Additionally, if $\overline C<1$, using that $\delta_{N}$ is non decreasing it holds that $C(t)\le C t$ for some constant $C$ independent of $t$.
This concludes the proof of Proposition~\ref{prop:conv}.

\section{Proofs for the longtime dynamics}
\label{sec:proofs_longtime}
We prove here the results of Section~\ref{sec:MFL}. Recall that we consider the system \eqref{eq:lambdaAB} under Assumption~\ref{ass:Phis}. Recall  notations \eqref{eq:def_ells}. For $t\geq0$, introduce $u_{ A}(t):= \sup_{ s \geq t}  \lambda_{ s}^{A}$, $u_{ B}(t)= \sup_{ s \geq t}  \lambda_{ s}^{B}$, $v_{ A}(t):= \inf_{ s\geq t}  \lambda_{ s}^{A}$ and $v_{ B}(t):= \inf_{ s\geq t}  \lambda_{ s}^{B}$. Both $ t \mapsto u_{ A}(t)$ and $ t \mapsto u_{ B}(t)$ (resp. $ t \mapsto v_{ A}(t)$ and $ t \mapsto v_{ B}(t)$) are non-increasing (resp. non-decreasing), and we have that $ u_{ A}(t) \xrightarrow[t\to \infty]{} \bar\ell_{ A}$, $ u_{ B}(t) \xrightarrow[t\to \infty]{} \bar\ell_{ B}$, $ v_{ A}(t) \xrightarrow[t\to \infty]{} \underline\ell_{ A}$ and $ v_{ B}(t) \xrightarrow[t\to \infty]{} \underline\ell_{ B}$. 
\subsection{Some a priori estimates and first proofs}
\label{sec:proof_apriori_est}
\subsubsection{General estimates}
\label{sec:longtime_Asub}
\begin{remark}
\label{rem:ineq}
We state a bound that is constantly used in the following: for any kernel $ h\geq0$ and function $ \lambda\geq0$, for $v(u):= \inf_{ u\geq s} \lambda(u)$ and any $s\geq0$, 
\begin{align*}
\int_{ 0}^{s} h(s-u) \lambda(u) {\rm d}u 
\geq v (s/2)\int_{s/2}^{s} h(s-u) {\rm d}u= v(s/2)\int_{0}^{s/2} h(u) {\rm d}u.
\end{align*}
\end{remark}
A first a priori estimate is given by Proposition~\ref{prop:ineq_ells_1}:
\begin{proposition}
\label{prop:ineq_ells_1}
Consider Model \eqref{eq:lambdaAB} under Assumption~\ref{ass:Phis}. Then, the following inequalities on $\underline\ell_{ A}, \underline\ell_{ B},\bar \ell_{ B}\in[0, +\infty]$ are true:
\begin{align}
\underline \ell_{ B}&\geq \Phi_{ A\to B} \left(\underline\ell_{ A} \kappa_{ 4}\right),\label{eq:ulB_ulA_gen}\\
\bar \ell_{ B}&\leq  \Phi_{ B}(0) + \left\Vert \Phi_{ B} \right\Vert_{ L}\bar\ell_{ B} \kappa_{ 3} + \Phi_{ A\to B} \left(\bar \ell_{ A} \kappa_{ 4}\right).\label{eq:bound_ellB_1}
\end{align}
Moreover, in the particular case of the Model \eqref{eq:lambdaAB_lin} (Example~\ref{ex:linear}), we have
\begin{align}
\underline \ell_{ B}&\geq \mu_{ B} +  \underline\ell_{ B} \kappa_{ 3}+ \Phi_{ A\to B} \left(\underline\ell_{ A} \kappa_{ 4}\right),\label{eq:ulB_ulA}\\
\underline\ell_{ A}&\geq \left(\mu_{ A} + \kappa_{ 1} \underline \ell_{ A}\right) \Phi_{ B\to A} \left(\bar \ell_{ B} \kappa_{ 2}\right). \label{eq:ulA_blB}
\end{align}
\end{proposition}

\begin{remark}
Note that in the case $\bar \ell_{ B}=+\infty$, \eqref{eq:bound_ellB_1} is  trivial. When $ \bar \ell_{ B}= +\infty$, \eqref{eq:ulA_blB} has to be understood as the trivial inequality $ \underline{\ell}_{ A}\geq0$, whatever the value of $ \underline{ \ell}_{ A}\in[0, +\infty]$ (see the end of the proof below).
\end{remark}
\begin{proof}[Proof of Proposition~\ref{prop:ineq_ells_1}]
First we prove \eqref{eq:ulB_ulA_gen}: since $ \Phi_{ B}\geq0$ and $ \Phi_{ A\to B}$ is nondecreasing, using Remark~\ref{rem:ineq} for $ h_{ 4}$ and $ \lambda_{ A}$,  $s\geq 0$, it holds
\begin{align*}
 \lambda_{ s}^{B}&\geq  \Phi_{ A\to B} \left(\alpha  v_{ A}(s/2)\int_{ 0}^{s/2} h_{ 4}(u){\rm d}u\right).
\end{align*}
Taking $\liminf_{ s\to\infty}$ on both sides gives \eqref{eq:ulB_ulA_gen}. The treatment of \eqref{eq:ulB_ulA} in the linear case of Example~\ref{ex:linear} is similar: it suffices to take $\liminf_{ s\to\infty}$ in the following inequality: for $s\geq0$,
\begin{align*}
 \lambda_{ s}^{B}&\geq \mu_{ B} + (1- \alpha) v_{ B}(s/2)\int_{ 0}^{s/2} h_{ 3}(u) {\rm d}u + \Phi_{ A\to B} \left(\alpha  v_{ A}(s/2)\int_{ 0}^{s/2} h_{ 4}(u){\rm d}u\right).
\end{align*}
We now turn to \eqref{eq:bound_ellB_1}: by Lipschitz-continuity of $ \Phi_{ B}$ and the fact that $ \Phi_{ A\to B}$ is nondecreasing, for all $s\geq t\geq 0,$ it holds\begin{align*}
 \lambda_{ s}^{B}&\leq  \Phi_{ B}(0) + \left\Vert \Phi_{ B} \right\Vert_{ L} \left\lbrace (1- \alpha) \int_{ 0}^{t} h_{ 3}(s-u)  \lambda_{ u}^{B} {\rm d}u+ (1- \alpha) u_{ B}(t)\int_{ 0}^{s-t} h_{ 3}(u) {\rm d}u\right\rbrace \\
&\hspace{1cm}+ \Phi_{ A\to B} \left(\alpha \int_{ 0}^{t} h_{ 4}(s-u)  \lambda_{ u}^{A} {\rm d}u+ \alpha  u_{ A}(t)\int_{ 0}^{s-t} h_{ 4}(u) {\rm d}u\right).
\end{align*}
Using now that $h_{ 3}(u), h_{ 4}(u) \xrightarrow[ u\to \infty]{}0$ and that the continuous functions $ \lambda_{ A}$ and $ \lambda_{ B}$ are bounded on $[0, t]$, the dominated convergence theorem implies that for fixed $t\geq0$, $\int_{ 0}^{t} h_{ 3}(s-u)  \lambda_{ u}^{B}{\rm d} u \xrightarrow[ s\to \infty]{}0$ and $\int_{ 0}^{t} h_{ 4}(s-u)  \lambda_{ u}^{A}{\rm d} u \xrightarrow[ s\to \infty]{}0$.
Taking $\limsup_{ s\to\infty}$ with $t$ fixed yields
\begin{align}
\label{aux:ellB_VS_uB}
\bar \ell_{ B}&\leq  \Phi_{ B}(0) + \left\Vert \Phi_{ B} \right\Vert_{ L} \kappa_{ 3}  u_{ B}(t) + \Phi_{ A\to B} \left( \kappa_{ 4} u_{ A}(t)\right).
\end{align}
Taking $t\to\infty$ gives \eqref{eq:bound_ellB_1}. Concerning \eqref{eq:ulA_blB}, Remark~\ref{rem:ineq} and $ \Phi_{ B\to A}$ is non-increasing give for $s\geq t$, 
{\small\begin{align}
 \lambda_{ s}^{A}&\geq \left(\mu_{ A} + \alpha v_{ A} \left( \frac{ s}{ 2}\right)\int_{ 0}^{s/2} h_{ 1}(u){\rm d} u\right) \Phi_{ B\to A} \left((1- \alpha) \left\lbrace\int_{ 0}^{t} h_{ 2}(s-u)  \lambda_{ u}^{B} {\rm d}u + u_{ B}(t)\int_{0}^{s-t} h_{ 2}(u) {\rm d}u \right\rbrace\right) .\label{aux:lAlB_1}
\end{align}}
Since $ s \mapsto  \lambda_{ s}^{B}$ is continuous, it is bounded on each $[0, t]$ for all $t\geq0$, we get that $u_{ B}(t)=+\infty$ for some $t\geq0$ is equivalent to $u_{ B}(t)=+\infty$ for all $t\geq0$, itself equivalent to $\bar{ \ell}_{ B}=+\infty$. In this case, \eqref{aux:lAlB_1} reduces to $  \lambda_{ s}^{A}\geq0$ and taking $\limsup_{ s\to\infty}$ gives \eqref{eq:ulA_blB} in this case. Otherwise, proceeding as above, we see that for fixed $t\geq0$, $\int_{ 0}^{t} h_{ 2}(s-u)  \lambda_{ u}^{B}{\rm d} u \xrightarrow[ s\to \infty]{}0$. Taking $\liminf_{ s\to\infty}$ with $t$ fixed and then $t\to \infty$ in the previous inequality leads to \eqref{eq:ulA_blB}.
\end{proof}
Consider the following constraint (which reduces exactly to \eqref{eq:uB_in_I} in the linear Model~\eqref{eq:lambdaAB_lin})
\begin{equation}
\label{hyp:A_sub}
\kappa_{ 1} \left\Vert \Phi_{ A} \right\Vert_{ L}\Phi_{ B\to A} \left(\underline{ \ell}_{ B} \kappa_{ 2} \right)<1.
\end{equation}
\begin{proposition}
\label{prop:bounds_ells}
Consider Model \eqref{eq:lambdaAB} under Assumption~\ref{ass:Phis} and suppose \eqref{hyp:A_sub}. Then, it holds that
\begin{equation}
\label{eq:apriori_ellA}
\bar \ell_{ A} \leq \frac{ \Phi_{ A}(0)\Phi_{ B\to A} \left( \kappa_{ 2} \underline{ \ell}_{ B}\right)}{1- \kappa_{ 1} \left\Vert \Phi_{ A} \right\Vert_{ L}\Phi_{ B\to A} \left( \kappa_{ 2} \underline{ \ell}_{ B}\right)}.
\end{equation}
In particular $ 0\leq \underline \ell_{ A} \leq \bar \ell_{ A} <\infty$ implying that population $A$ is subcritical. 

Now restrict to Model \eqref{eq:lambdaAB_lin}: Inequality \eqref{eq:apriori_ellA} becomes, for the function $ \Psi_{ 1}$ given in \eqref{eq:Psi1},
\begin{equation}
\bar \ell_{ A} \leq \Psi_{ 1} \left(\underline \ell_{ B}\right)\label{eq:bound_ellA_1},
\end{equation}
and we have
\begin{equation}
\underline\ell_{ A} \geq  \Psi_{ 1} \left(\bar\ell_{ B}\right)\label{eq:bound_ellA_2}.
\end{equation}
\end{proposition}

\begin{proof}[Proof of Proposition~\ref{prop:bounds_ells}]

We first claim that, under the present hypotheses,
\begin{equation}
\label{eq:uniform_control_uA}
u_{ A}(t)<\infty,\ t\geq0.
\end{equation}
Indeed, recall that $u_{ A}(t):= \sup_{ s\geq t}  \lambda_{ s}^{A}$ and (with a slight abuse of notations) denote for $t\leq t^{ \prime}$, $u_{ A}(t, t^{ \prime})= \sup_{ s\in[t, t^{ \prime}]}  \lambda_{ s}^{A}<\infty$. Then, for $t\leq s\leq t^{ \prime}$, using Remark~\ref{rem:ineq}, the fact that $\Phi_{A}$ is Lipschitz continuous, $ \Phi_{ B\to A}$ is non-increasing and  $ z \mapsto v_{ B}(z) \int_{ 0}^{z} h_{ 2}(u) {\rm d}u$ is non-decreasing, gives
\begin{align*}
 \lambda_{ s}^{A} &\leq \left( \Phi_{ A}(0)+\alpha \left\Vert \Phi_{ A} \right\Vert_{ L} \left\lbrace  \left\Vert \lambda^{ A} \right\Vert_{ \infty, t}\int_{ 0}^{t} h_{ 1}(s-u) {\rm d}u +  u_{ A}(t, t^{\prime})\int_{ t}^{s} h_{ 1}(s-u) {\rm d}u\right\rbrace\right) \\
&\hspace{3cm}\times\Phi_{ B\to A} \left((1- \alpha) v_{ B} \left( \frac{ t}{ 2}\right) \int_{ 0}^{ \frac{ t}{ 2}} h_{ 2}(u) {\rm d}u\right),\\
&\leq\left( \Phi_{ A}(0)+\alpha \left\Vert \Phi_{ A} \right\Vert_{ L} \left\lbrace  \left\Vert \lambda^{ A} \right\Vert_{ \infty, t} \left\Vert h_{ 1} \right\Vert_{ 1}+  u_{ A}(t, t^{\prime}) \left\Vert h_{ 1} \right\Vert_{ 1}\right\rbrace\right) \Phi_{ B\to A} \left((1- \alpha) v_{ B} \left( \frac{ t}{ 2}\right) \int_{ 0}^{ \frac{ t}{ 2}} h_{ 2}(u) {\rm d}u\right).
\end{align*}
Recalling the definition of $ \kappa_{ 1}$ in \eqref{eq:kappas}
\begin{equation}
\label{aux:uAttprime}
u_{ A}(t, t^{ \prime})\leq\left( \Phi_{ A}(0)+ \kappa_{ 1} \left\Vert \Phi_{ A} \right\Vert_{ L} \left\lbrace  \left\Vert \lambda^{ A} \right\Vert_{ \infty, t} + u_{ A}(t, t^{\prime})\right\rbrace \right) \Phi_{ B\to A} \left((1- \alpha) v_{ B} \left( \frac{ t}{ 2}\right) \int_{ 0}^{ \frac{ t}{ 2}} h_{ 2}(u) {\rm d}u\right).
\end{equation}
Since $\Phi_{ B\to A}\big ((1- \alpha) v_{ B} ( \frac{ t}{ 2}) \int_{ 0}^{ \frac{ t}{ 2}} h_{ 2}(u) {\rm d}u\big) \xrightarrow[ t\to \infty]{}\Phi_{ B\to A} \left( \underline{ \ell}_{ B} \kappa_{ 2}\right)$, there exists, by \eqref{hyp:A_sub}, some $t_{ \ast}$, such that for $t\geq t_{ \ast}$, 
$\kappa_{ 1} \left\Vert \Phi_{ A} \right\Vert_{ L}\Phi_{ B\to A} \big((1- \alpha) v_{ B} ( \frac{ t}{ 2}) \int_{ 0}^{ \frac{ t}{ 2}} h_{ 2}(u) {\rm d}u\big)<1.
$ 
Thus, we deduce from \eqref{aux:uAttprime} that, for $t\geq t_{ \ast}$
\begin{equation*}
u_{ A}(t, t^{ \prime})\leq \frac{\left( \Phi_{ A}(0)+ \left\Vert \Phi_{ A} \right\Vert_{ L}\left\Vert \lambda^{ A} \right\Vert_{ \infty, t} \kappa_{ 1}\right) \Phi_{ B\to A} \left((1- \alpha) v_{ B} \left( \frac{ t}{ 2}\right) \int_{ 0}^{ \frac{ t}{ 2}} h_{ 2}(u) {\rm d}u\right)}{ 1- \kappa_{ 1} \left\Vert \Phi_{ A} \right\Vert_{ L}\Phi_{ B\to A} \left((1- \alpha) v_{ B} \left( \frac{ t}{ 2}\right) \int_{ 0}^{ \frac{ t}{ 2}} h_{ 2}(u) {\rm d}u\right)}.
\end{equation*}
Letting now $t^{ \prime}\to\infty$ in the previous inequality yields that $u_{ A}(t) < \infty$ for all $t\geq t_{ \ast}$. Noting that for $t\in [0, t_{ \ast}]$, $u_{ A}(t) \leq \max \left( \left\Vert \lambda^{ A} \right\Vert_{ \infty, t_{ \ast}}, u_{ A}(t_{ \ast})\right)$, this proves Claim \eqref{eq:uniform_control_uA}.

We now turn to the proof of \eqref{eq:apriori_ellA}. Using again Remark~\ref{rem:ineq} and the fact that $ \Phi_{ A}$ is Lipschitz and $ \Phi_{ B\to A}$ is non-increasing, we obtain that for all $s\geq t\geq 0$,
\begin{align}
 \lambda_{ s}^{A}
&\leq\left(\Phi_{ A}(0) + \left\Vert \Phi_{ A} \right\Vert_{ L} \left\lbrace \alpha \int_{ 0}^{t} h_{ 1}(s-u)  \lambda_{ u}^{A}{\rm d} u+ \alpha u_{ A}(t)\int_{0}^{s-t} h_{ 1}(u){\rm d} u\right\rbrace\right) \nonumber\\
&\hspace{3cm}\times\Phi_{ B\to A} \left((1- \alpha) v_{ B}(s/2)\int_{ 0}^{s/2} h_{ 2}(u) {\rm d}u\right)\label{aux:lA_sub}.
\end{align}
The dominated convergence theorem and  $h_{ 1}(u) \xrightarrow[ u\to \infty]{}0$ gives  $\int_{ 0}^{t} h_{ 1}(s-u)  \lambda_{ u}^{A}{\rm d} u \xrightarrow[ s\to \infty]{}0$ for any fixed $t\geq0$. Taking $\limsup_{ s\to\infty}$ in \eqref{aux:lA_sub} (for fixed $t$), we obtain,
$\bar \ell_{ A}\leq\left( \Phi_{ A}(0) + \left\Vert \Phi_{ A} \right\Vert_{ L} \kappa_{ 1}u_{ A}(t)\right) \Phi_{ B\to A} \left( \underline \ell_{ B} \kappa_{ 2}\right)$.
This together with \eqref{eq:uniform_control_uA} proves in particular that $\bar \ell_{ A}<\infty$. Taking finally $t\to\infty$, we obtain,
\begin{align}
\label{aux:ellA_bar1}
\bar \ell_{ A}&\leq\left( \Phi_{ A}(0) + \kappa_{ 1} \left\Vert \Phi_{ A} \right\Vert_{ L}\bar \ell_{ A}\right) \Phi_{ B\to A} \left( \underline \ell_{ B} \kappa_{ 2}\right).
\end{align}
Using \eqref{hyp:A_sub} in \eqref{aux:ellA_bar1} yields \eqref{eq:apriori_ellA}. In the case of Example~\ref{ex:linear} (where $ \left\Vert \Phi_{ A} \right\Vert_{ L}=1$), \eqref{eq:bound_ellA_1} is a straightforward consequence of \eqref{eq:apriori_ellA}.
Turning now to the proof of \eqref{eq:bound_ellA_2}, using hypothesis \eqref{hyp:A_sub} together with $\bar \ell_{ B}\geq \underline{ \ell}_{ B}$, we obtain that $ \kappa_{ 1}\Phi_{ B\to A} \left(\bar \ell_{ B} \kappa_{ 2}\right)<1$, which leads to \eqref{eq:bound_ellA_2}, using \eqref{eq:ulA_blB}.
\end{proof}
Based on the previous a priori estimates, we can first prove the convergence theorems in the not fully coupled cases: the proof of the easy Proposition~\ref{prop:k2_0_k4_0} being an immediate consequence of \cite{MR3449317}, Th.~10 and 11, we do not give its proof here (and anyway, its proof may be easily adapted from the arguments that follow below, so we leave the details to the reader).
\subsubsection{The case $ \kappa_{ 2}>0$ and $ \kappa_{ 4}=0$.}
\label{sec:proof_k2_pos_k4_0}
We prove here Proposition~\ref{prop:k2_pos_k4_0}. As $ \kappa_{ 4}=0$, the dynamics of $B$ reduces to a single Hawkes dynamics independent from $A$. Therefore, the longtime asymptotics of $  \lambda_{ t}^{B}$ depending on the dichotomy $\kappa_{ 3}<1$ or $ \kappa_{ 3}>1$ can easily be retrieved from \eqref{eq:bound_ellB_1} and \eqref{eq:ulB_ulA}, since $ \Phi_{ A\to B}(\cdot \kappa_{ 4})\equiv 0$ here (see also \cite{MR3449317}, Th.~10 and 11).

If $ \kappa_{ 3}<1$, we have that $ \underline{ \ell}_{ B}= \bar \ell_{ B}= \frac{ \mu_{ B}}{ 1- \kappa_{ 3}}$. Here, the condition $ \kappa_{ 1} \Phi_{ B\to A} \left( \frac{ \kappa_{ 2} \mu_{ B}}{ 1- \kappa_{ 3}}\right)<1$ precisely boils down to \eqref{hyp:A_sub}. Therefore, the results of Proposition~\ref{prop:bounds_ells} are valid: population $A$ is subcritical and, by \eqref{eq:bound_ellA_1} and \eqref{eq:bound_ellA_2}, we obtain that $ \lambda_{  t}^{A} \xrightarrow[ t\to\infty]{} \Psi_{ 1} \left( \frac{ \mu_{ B}}{ 1- \kappa_{ 3}}\right)$,  the desired result.

Now turn to the case $ \kappa_{ 1} \Phi_{ B\to A} \left( \frac{ \kappa_{ 2} \mu_{ B}}{ 1- \kappa_{ 3}}\right)>1$: this together with \eqref{eq:ulA_blB}  imply $\underline\ell_{ A}
\geq \frac{ \mu_{ A}}{ \kappa_{ 1}} + \rho\underline \ell_{ A} $, with $ \rho:=\kappa_{ 1} \Phi_{ B\to A} \left( \frac{ \kappa_{ 2} \mu_{ B}}{ 1- \kappa_{ 3}}\right) >1$. Since $ \mu_{ A}>0$ and $\underline\ell_{ A} \geq  \rho\underline \ell_{ A} $ , the only possibility is $ \underline{ \ell}_{ A}=+\infty$, which gives that $ \lambda_{  t}^{A} \xrightarrow[ t\to\infty]{}+\infty$.

Consider now the case $ \kappa_{ 3}>1$ and $ \mu_{ B}>0$. Then, $  \lambda_{ t}^{B} \xrightarrow[ t\to\infty]{}+\infty$, as $ \lim_{ +\infty}\Phi_{ B\to A}=0$, condition \eqref{hyp:A_sub} is verified and one obtains from \eqref{eq:bound_ellA_1} (recall that $ \lim_{ +\infty} \Psi_{ 1}=0$) that $ \bar \ell_{ A}=0$, which gives $ \lambda_{  t}^{A} \xrightarrow[ t\to \infty]{}0$. Proposition~\ref{prop:k2_pos_k4_0} is proven.

\subsubsection{The case $ \kappa_{ 2}=0$ and $ \kappa_{ 4}>0$}
\label{sec:proof_k2_0_k4_pos}
We prove Proposition~\ref{prop:k2_0_k4_pos}. The dynamics of population $A$ follows a standard Hawkes dynamics, isolated from population $B$. Each estimate concerning $ \lambda_{  t}^{A}$ as $t\to\infty$ hence follows from  e.g. \cite{MR3449317}. If $ \kappa_{ 1}<1$,it holds $ \lambda_{  t}^{A} \xrightarrow[ t\to\infty]{} \frac{ \mu_{ A}}{ 1- \kappa_{ 1}}$.  Suppose $ \kappa_{ 3}<1$, the convergence result of $  \lambda_{ t}^{B}$ as $t\to\infty$ is a direct application of \eqref{eq:bound_ellB_1} and \eqref{eq:ulB_ulA} using that $\|\Phi_{B}\|_{L}=1$. The case $ \kappa_{ 3}>1$ follows also from \eqref{eq:ulB_ulA}: we have $ \underline{ \ell}_{ B} \geq \mu_{ B} + \kappa_{ 3} \underline{ \ell}_{ B}$ which gives that $ \underline{ \ell}_{ B}=+\infty$ since $ \kappa_{ 3}>1$ and $ \mu_{ B}>0$. In the case $ \kappa_{ 1}>1$, $ \lambda_{  t}^{A} \xrightarrow[ t\to\infty]{} +\infty$ and one obtains from \eqref{eq:ulB_ulA_gen} that $ \underline{ \ell}_{ B}= +\infty$, since $ \kappa_{ 4}>0$ and $ \lim_{ +\infty} \Phi_{ A\to B}=+\infty$. Proposition~\ref{prop:k2_0_k4_pos} is proven.

\subsection{Proofs in the fully-coupled case $ \kappa_{ 2} \kappa_{ 4}>0$}
\label{sec:proof_fully_coupled}
We prove here the results of Section~\ref{sec:fully_coupled}.
\subsubsection{Proofs of auxiliary results}

\paragraph{Proof of the general subcriticality result (Theorem~\ref{th:A_subcritical}):}
\label{sec:th_A_subcritical}
Proceed by contradiction and suppose that $ \underline{ \ell}_{ A}= +\infty$. Since  $ \kappa_{ 4}>0$ and $ \Phi_{ A\to B}(x) \xrightarrow[ x\to\infty]{}+\infty$ by Assumption~\ref{ass:Phis}, one deduces from \eqref{eq:ulB_ulA_gen} that $ \underline{ \ell}_{ B}= +\infty$. As $ \kappa_{ 2}>0$ and $ \Phi_{ B\to A}(x) \xrightarrow[ x\to\infty]{}0$, then inequality \eqref{hyp:A_sub} is verified. Hence, the conclusions of Proposition~\ref{prop:bounds_ells} hold and $ 0 \leq \underline{ \ell}_{ A} \leq \bar \ell_{ A}<\infty$ which leads to a contradiction. In the case where $  \lambda_{ t}^{B} \xrightarrow[ t\to\infty]{}+\infty$, \eqref{hyp:A_sub} holds and point 2.  is an immediate consequence of \eqref{eq:apriori_ellA} and the fact that $ \Phi_{ B\to A}(x) \xrightarrow[ x\to\infty]{}0$.
\paragraph{The case population $B$ is supercritical (Proposition~\ref{prop:B_sup}):}
\label{sec:th_B_sup}
We use here \eqref{eq:ulB_ulA}: since $ \Phi_{ A\to B}\geq0$, we obtain from this inequality that both $ \underline{ \ell}_{ B}\geq \mu_{ B}$ and $ \underline{\ell_{ B}} \geq \kappa_{ 3} \underline{ \ell}_{ B}$. Since $ \mu_{ B}>0$, $\underline{ \ell}_{ B}>0$ and since $ \kappa_{ 3}>1$, the only possibility is that $ \underline{ \ell}_{ B}= +\infty$, which gives the first part of the result. The second part is a direct consequence of Theorem~\ref{th:A_subcritical}.

\subsubsection{Proof of the main convergence result (Theorem~\ref{th:full_conv})}
\label{sec:th_full_conv}
Recall that we concentrate now on Model \eqref{eq:lambdaAB_lin} (Example~\ref{ex:linear}), satisfying Assumption~\ref{ass:Phis},  $\kappa_{2}>0$ and $\kappa_{4}>0$, in case population $B$ is subcritical (i.e. $\kappa_{3}<1$). We divide the proof into several steps.
\paragraph{Subcriticality of population $B$ and a priori estimates (item 1)}
As $t \mapsto  \lambda_{ t}^{B}$ is continuous it is bounded on any compact interval, it suffices to prove that prove that $\bar \ell_{ B}<\infty$. Proceeding in a way similar to the proof of Proposition~\ref{prop:bounds_ells}, we define for $t\leq t^{ \prime}$, $u_{ B}(t, t^{ \prime})= \sup_{ s\in[t, t^{ \prime}]}  \lambda_{ s}^{B}$. Then, for $t\leq s\leq t^{ \prime}$, it holds that
\begin{align*}
 \lambda_{ s}^{B} &\leq\mu_{ B} +  \kappa_{ 3}\left\Vert \lambda^{ B} \right\Vert_{ \infty, t} + \kappa_{ 3} u_{ B}(t, t^{ \prime}) +    \Phi_{ A\to B} \left(\kappa_{ 4}\left\Vert \lambda^{ A} \right\Vert_{ \infty, t}  +   \kappa_{ 4}u_{ A}(t)\right).
\end{align*}
As before, using $\kappa_{3}<1$, making $t^{\prime}\to \infty$ and using \eqref{eq:uniform_control_uA}, we obtain
\begin{equation}
\label{eq:uniform_control_uB}
u_{ B}(t)<\infty, \ t\geq0.
\end{equation}
We deduce immediately from \eqref{eq:uniform_control_uB} and \eqref{aux:ellB_VS_uB} that $\bar \ell_{ B}<\infty$. Let us now prove the a priori bound $\underline{\ell}_{ B} \geq \frac{ \mu_{ B}}{ 1- \kappa_{ 3}}$. If $ \underline{\ell}_{ B}=+\infty$, it is trivially verified and if $ \underline{\ell}_{ B}<+\infty$, one obtains from \eqref{eq:ulB_ulA} that $ \underline{ \ell}_{ B}\geq \mu_{ B} + \underline{ \ell}_{ B} \kappa_{ 3}$ which gives the result.

Let us prove now that $ \bar\ell_{ B}\in \mathcal{ I}_{ \kappa_{ 1}, \kappa_{ 2}}$ (recall the definition of $ \mathcal{ I}_{ \kappa_{ 1}, \kappa_{ 2}}$ in \eqref{eq:I_kappas}). Indeed, recall that, under the present hypotheses, $ \underline{ \ell}_{ A}<\infty$ (Theorem~\ref{th:A_subcritical}). Recall now the a priori bound \eqref{eq:ulA_blB}:  $\underline\ell_{ A}\geq \left(\mu_{ A} + \kappa_{ 1} \underline \ell_{ A}\right) \Phi_{ B\to A} \left(\bar \ell_{ B} \kappa_{ 2}\right)$ so that $ \underline{ \ell}_{ A} \left( 1- \kappa_{ 1} \Phi_{ B\to A} \left( \bar \ell_{ B} \kappa_{ 2}\right)\right)\geq \mu_{ A} \kappa_{ 1}\Phi_{ B\to A} \left( \bar{ \ell}_{ B} \kappa_{ 2}\right)\geq0$. There are two possibilities: either $ \underline{ \ell}_{ A}>0$ and the previous inequality gives $\kappa_{ 1} \Phi_{ B\to A} \left(\bar{ \ell}_{ B} \kappa_{ 2} \right)<1$, since $ \mu_{ A}>0$; either $ \underline{ \ell}_{ A}=0$ in such a case $ \kappa_{ 1}\Phi_{ B\to A} \left(\bar \ell_{ B} \kappa_{ 2}\right)=0$,  since $ \mu_{ A}>0$ and $\bar\ell_{ B}\in \mathcal{ I}_{ \kappa_{ 1}, \kappa_{ 2}}$ holds trivially. Hence \eqref{eq:apriori_ellsB} is true.

\paragraph{Convergence of $ \lambda_{ A}$ implies convergence of $ \lambda_{ B}$ (Item 2)}
Since $ \Phi_{ B}(0)= \mu_{ B}$ and $ \left\Vert \Phi_{ B} \right\Vert_{ L}=1$, \eqref{eq:bound_ellB_1} and \eqref{eq:ulB_ulA} read here
\begin{align}
\bar \ell_{ B} &\leq  \Psi_{ 2} \left(\bar \ell_{ A}\right)\label{eq:bound_ellB_1_sub}\quad\mbox{and}\quad
\underline \ell_{ B} \geq \Psi_{ 2} \left(\underline\ell_{ A}\right).
\end{align}
Item 2. of Theorem~\ref{th:full_conv} follows immediately.

\paragraph{Item 3. of Theorem~\ref{th:full_conv}}
Suppose here \eqref{eq:uB_in_I} (which is equivalent to \eqref{hyp:A_sub} in the context of system \eqref{eq:lambdaAB_lin}). Let us first prove \eqref{eq:ellsBinU}, that is $(\underline{ \ell}_{ B}, \bar\ell_{ B}) \in \mathcal{ U}_{ \Phi}$ (recall the definition of $ \Phi$ in \eqref{eq:Phi} and of $ \mathcal{ U}_{ \Phi}$ in \eqref{eq:def_UPhi}). Indeed, the conclusions of Proposition~\ref{prop:bounds_ells} hold and applying the nondrecreasing function $ \Psi_{ 2}$ (recall its definition in \eqref{eq:Psi2}) to \eqref{eq:bound_ellA_1} and \eqref{eq:bound_ellA_2} and combining with \eqref{eq:bound_ellB_1_sub} 
 gives that 
\begin{equation}
\label{eq:ell_B_UPhi}
\Phi(\bar \ell_{ B})\leq \underline \ell_{ B} \text{ and } \bar \ell_{ B} \leq \Phi(\underline \ell_{ B}).
\end{equation}
It remains to prove that $ \underline{\ell}_{ B} \leq \ell \leq \bar \ell_{ B}$. First, suppose that $\underline{\ell}_{ B} \geq \ell$: since $ \Phi$ is nondecreasing, we obtain $ \Phi(\underline{\ell}_{ B}) \leq \Phi(\ell)=\ell$ so that $\ell \leq \underline{\ell}_{ B} \leq \bar \ell_{ B} \leq \Phi \left( \underline{ \ell}_{ B}\right)\leq \ell$ which gives that $ \underline{ \ell}_{ B}= \bar \ell_{ B}= \ell$. In a same way, if $ \bar \ell_{ B}\leq \ell$, then $ \Phi \left( \bar \ell_{ B}\right)\geq \ell$ so that \eqref{eq:ell_B_UPhi} gives again $\ell \leq \Phi( \bar \ell_{ B})\leq \underline{ \ell}_{ B} \leq \bar \ell_{ B}\leq \ell$ so that $ \underline{ \ell}_{ B}= \bar \ell_{ B}= \ell$. Hence, in any case, $ \underline{\ell}_{ B} \leq \ell \leq \bar \ell_{ B}$ is always true. This means exactly \eqref{eq:ellsBinU}. If now we suppose Assumption~\ref{ass:U}, we obtain that $\underline{ \ell}_{ B}= \bar \ell_{ B}=\ell$ and hence the convergence of $  \lambda_{ t}^{B}$ as $t\to \infty$ to the unique fixed-point $\ell$ of $ \Phi$. The fact that $0\leq \underline{\ell}_{ A}\leq \bar \ell_{ A}<\infty$ follows from Proposition~\ref{prop:bounds_ells} and the convergence of $ \lambda_{  t}^{A} \xrightarrow[ t\to\infty]{} \ell_{ A}:= \Psi_{ 1}(\ell_{ B})$ from \eqref{eq:bound_ellA_1} and \eqref{eq:bound_ellA_2}. This gives the implication $(a)\Rightarrow (b)$ of item 3. Conversely, if we suppose now that $  \lambda_{ t}^{B}$ has a finite limit as $t\to\infty$, then $ \underline{ \ell}_{ B}= \bar \ell_{ B}$ and, by \eqref{eq:apriori_ellsB}, \eqref{eq:uB_in_I} is true. This completes the proof of Theorem~\ref{th:full_conv}. 

\section*{Acknowledgements}
E.L. acknowledges the support of grants ANR-19-CE40-0024 (CHAllenges in MAthematical NEuroscience) and ANR-19-CE40-0023 (PERISTOCH). This research has been conducted within the FP2M federation (CNRS FR 2036).

\end{document}